\documentclass{article}
\usepackage{amsthm}
\usepackage{amsmath}
\usepackage{amssymb}
\usepackage{hyperref}
\usepackage{xcolor} 
\usepackage{geometry} 
\usepackage{graphicx}
\usepackage{tikz}
\usepackage[english]{babel}
\usepackage{subcaption}
\usepackage{adjustbox}
\usepackage{float}
\usepackage[most]{tcolorbox}
\usepackage{bbm}
\usepackage{caption}
\usepackage{makecell}
\usepackage{xcolor}

\usepackage[normalem]{ulem}
\usepackage{todonotes}

\newtheorem{proposition}{Proposition}
\newtheorem{theorem}{Theorem}
\newtheorem{lemma}{Lemma}
\newtheorem{example}{Example}
\newtheorem{corollary}{Corollary}
\newtheorem{proof*}{Proof}
\newtheorem{definition}{Definition}
\newtheorem{conjecture}{Conjecture}
\title{Fully Leafed Induced Subtrees in Penrose P2 Tilings}
\author{Mathieu Cloutier\footnote{Université du Québec à Trois-Rivières, Trois-Rivières, Québec} \and Alain Goupil\footnote{Université du Québec à Trois-Rivières, Trois-Rivières, Québec} \and Alexandre Blondin Massé\footnote{Université du Québec à Montréal, Montréal, Québec}}
\date{}

\begin{document}

\maketitle

\begin{center}
    \textbf{Abstract}\\
\medskip
    \begin{minipage}{0.9\textwidth}
\noindent
    In a recent article by C. Porrier, A. Blondin Massé and A. Goupil, a first bi-infinite fully leafed induced subcaterpillar of Penrose P2 tilings is presented. In this paper, we formally construct this caterpillar for the first time. We then prove that every fully leafed induced subtree in Penrose P2 tilings is a caterpillar with at most one appendix of at most two internal tiles, and we characterize fully leafed induced subtrees that have the property of saturation. We also refute the conjecture that there is a unique bi-infinite fully leafed induced subcaterpillar by constructing a new one. Finally, we present progress on the construction of all bi-infinite fully leafed induced subcaterpillars in Penrose P2 tilings.\\
    \end{minipage}
\end{center}

\begin{flushleft}
\textbf{Keywords:} Aperiodic tilings, Penrose P2, HBS, kite, dart, graph theory, fully leafed induced subtrees, bi-infinite caterpillars
\end{flushleft}

\section{Introduction}

The mathematical study of aperiodic tilings is relatively recent \cite{grunbaum1987tilings}. The first such tiling was discovered by Robert Berger in 1966, who required 20,426 distinct types of Wang tiles for his construction. Mathematicians naturally sought to reduce this number, and in 1973, Roger Penrose introduced a method to tile the plane aperiodically using only six different types of tiles. Penrose later succeeded in reducing the number to two, with the Penrose P2 tilings (or the kite and dart tilings) (see Figure \ref{fig:pavage}). Today, Penrose P2 tilings are used to model the atomic structure of quasicrystals\footnote{Further details can be found in \cite{madison2013symmetry}.}.

Given a tiling, one may seek the induced subtrees (i.e. connected and acyclic subgraphs) of $n$ tiles that maximize the number of tiles of degree 1 (leaves). These subtrees are called fully leafed induced subtrees. They are used in chemistry and physics to model the adsorption process, that is, how particles accumulate on a surface, in situations where the exposed boundary of the deposit is maximized. The structure of fully leafed induced subtrees has been studied for classical tilings, such as the square and the cubic tilings \cite{blondin2018fully}. In this paper, we investigate the same question in the context of Penrose P2 tilings. The results we provide here can be used to study adsorption on quasicrystal surfaces \cite{mcgrath2010surface}.

Section 2 introduces the formal definitions required for this work. In Section 3, we revisit \cite{porrier2023leaf} and provide a rigorous proof of the existence of a bi-infinite fully leafed induced subcaterpillar in the Penrose P2 tilings. Section 4 focuses on the structure of fully leafed induced subtrees in P2 tilings. As conjectured in \cite{porrier2023leaf}, we prove that every such subtree is a caterpillar with at most one appendix of at most two internal tiles. We also examine in this section the notion of saturated fully leafed induced subtrees, originally introduced in \cite{masse2018saturated} for other types of tilings and for graphs. Section 5 presents a collection of results that may help in identifying bi-infinite fully leafed induced subcaterpillars in Penrose P2 tilings. In Section 6, we disprove the conjecture that the bi-infinite caterpillar described in \cite{porrier2023leaf} is the unique bi-infinite fully leafed induced subcaterpillar of Penrose P2 tilings. Finally, Section 7 summarizes the main results of this paper and outlines some related open questions.

\section{Definitions}
We present here the definitions and terminology that will be used throughout this paper.
\subsection{Graph theory}
We first recall some definitions from graph theory, which can be found in \cite{diestel2025graph}.
A connected graph without cycles is called a \textbf{tree}. A \textbf{leaf} of a tree $T$ is a vertex of $T$ of degree one, and a vertex of degree 2 or more is an \textbf{internal vertex}. For a graph $G=(V,E)$ with vertex set $V$ and edge set $E$, and a set of vertices $U \subset V$, let $\mathcal{P}_2(U)$ denote the set of two-element subsets of $U$. The graph $G[U] = (U, E \cap \mathcal{P}_2(U))$ is called the \textbf{subgraph of $G$ induced by $U$}. An \textbf{induced subtree} is an induced subgraph that is a tree. The next two definitions were first presented in \cite{blondin2018fully}.

\begin{definition}[\cite{blondin2018fully}]
    A \textbf{fully leafed induced subtree} of order $n$ in a graph $G$ is an induced subtree of $G$ that maximizes the number of leaves among all induced subtrees of order $n$ in $G$.
\end{definition}

\begin{definition}[\cite{blondin2018fully}]
    The \textbf{leaf function} of a graph $G$, denoted $L_G$, is the function that assigns to each positive integer $n$ the number of leaves of a fully leafed induced subtree of order $n$ in $G$.
\end{definition}

Let $G=(V,E)$ be a graph with vertex set $V$ and edge set $E$. Define $V' = V \setminus \{v \in V : v \text{ has degree } 1\}$ and $E' = E \setminus \{\{x,y\} \in E : x \text{ or } y \text{ has degree } 1\}$. The graph $G' = (V',E')$ is called the \textbf{derived graph of $G$}. If a tree $C$ is such that all vertices of $C'$ have degree at most two (i.e., $C'$ is a path), then $C$ is called a \textbf{caterpillar}.

\subsection{Penrose P2 tilings}
The following definitions on tilings are inspired by those found in \cite{grunbaum1987tilings}. A \textbf{tiling of the plane} is a covering of $\mathbb{R}^2$ by a countable family $T$ of closed subsets of $\mathbb{R}^2$ such that for all pairs $(\mathcal{T}_i,\mathcal{T}_j)\in T^2$ with $\mathcal{T}_i\neq \mathcal{T}_j$, we have $\overset{\circ}{\mathcal{T}_i}\cap\overset{\circ}{\mathcal{T}_j}=\varnothing$. The subsets $\mathcal{T}_i$ are called \textbf{tiles}. We will use the term \textbf{patch} to refer to a connected set of tiles. Two patches $P_1$ and $P_2$ are \textbf{adjacent} if $P_1\cup P_2$ is a patch. For $P$ a patch, the \textbf{kingdom} of $P$ is the largest patch containing $P$ that is contained in every tiling containing $P$.

The tilings studied in this paper are the Penrose P2 tilings, also known as \textit{kites and darts tilings}, or more simply \textit{P2 tilings}. Penrose P2 tilings are constructed from two prototiles (i.e., tile types), illustrated in Figure \ref{fig:prototuiles}. P2 tilings are aperiodic and there exists an uncountable infinity of P2 tilings. P2 tilings satisfy the local isomorphism property, meaning that every local pattern occurring in one P2 tiling occurs in all P2 tilings. For brevity, the tiles used to construct a P2 tiling will often be referred to as P2 tiles.

\begin{figure}[H]
    \centering
    \begin{subfigure}{0.15\textwidth}
        \centering
        \includegraphics[width=\textwidth]{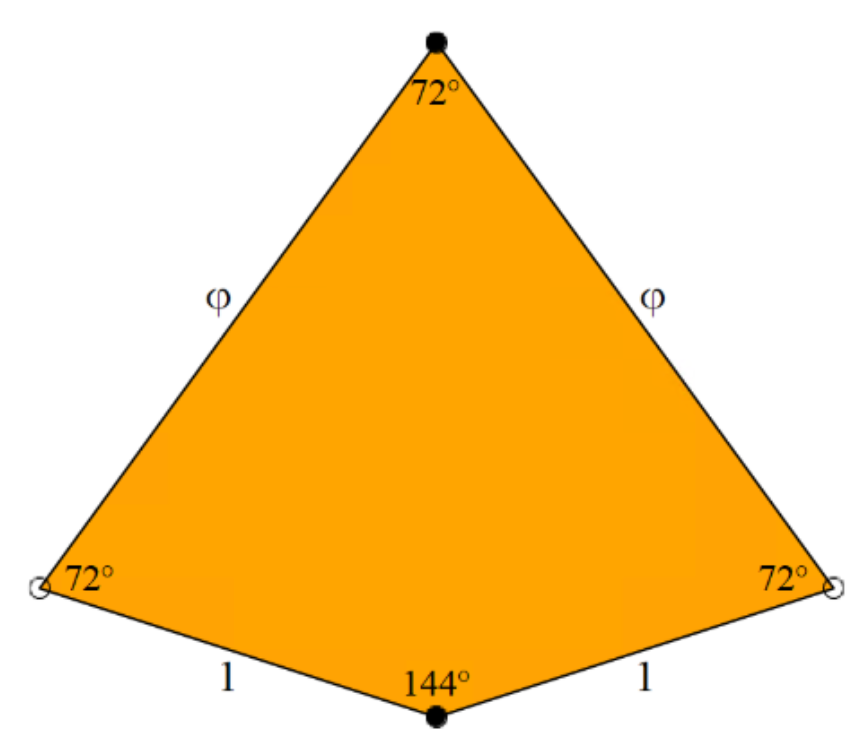}
        \caption{Kite}
    \end{subfigure} 
    \hspace{2cm}
    \begin{subfigure}{0.15\textwidth}
        \centering
        \includegraphics[width=\textwidth]{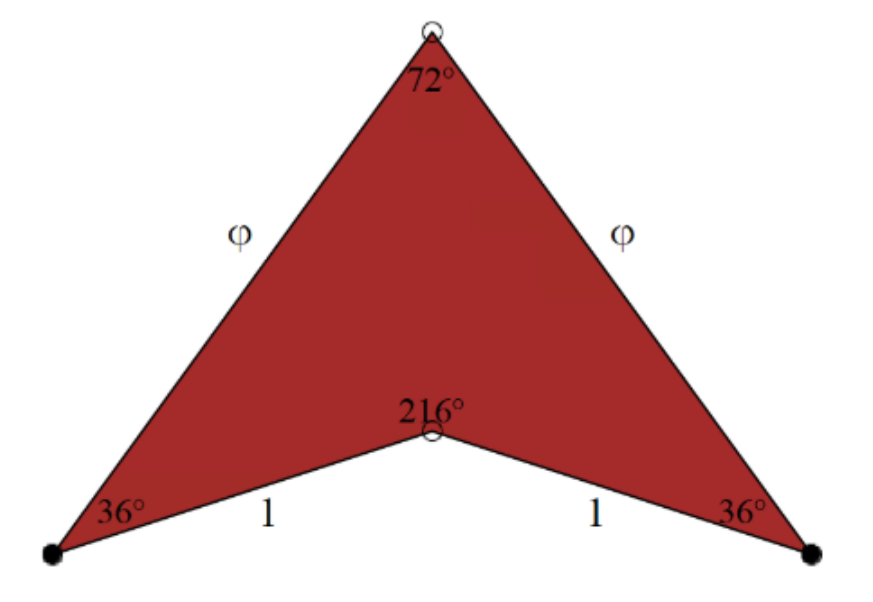}
        \caption{Dart}
    \end{subfigure}
    \caption{Prototiles of P2 tilings. Tiles may meet only when filled (resp. empty) vertices coincide. $\varphi = \frac{1+\sqrt{5}}{2}$ is the golden ratio, giving the ratio of long to short edges in both prototiles.}
    \label{fig:prototuiles}
\end{figure}

From the construction rules of P2 tile patches, one can enumerate the patches forming a minimal neighborhood of a vertex. These patches are shown in Figure \ref{fig:collages de Conway}.

\begin{figure}[H]
    \centering
    \begin{subfigure}{0.1\textwidth}
        \centering
        \includegraphics[width=\textwidth]{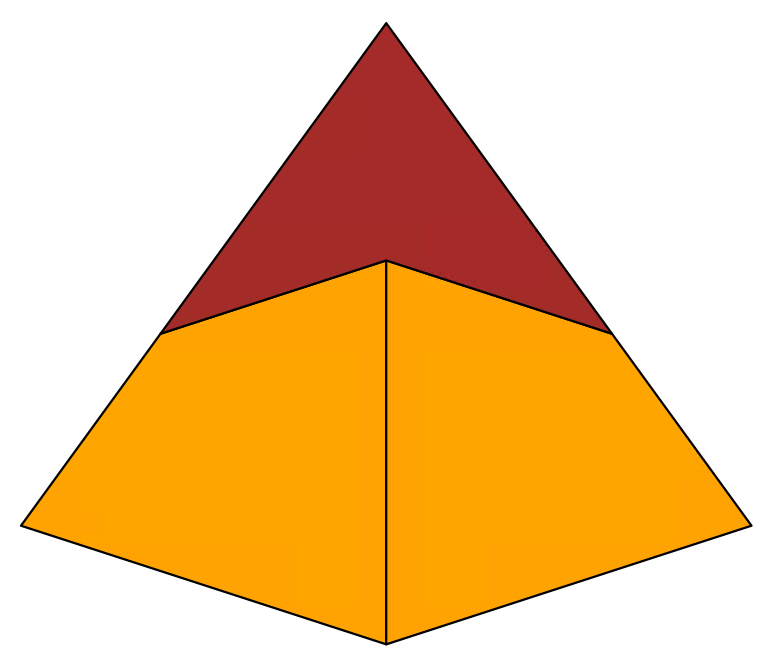}
        \caption{Ace}
    \end{subfigure} 
    \hspace{.4 cm}
    \begin{subfigure}{0.1\textwidth}
        \centering
        \includegraphics[width=\textwidth]{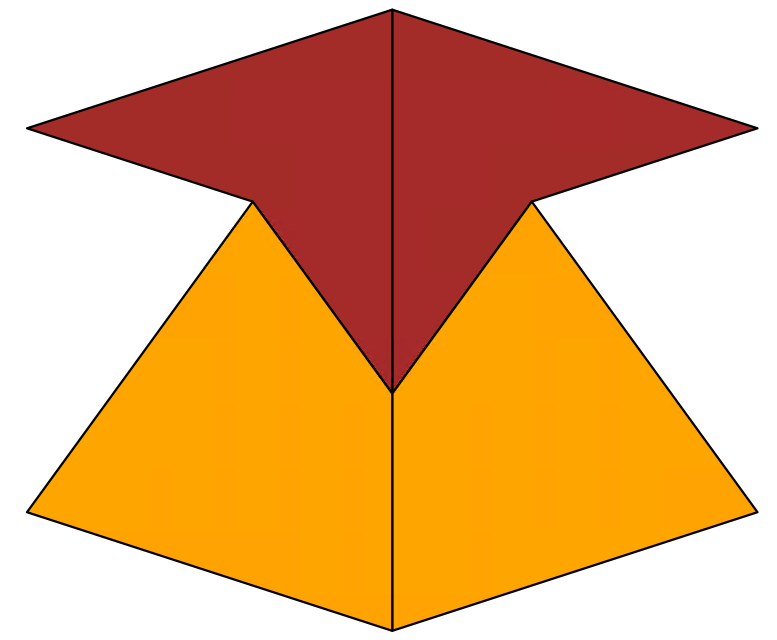}
        \caption{Deuce}
    \end{subfigure}
    \hspace{.4 cm}
    \begin{subfigure}{0.1\textwidth}
        \centering
        \includegraphics[width=\textwidth]{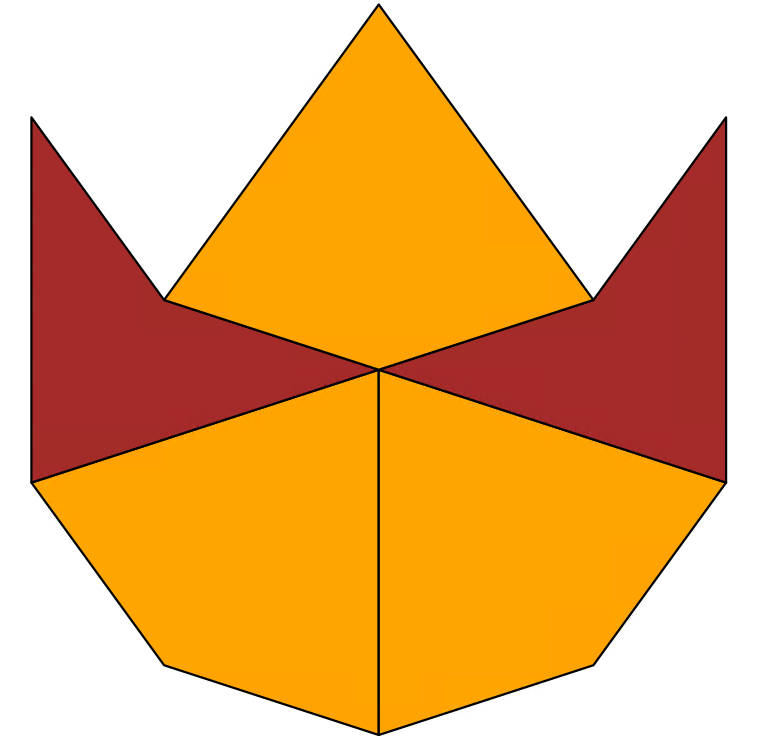}
        \caption{Jack}
    \end{subfigure}
    \hspace{.4 cm}
    \begin{subfigure}{0.1\textwidth}
        \centering
        \includegraphics[width=\textwidth]{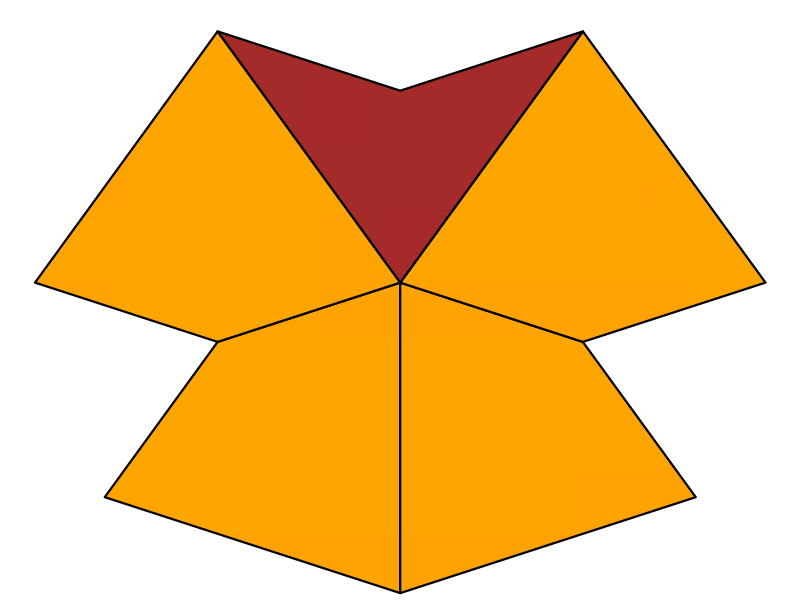}
        \caption{Queen}
    \end{subfigure}
    \hspace{.4 cm}
    \begin{subfigure}{0.1\textwidth}
        \centering
        \includegraphics[width=\textwidth]{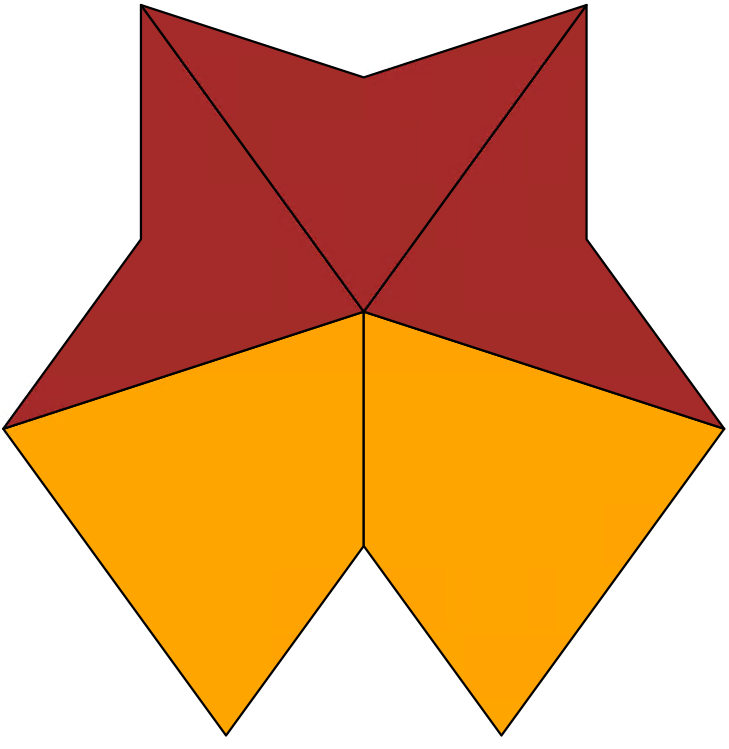}
        \caption{King}
        \label{king}
    \end{subfigure}
    \hspace{.4 cm}
    \begin{subfigure}{0.1\textwidth}
        \centering
        \includegraphics[width=\textwidth]{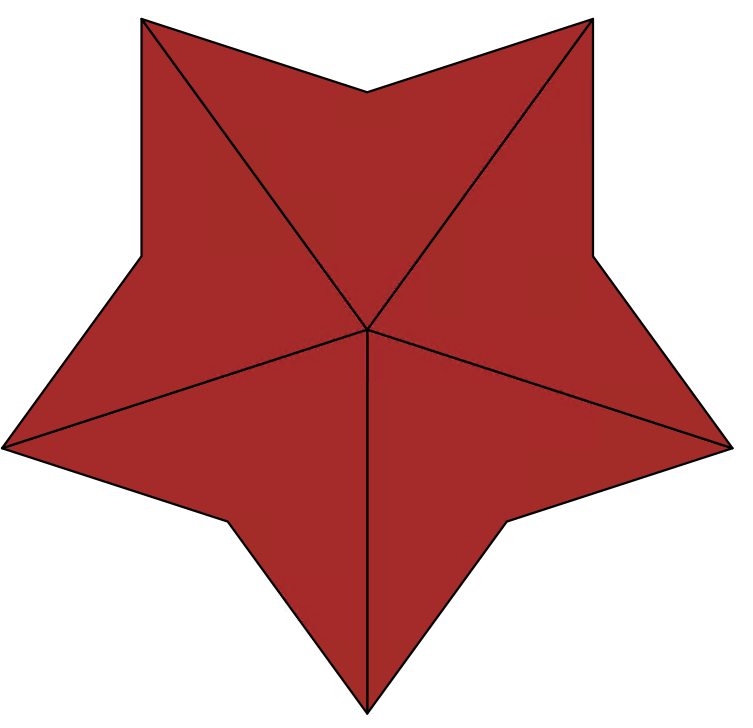}
        \caption{Star}
        \label{star}
    \end{subfigure}
    \hspace{.4 cm}
    \begin{subfigure}{0.1\textwidth}
        \centering
        \includegraphics[width=\textwidth]{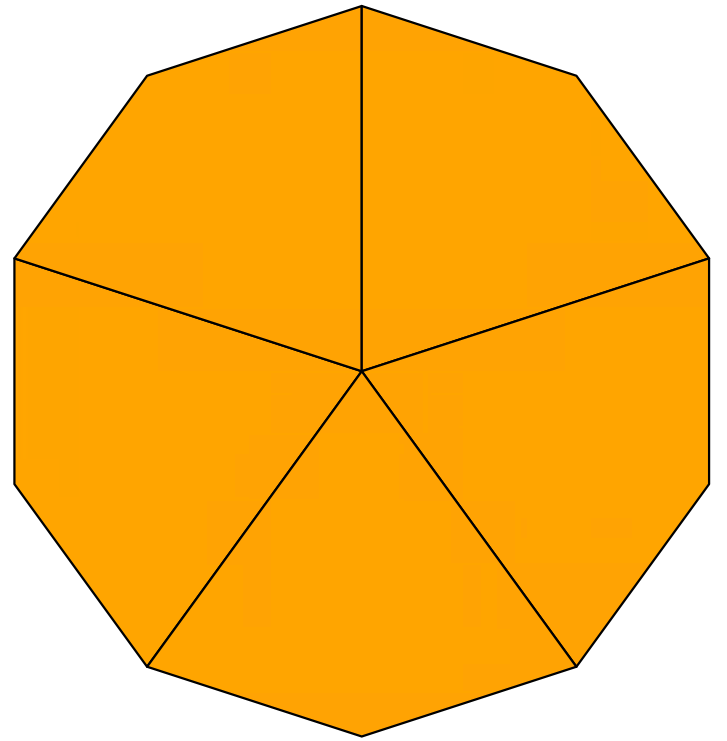}
        \caption{Sun}
        \label{sun}
    \end{subfigure}
    \caption{The seven minimal vertex neighborhoods in P2 tilings up to isometry}
    \label{fig:collages de Conway}
\end{figure}

Starting from  a kite (resp. a dart), outlined by a blue dashed contour in Figure \ref{fig:inflation cerf-volant} (resp. \ref{fig:inflation fléchette}), the \textbf{inflation} of this kite (resp. dart) is the patch of four (resp. three) smaller P2 tiles shown in brown and orange in Figure \ref{fig:inflation cerf-volant} (resp. \ref{fig:inflation fléchette}). The tiles in the inflation are $\varphi$ times smaller than the original tile.

\begin{figure}[H]
    \centering
    \begin{subfigure}{0.25\textwidth}
        \centering
        \includegraphics[width=0.6\textwidth]{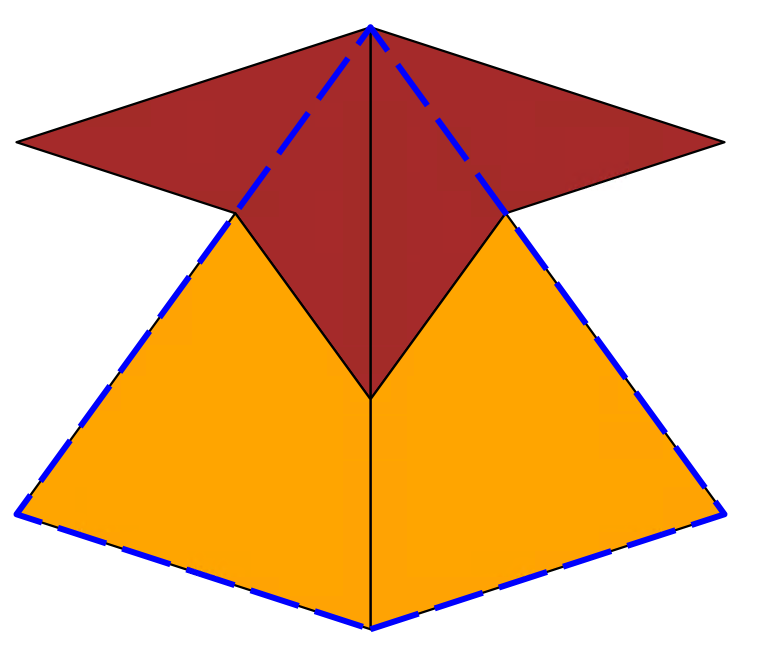}
        \caption{Inflation of a kite}
        \label{fig:inflation cerf-volant}
    \end{subfigure}
    \begin{subfigure}{0.25\textwidth}
        \centering
        \includegraphics[width=0.6\textwidth]{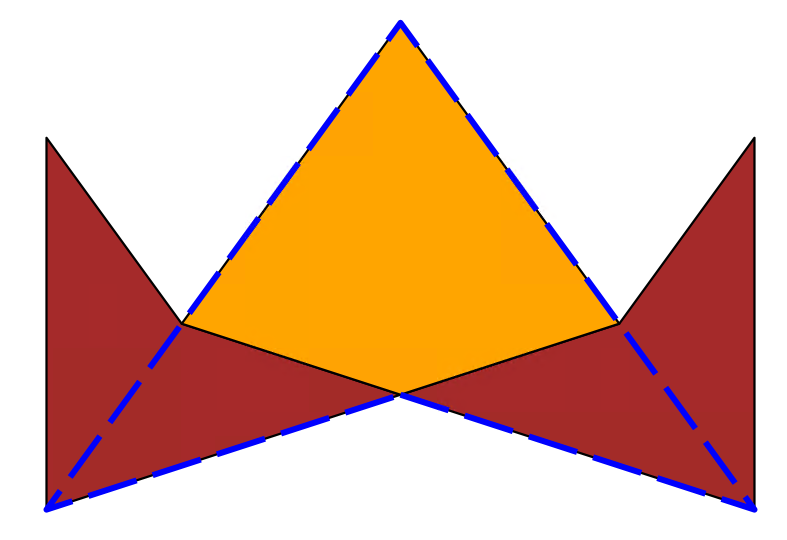}
        \caption{Inflation of a dart}
        \label{fig:inflation fléchette}
    \end{subfigure}
    \caption{Inflation of P2 tiles}
\end{figure}

Then, the \textbf{inflation} of a patch $P$ of P2 tiles is the patch obtained by inflating every tile of $P$. By iterating inflation infinitely many times and by enlarging the tiles at each iteration, we can then obtain P2 tilings.

The vertices and edges of a P2 tiling form an infinite graph whose vertices and edges correspond to those of the tiles themselves. In what follows, we focus on the dual of this graph, which we call a \textbf{P2-graph}. In a P2-graph, the terms \textit{vertex} and \textit{tile} are used interchangeably. We shall study the fully leafed induced subtrees of P2-graphs.

\begin{figure}[H]
    \centering
    \includegraphics[width=0.5\linewidth]{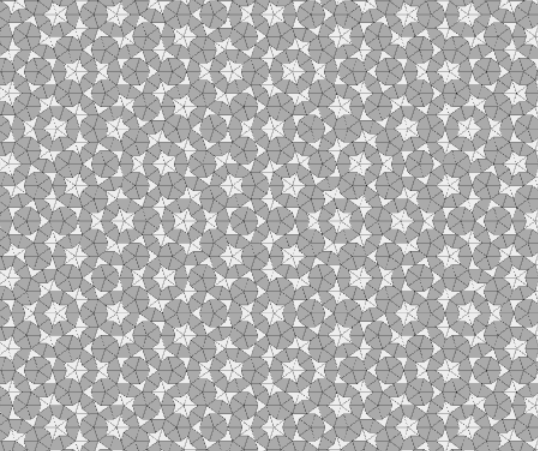}
    \caption{A region of a Penrose P2 tiling. The kites are shown in grey and the darts in white.}
    \label{fig:pavage}
\end{figure}

From this point on, since the object of study is clear, for brevity, we may use the terms \textit{tree} and \textit{caterpillar} without specifying each time that it is an induced subgraph of a P2-graph.

\section{Formal construction of a bi-infinite fully leafed caterpillar}

The main goal of this section is to formalize the construction of a bi-infinite fully leafed caterpillar of P2 tilings, that was first presented in \cite{porrier2023leaf}. First, we need the following results from \cite{porrier2023leaf} and we refer the reader to that article for the proofs.

\begin{lemma}[\cite{porrier2019leaf}]\label{max 3}
    The degree of each tile of an induced subtree of a P2-graph is at most $3$.
\end{lemma}

\begin{lemma}[\cite{porrier2023leaf}] \label{lem2}
    Let $I$ be an induced subtree of a P2-graph. If $n_1$ and $n_3$ denote respectively the number of leaves and the number of tiles of degree $3$ in $I$, then $n_1 = n_3 + 2$. 
\end{lemma}

\begin{corollary} \label{cor 1}
    $I$ is a fully leafed induced subtree of a P2-graph $G_{P2}$ if and only if $I$ is an induced subtree of $G_{P2}$ that maximizes the number of vertices of degree $3$ among all induced subtrees of $G_{P2}$ of the same order as $I$.
\end{corollary}

\begin{proof}
    This follows directly from Lemma \ref{lem2}.
\end{proof}

\begin{lemma}[\cite{porrier2023leaf}] \label{lemme 3}
    If $I$ is an induced subtree of a P2-graph in which all internal vertices have degree $3$ in $I$, then $I$ is a caterpillar.
\end{lemma}

\begin{corollary} \label{cor 2}
    If $I$ is an induced subtree of a P2-graph in which all internal tiles have degree $3$, then $I$ is a fully leafed caterpillar.
\end{corollary}

\begin{proof}
    This follows directly from Corollary \ref{cor 1} and Lemma \ref{lemme 3}.
\end{proof}

More precisely, we have from \cite{porrier2023leaf} the following result.

\begin{lemma}[\cite{porrier2023leaf}] \label{lemme 4}
    If $I$ is a fully leafed induced subtree of a P2-graph with $8$ internal tiles or less, then $I$ is one of the caterpillars in Figure~\ref{fig:chenilles premières} or a proper subcaterpillar of one of them.
\end{lemma}

Let us call any of the caterpillars in Figure \ref{fig:chenilles premières} a \textbf{prime caterpillar} and label them $PC_1$ to $PC_6$ according to the order in which they appear in Figure \ref{fig:chenilles premières}. Figure \ref{fig:chenilles premières} and all the other figures of this paper use the same color code. Brown tiles are internal tiles of degree $3$, dark green tiles are leaves, and light-green tiles (with dotted outlines) indicate potential leaves, of which only one may be chosen when two of them are overlapping or adjacent. White tiles belong to the tiling but they are not part of the subtree nor of any of its extensions as a subtree.

\begin{figure}[H]
    \centering
    \begin{subfigure}{0.16\textwidth}
        \centering
        \includegraphics[width=\textwidth]{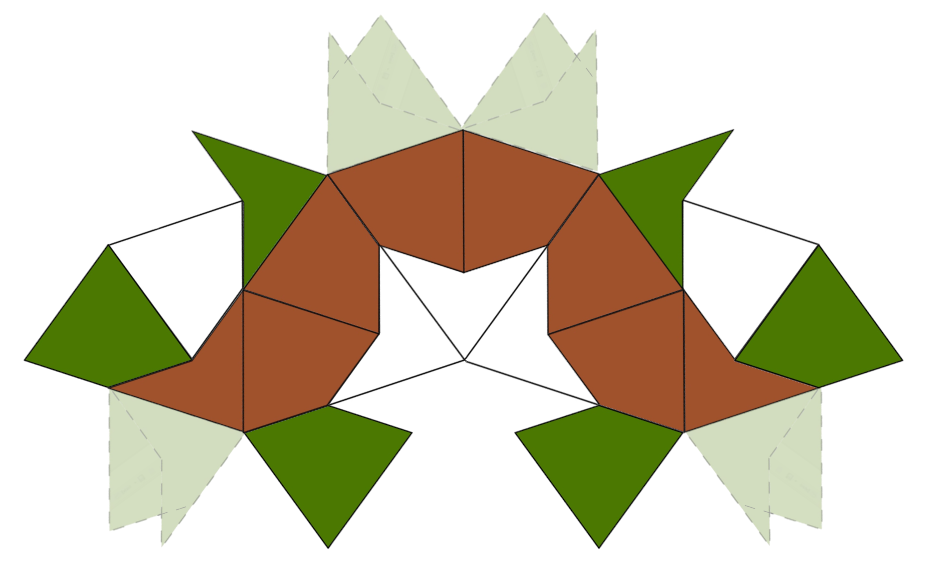}
        \caption{$PC_1$}
        \label{fig:CP1}
    \end{subfigure} \hfill
    \begin{subfigure}{0.15\textwidth}
    \includegraphics[width=\textwidth, trim=0cm 0.4cm 0cm 0cm, clip]{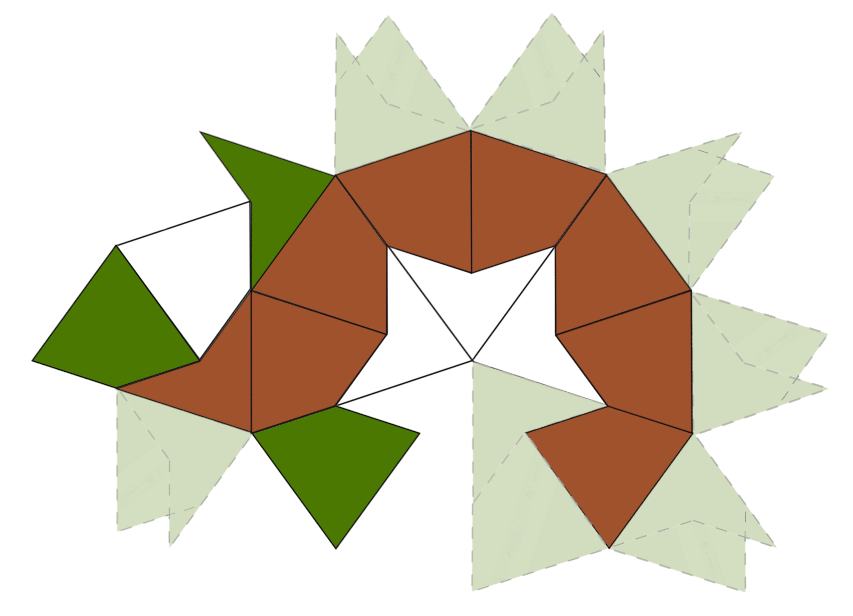}
    \caption{$PC_2$}
    \label{fig:CP2}
    \end{subfigure} \hfill
    \begin{subfigure}{0.16\textwidth}
    \includegraphics[width=\textwidth, trim=0cm 0cm 0cm 0cm, clip]{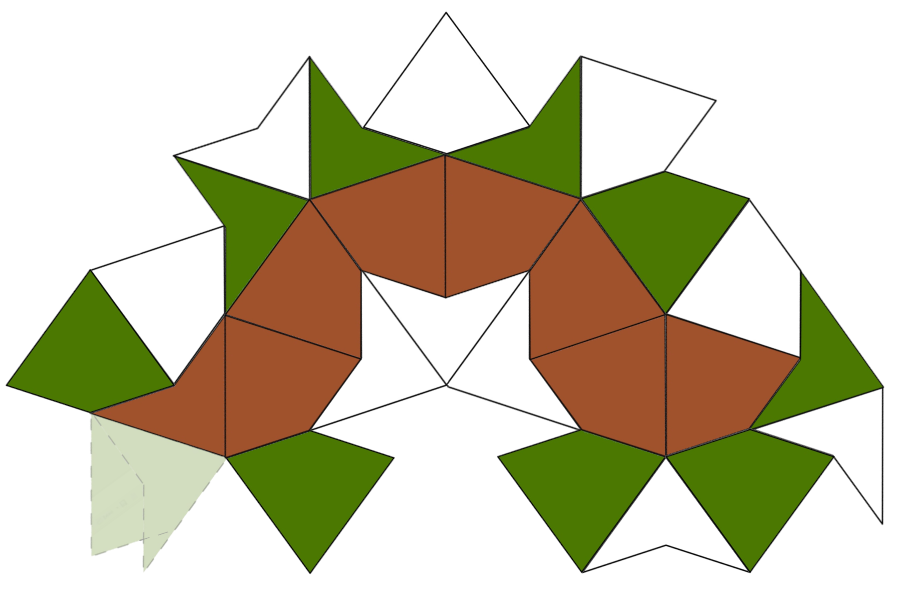}
    \caption{$PC_3$}
    \label{fig:CP3}
    \end{subfigure} \hfill
    \begin{subfigure}{0.15\textwidth}
    \includegraphics[width=\textwidth, trim=0cm 0.5cm 0cm 0cm, clip]{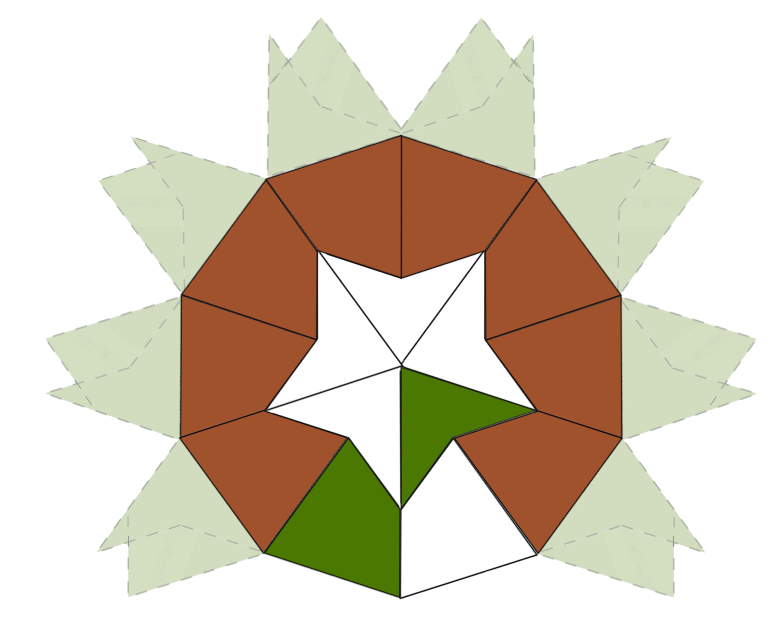}
    \caption{$PC_4$}
    \label{fig:CP4}
    \end{subfigure} \hfill
    \begin{subfigure}{0.16\textwidth}
    \includegraphics[width=\textwidth]{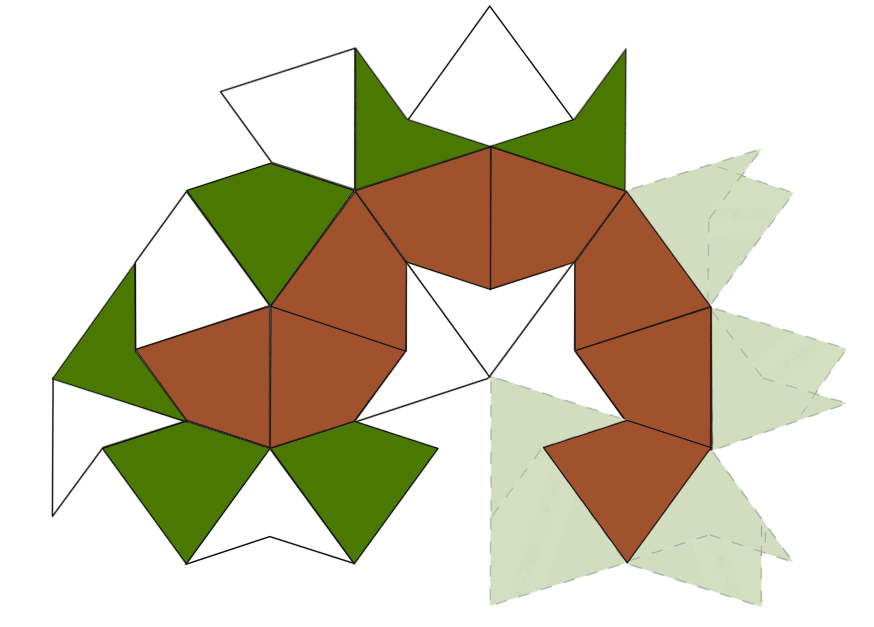}
    \caption{$PC_5$}
    \label{fig:CP5}
    \end{subfigure} \hfill
    \begin{subfigure}{0.16\textwidth}
    \includegraphics[width=\textwidth, trim=0.5cm 0cm 0cm 0cm, clip]{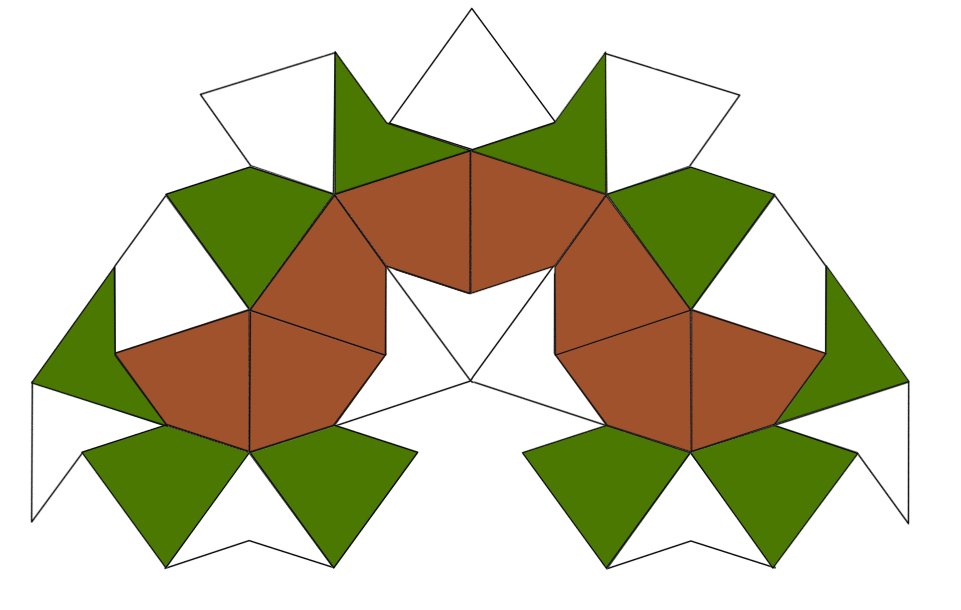}
    \caption{$PC_6$}
    \label{fig:CP6}
    \end{subfigure} \hfill
    \caption{The six possible prime caterpillars, up to isometry. A prime caterpillar is only determined by its derived path. In other words, there are six different prime caterpillars, up to the choice of leaves.}
    \label{fig:chenilles premières}
\end{figure}

From Figure \ref{fig:chenilles premières}, one observes that every prime caterpillar is adjacent to a patch of three darts, in white, in the center of the caterpillars. By checking the seven minimal neighborhoods of a vertex in the P2 tiling in Figure \ref{fig:collages de Conway}, it follows that each prime caterpillar is adjacent to a P2 star (Figure \ref{star}). That will be relevant later.

By examining all possible cases, one finds that it is impossible to extend any prime caterpillar so as to obtain a fully leafed induced subtree with $9$ internal tiles of degree $3$ and none of degree $2$. Consequently, any fully leafed induced subtree of a P2-graph with more than $8$ internal tiles must contain at least one tile of degree $2$. Figure~\ref{fig:a fully leafed caterpillar with 17 internal vertices} illustrates such a case. In this paper, degree-2 tiles will always be in blue.

\begin{figure}[H]
    \centering
    \includegraphics[width=0.4\textwidth]{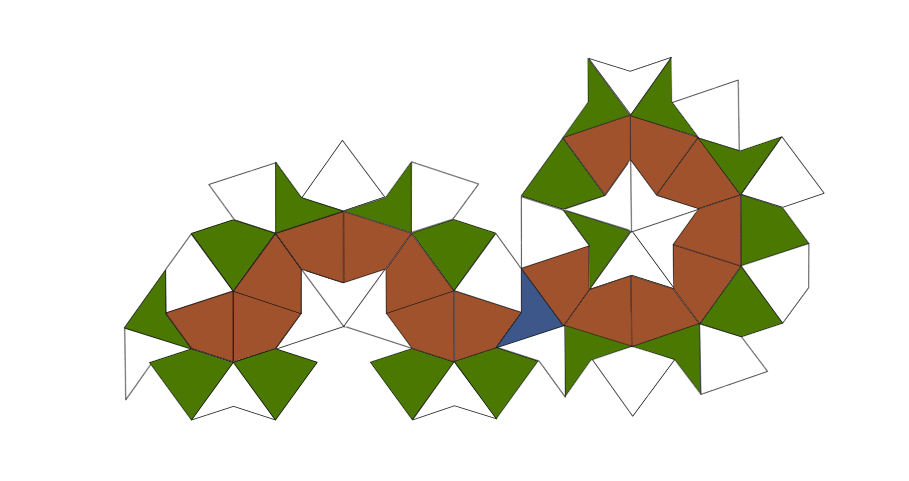}
    \caption{A fully leafed induced subtree of a $P2$-graph with 17 internal vertices. This subtree extends the prime caterpillar $PC_6$ depicted in Figure \ref{fig:CP6}. One leaf of this caterpillar has become a degree-2 tile (in blue), which allowed the addition of eight further internal degree-3 tiles.}
    \label{fig:a fully leafed caterpillar with 17 internal vertices}
\end{figure}

The following definition first appeared in  \cite{masse2018saturated}. However, we adopt a modified version found in \cite{porrier2023leaf}.

\begin{definition}[\cite{masse2018saturated} and \cite{porrier2023leaf}]
Let $I$ be a fully leafed induced subtree of a P2-graph and let $I_1$ and $I_2$ be fully leafed induced subtrees of the same P2-graph such that $I_1 \cup I_2 = I$ and $I_1 \cap I_2 = t$, where $t$ is a leaf of both $I_1$ and $I_2$. We say that $I$ is a \textbf{grafting} of $I_1$ and $I_2$ at $t$, or that $I_1$ is grafted onto $I_2$ at $t$ (and vice versa), and we denote this by $I=I_1 \diamond_t I_2$. When it is not important to identify the tile $t$ or when this tile is already identified in a figure, we simply write $I=I_1 \diamond I_2$.

\end{definition}

If we have $I_1 \diamond I_2$, $I_2 \diamond I_3, \ldots, I_{n-2} \diamond I_{n-1}$, and $I_{n-1} \diamond I_n$, we denote this sequence of graftings by $I_1 \diamond I_2 \diamond \ldots \diamond I_n=\Diamond_{i=1}^n I_i$.

\begin{definition}
    A fully leafed caterpillar is \textbf{bi-infinite} if its construction requires infinitely many graftings of finite fully leafed caterpillars and if it extends infinitely in both directions.
\end{definition}

Now that we have introduced the basic concepts underlying the construction of fully leafed induced subtrees of a P2-graph, we can proceed to the formal construction of a bi-infinite fully leafed caterpillar.

The bi-infinite caterpillar  presented next was first introduced in \cite{porrier2023leaf}. Here we  prove that this construction corresponds to a bi-infinite caterpillar and we characterize it by a bi-infinite word.

\begin{definition}
    A \textbf{Porrier caterpillar} is defined as any of the four caterpillars of Figure \ref{fig:Chenilles de Porrier}, obtained as the grafting of  seven prime caterpillars.
\end{definition}

In Figure \ref{fig:Chenilles de Porrier}, marks have been added to facilitate the recognition of a Porrier caterpillar: a purple chain connects the vertices that are in the center of the stars adjacent to the prime caterpillars. The center of these stars is marked with a red, green, or blue vertex depending on whether the star is adjacent respectively to 0, 1, or 2 suns. This color code will be used throughout this paper. Reading the vertex colors along the purple chain from left to right, there are four possible such chains labeled A, B, C, and D. The chain A is described by the color sequence green-red-blue-green-blue-red-green, the chain B corresponds to the color sequence blue-red-blue-green-blue-red-green, the chain C to the color sequence green-red-blue-green-blue-red-blue and the chain D to the color sequence blue-red-blue-green-blue-red-blue.

\begin{theorem} \label{thm 1}
In any P2-graph, there exists a bi-infinite fully leafed caterpillar.
\end{theorem}

\begin{proof}
Consider each of the four Porrier caterpillars depicted in Figure~\ref{fig:Chenilles de Porrier}. For each Porrier caterpillar, we examine the kingdom that it determines and apply three successive inflations to this kingdom. Within this construction, we find a fully leafed caterpillar obtained by grafting seven Porrier caterpillars. The resulting caterpillars for each of the cases are illustrated in Figure~\ref{fig:3 inf de chaînes}.

By grafting two Porrier caterpillars to form a single caterpillar and applying successive inflations, one can construct a caterpillar composed of 14 grafted Porrier caterpillars, as illustrated in Figure~\ref{fig:greffages par inflation}. In Figure \ref{fig:greffages par inflation}, the black chain represents the two initial Porrier caterpillars before inflations (ignoring the extremities where grafting does not occur), while the purple and orange chains indicate the chains formed of seven Porrier caterpillars each. The yellow segments highlight the grafting region between two initial Porrier caterpillars (before applying inflations) and the light-green segments highlight the grafting region of two caterpillars identified by purple and orange chains.

\begin{figure}[H]
    \centering
    \begin{subfigure}{0.34\textwidth}
        \centering
        \includegraphics[width=\textwidth, trim=0cm 0.4cm 0cm 0.7cm, clip]{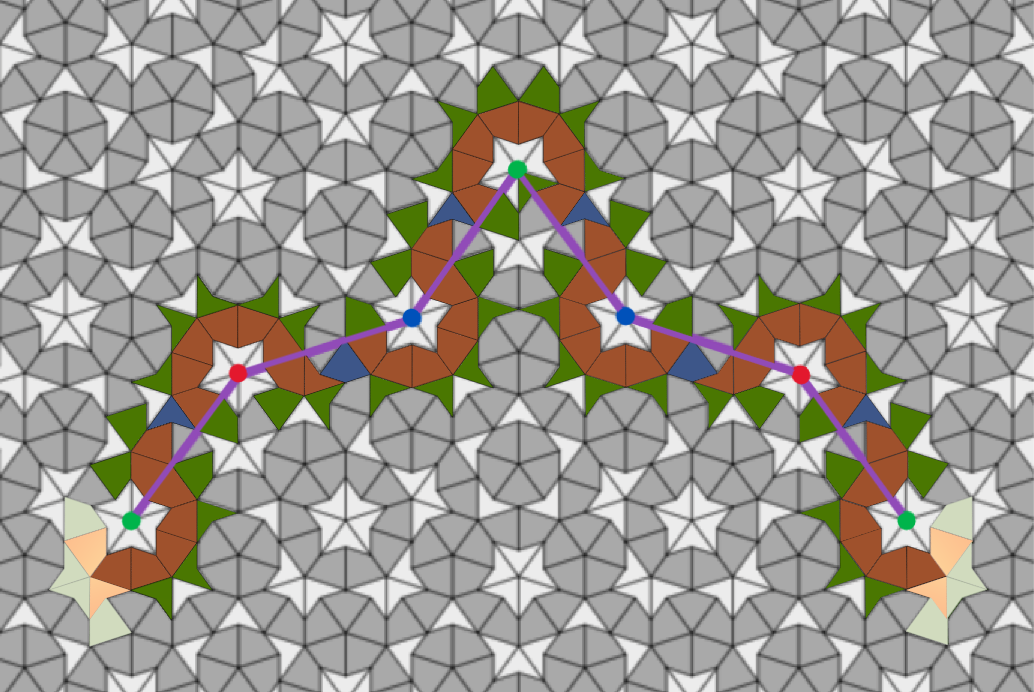}
        \caption{On chain A}
    \end{subfigure} 
    \begin{subfigure}{0.34\textwidth}
        \centering
        \includegraphics[width=\textwidth, trim=0cm 1.1cm 0cm 0cm, clip]{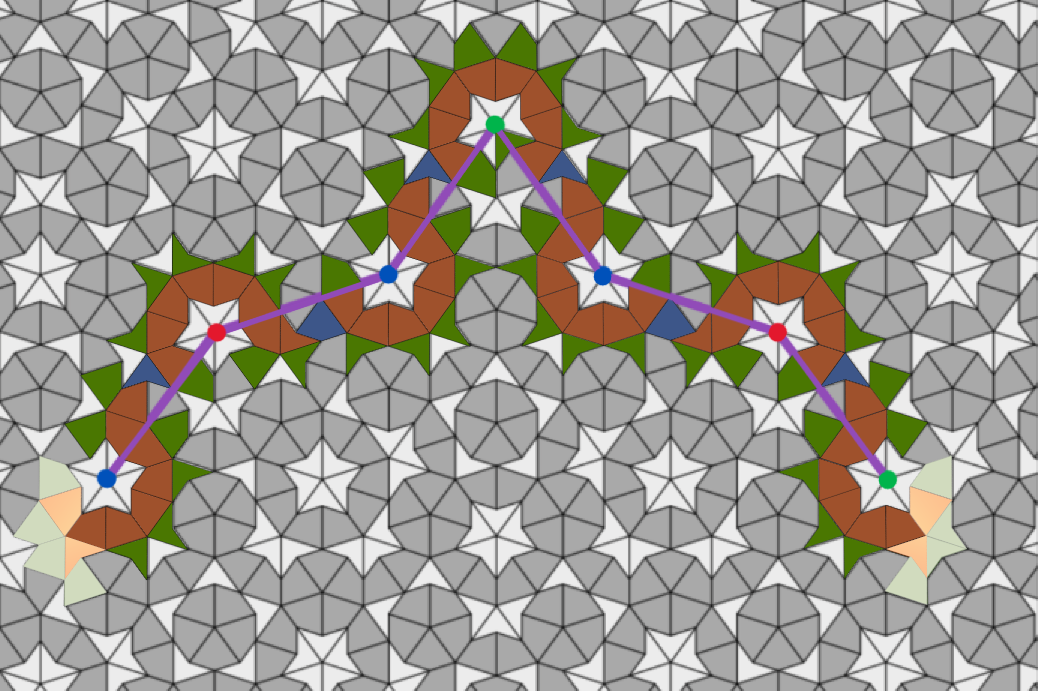}
        \caption{On chain B}
    \end{subfigure} 
    \begin{subfigure}{0.34\textwidth}
        \centering
        \includegraphics[width=\textwidth, trim=0cm 1.1cm 0cm 0cm, clip]{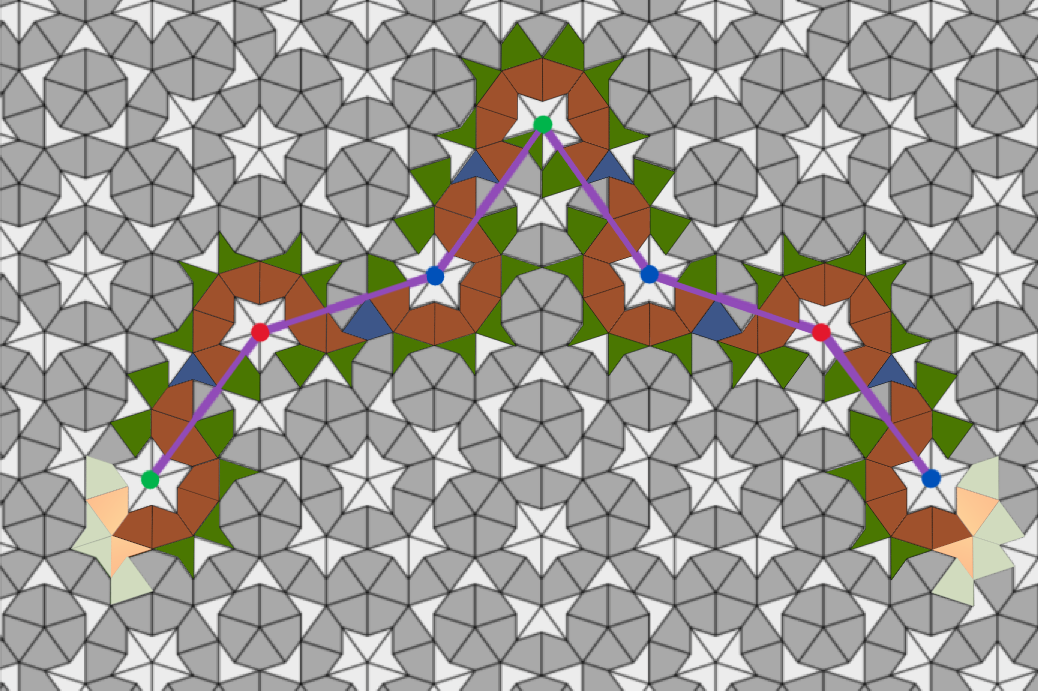}
        \caption{On chain C}
    \end{subfigure} 
    \begin{subfigure}{0.34\textwidth}
        \centering
        \includegraphics[width=\textwidth]{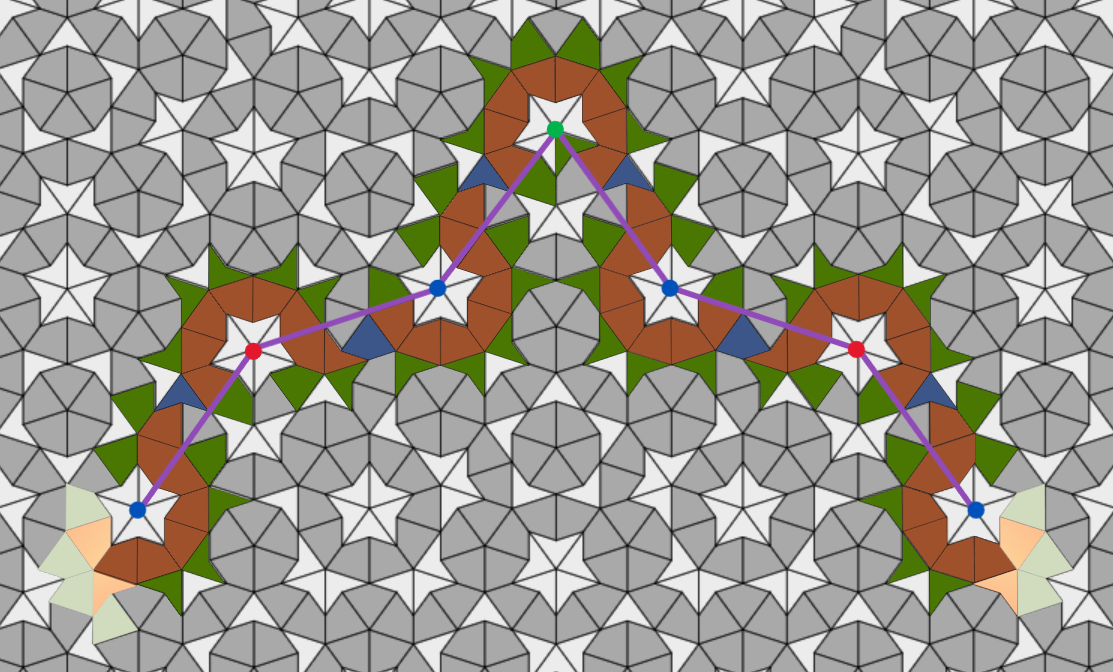}
        \caption{On chain D}
    \end{subfigure} 
    \caption{The four Porrier caterpillars. At each end, there are two beige tiles. Only one may be chosen, which then becomes an internal tile of degree 3, along with the two adjacent light-green tiles, which then become leaves.}
    \label{fig:Chenilles de Porrier}

\end{figure}

\begin{figure}[H]
    \centering
    \begin{subfigure}{0.335\textwidth}
        \centering
        \includegraphics[width=\textwidth]{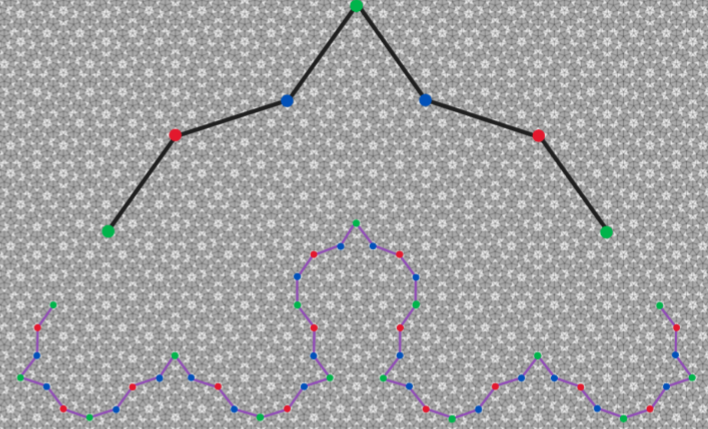}
        \caption{Case A}
    \end{subfigure} 
    \begin{subfigure}{0.33\textwidth}
        \centering
        \includegraphics[width=\textwidth, trim=0.1cm 0cm 0.3cm 0cm, clip]{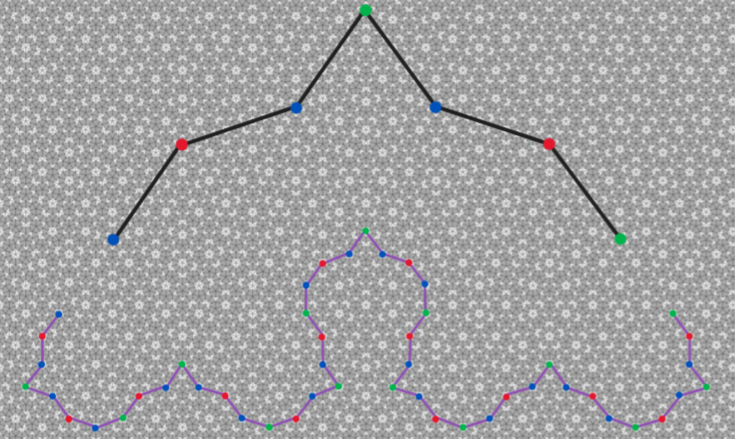}
        \caption{Case B}
    \end{subfigure} 
    \begin{subfigure}{0.33\textwidth}
        \centering
        \includegraphics[width=\textwidth]{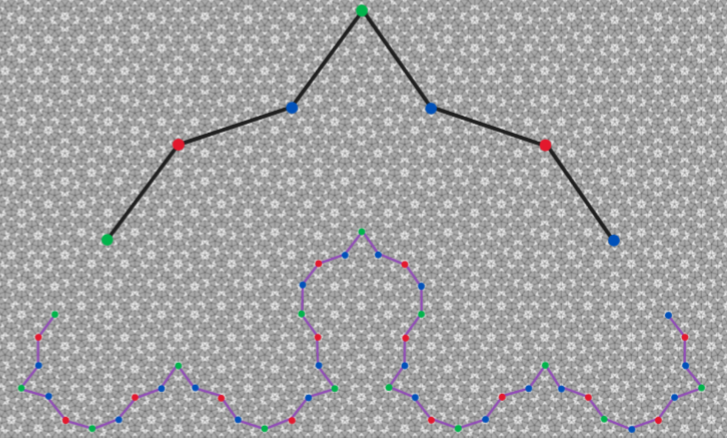}
        \caption{Case C}
    \end{subfigure} 
    \begin{subfigure}{0.335\textwidth}
        \centering
        \includegraphics[width=\textwidth, trim=1cm 0cm 2cm 0.5cm, clip]{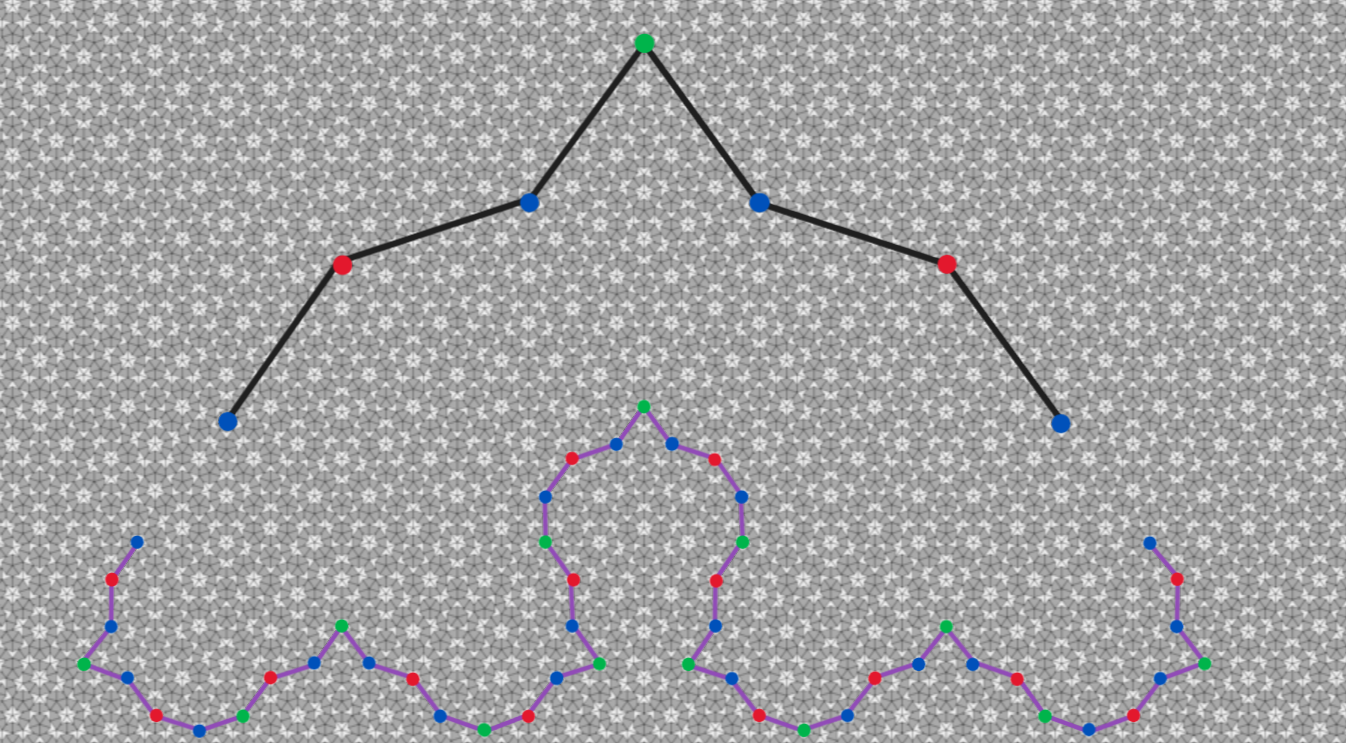}
        \caption{Case D}
    \end{subfigure} 
    \caption{The result of applying three consecutive inflations to a Porrier caterpillar. The black chains correspond to the purple chains from Figure \ref{fig:Chenilles de Porrier} (before applying three successive inflations). The purple chains indicate the structure associated with the grafting of seven Porrier caterpillars (after applying three successive inflations).}
    \label{fig:3 inf de chaînes}

\end{figure}

\begin{figure}[h]
    \centering
    \begin{subfigure}{0.45\textwidth}
        \centering
        \includegraphics[width=\textwidth]{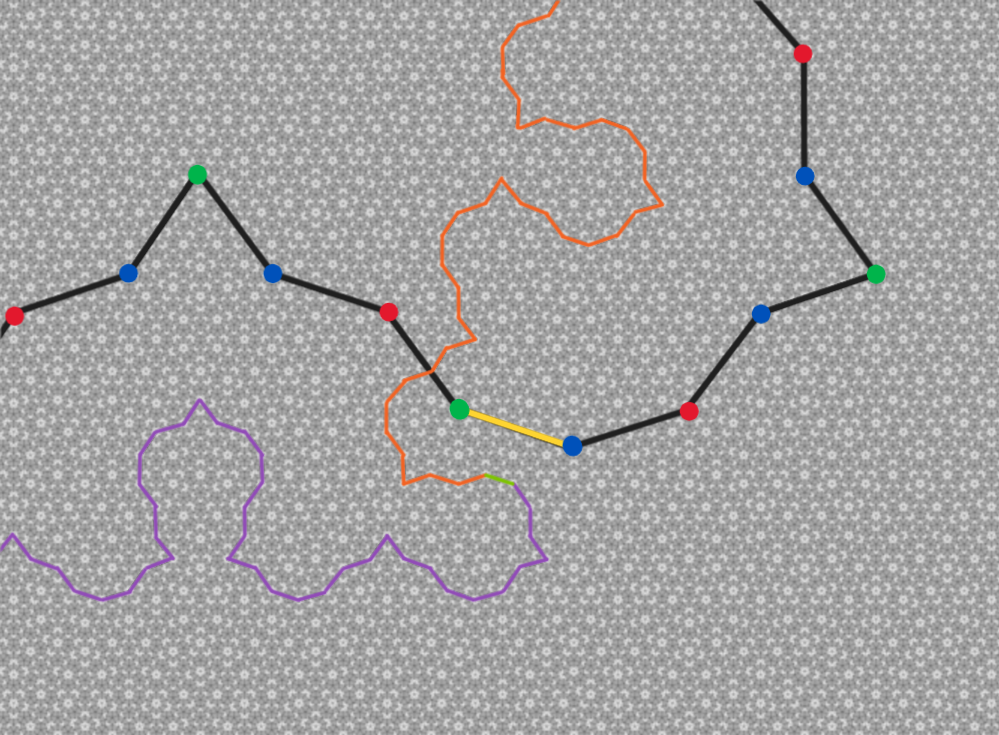}
        \caption{First possible grafting of two Porrier caterpillars}
    \end{subfigure} 
    \begin{subfigure}{0.45\textwidth}
        \centering
        \includegraphics[width=\textwidth, trim=0cm 0cm 0cm 0.81cm, clip]{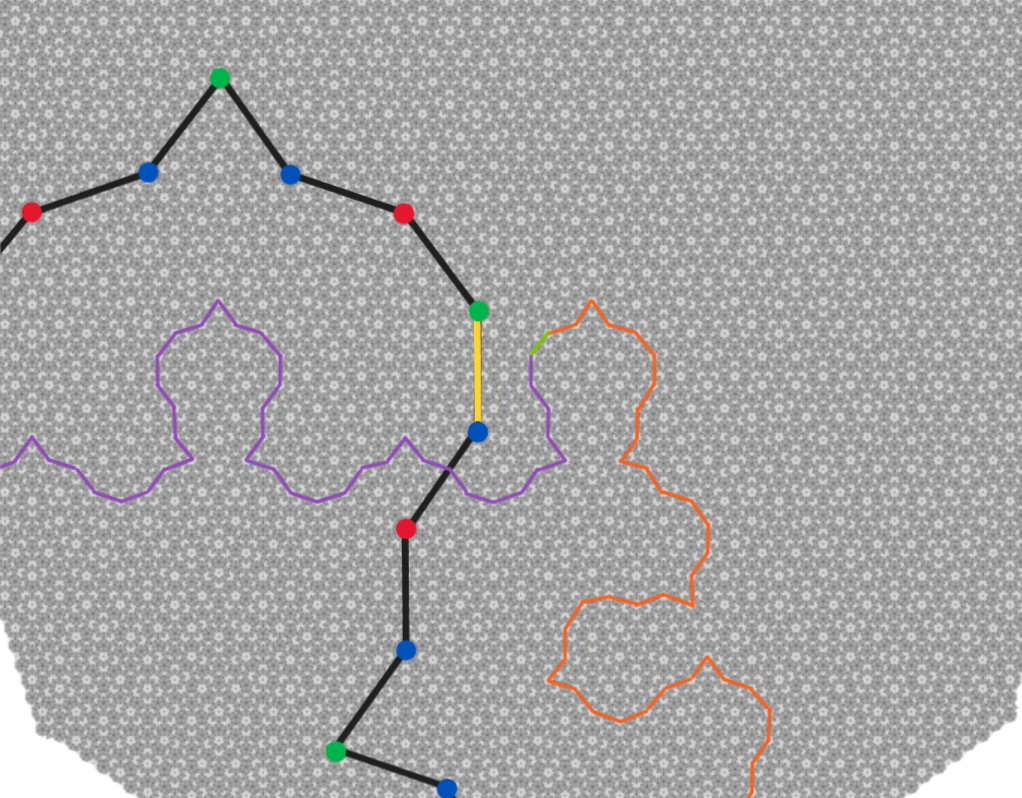}
        \caption{Second possible grafting of two Porrier caterpillars}
    \end{subfigure} 
    \caption{Graftings of two Porrier caterpillars and their successive three inflations forming a fully leafed caterpillar composed of 14 Porrier caterpillars}
    \label{fig:greffages par inflation}
\end{figure}

Consequently, for any $n \in \mathbb{N}$, it is possible to apply $3n$ successive inflations to a region determined by a Porrier caterpillar and obtain a fully leafed caterpillar formed by grafting $7n$ Porrier caterpillars. The previous figures ensure that the resulting graph is connected and acyclic: no cycles are formed locally, and by the local isomorphism property, in any P2-graph, by taking the limit, a bi-infinite fully leafed caterpillar exists.
\end{proof}

For practical reasons, we shall represent the previous construction using words over a four-letter alphabet. The vertex-color sequence followed by a chain associated with a Porrier caterpillar is always of the form \textit{$X$-red-blue-green-blue-red-$Y$}, where $X,Y \in \{\text{green, blue}\}$. This fact  defines four chains described by the alphabet $\Sigma := \{A,B,C,D\}$ and we define the letters $A,B,C,D$ by the following sequences:

\begin{align*}
    A &:= \text{green-red-blue-green-blue-red-green;}\\
    B &:= \text{blue-red-blue-green-blue-red-green;}\\
    C &:= \text{green-red-blue-green-blue-red-blue;} \\
    D &:= \text{blue-red-blue-green-blue-red-blue}.
\end{align*}

In accordance with Figure~\ref{fig:3 inf de chaînes}, we define the function $\Phi : \Sigma^* \to \Sigma^*$ by

\begin{align}
\Phi(A) &= ADADADA \label{1}\\
\Phi(B) &= DCADADA \label{2} \\
\Phi(C) &= ADADABD \label{3} \\
\Phi(D) &= DCADABD. \label{4}
\end{align}

From Figure~\ref{fig:greffages par inflation}, $\Phi$ is a morphism, and it is therefore fully determined by the four equalities $\eqref{1}$ to $\eqref{4}$.

\begin{corollary}
The procedure described in the proof of Theorem \ref{thm 1} constructs a uniquely determined bi-infinite fully leafed caterpillar.
\end{corollary}

\begin{proof}
We show that the sequences $(\Phi^n(A))_{n\in \mathbb{N}}, (\Phi^n(B))_{n\in \mathbb{N}}, (\Phi^n(C))_{n\in \mathbb{N}}$, and $(\Phi^n(D))_{n\in \mathbb{N}}$ all converge to the same bi-infinite word.

For each $X \in \Sigma$, we write $\Phi(X) = p_X D s_X$, where $p_X$ and $s_X$ are respectively the prefix and the suffix of length 3, according to Equations \eqref{1} to \eqref{4}. By induction on $n$, we first show that for all $X\in \Sigma$, 
\[
\Phi^n(X) = p_{n,X} D s_{n,X},
\]
with $p_{n,X}$ and $s_{n,X}$ respectively prefix and suffix of length $\frac{7^n-1}{2}$. The base case $n=1$ follows directly from the definition of $\Phi$. Now, assuming the statement holds for $n$, for all $X\in \Sigma$, we have
\[
\Phi^{n+1}(X) = \Phi^n(\Phi(X)) = \Phi^n(p_X)\,\Phi^n(D)\,\Phi^n(s_X), \hspace{.5 cm}
\]
and by the induction hypothesis, $\Phi^n(D) = p_{n,D} D s_{n,D}$, so
\[
\Phi^{n+1}(X) = (\Phi^n(p_X) p_{n,D})\, D \, (s_{n,D} \Phi^n(s_X)) = p_{n+1,X} D s_{n+1,X},
\]
with lengths $\lvert p_{n+1,X} \rvert= \lvert\Phi^n(p_X) \rvert + \lvert p_{n,D}  \rvert = 3 \cdot 7^n + \frac{7^n-1}{2} = \frac{7^{n+1}-1}{2}$ and $\lvert s_{n+1,X} \rvert = \frac{7^{n+1}-1}{2}$.

The sequence $(s_{n,X})_{n\in\mathbb{N}}$ therefore converges to the left-infinite word

\[
s_D \Phi(s_X) \Phi^2(s_X) \Phi^3(s_X) \dots,
\]

and similarly for the sequence $(p_{n,X})_{n\in\mathbb{N}}$. Hence, for all $X \in \Sigma$, $(\Phi^n(X))_{n\in\mathbb{N}}$ converges to a bi-infinite word.

Finally, the limit is independent of $X$ because, in the limit, every $\Phi^n(X)$ contains $\Phi^n(D)$ in the central position and $\lim_{n\rightarrow \infty}\Phi^n(D)$ is a bi-infinite word.
\end{proof}

\section{Structure of the fully leafed induced subtrees}

All the fully leafed induced subtrees studied so far are caterpillars. In this section, we make precise the graph structure of fully leafed induced subtrees in P2-graphs in general. Theorem~\ref{thm2} and Proposition~\ref{prop3} together prove the two parts of Conjecture~1 in \cite{porrier2023leaf}.

\subsection{General structure}

We describe here the general structure of the fully leafed induced subtrees of P2-graphs.

\begin{definition}
Let $\mathcal{C}$ be a fully leafed caterpillar and $A$ a fully leafed induced subtree of a P2-graph such that $\mathcal{C}\diamond A$ is a fully leafed induced subtree. We say that $A$ is an \textbf{appendix} of $\mathcal{C}$ if the tile shared by $\mathcal{C}$ and $A$ is not adjacent to a leaf of $\mathcal{C}'$.
\end{definition}

\begin{lemma} \label{app 2}
    An appendix of a prime caterpillar has at most two internal tiles (for the appendix).
\end{lemma}
\begin{proof}
    From Figure \ref{fig:chenilles premières}, we observe that the caterpillar of Figure \ref{fig:C6} is a subcaterpillar of any prime caterpillar. By the definition of an appendix, it is sufficient to consider all possible ways of extending this caterpillar into a fully leafed induced subtree in which a degree-2 tile is adjacent to a tile of the configuration in Figure \ref{fig:C6} (excluding, of course, the possibility of placing a degree-2 tile where we know a degree-3 tile is required in order to form a prime caterpillar). Figure \ref{fig:théorème 1} presents all possible constructions.
    
    \begin{figure}[H]
    \centering
    \includegraphics[width=0.2\textwidth]{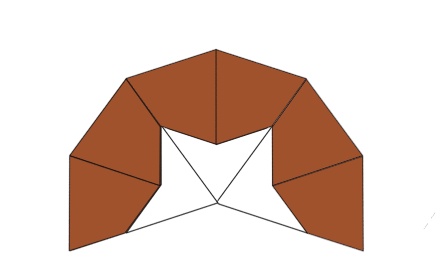}
    \caption{Caterpillar obtained by double derivation of any prime caterpillar}
    \label{fig:C6}

\end{figure} 

    \begin{figure}[H]
    \centering
    \includegraphics[width=1\textwidth]{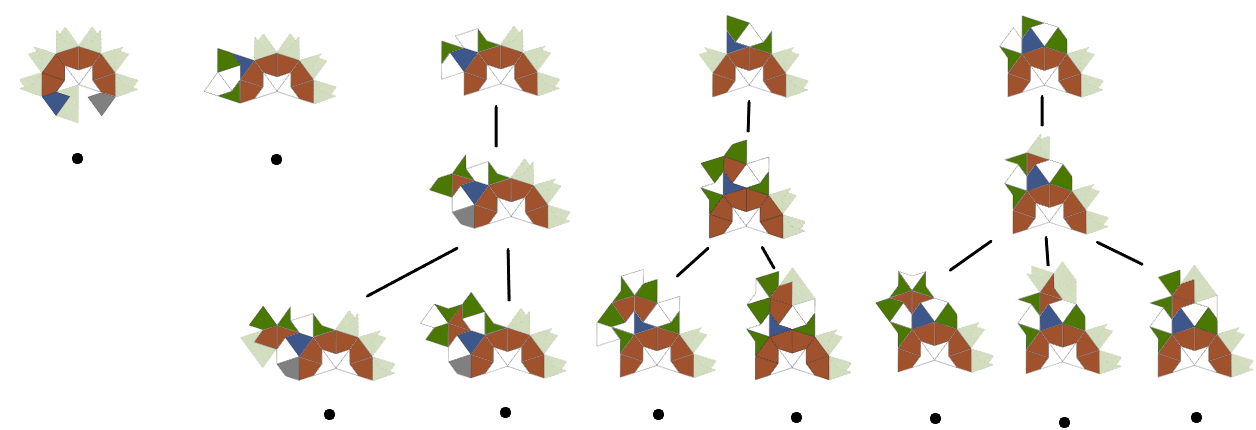}
    \caption{Set of all possible ways to graft an appendix onto a prime caterpillar (up to isometry). Grey tiles belong to the fully leafed induced subtree, without specifying whether the tile has degree 1 or 3. The first row illustrates all possible choices for the degree-2 tile in blue. The second row shows the three ways to graft an appendix of a single internal tile of degree $3$. The third row shows all possible ways to graft an appendix of two internal tiles of degree $3$. A dot is placed below an induced subtree when it becomes impossible to extend it into a fully leafed induced subtree.}
    \label{fig:théorème 1}
\end{figure}  
\end{proof}

\begin{theorem} \label{thm2}
    Every fully leafed induced subtree of a P2-graph is one of the following structures:
    \begin{enumerate}
        \item[1.] A proper subcaterpillar of a prime caterpillar structure;
        \item[2.] A caterpillar obtained by grafting one or several prime caterpillars and at most one proper subcaterpillar of a prime caterpillar structure;
        \item[3.] A caterpillar obtained by grafting one or several prime caterpillars, to which an appendix of at most two internal tiles is grafted.
    \end{enumerate}
\end{theorem}

\begin{proof}
Let $A$ be a fully leafed induced subtree of a P2-graph and let $k$ denote the order of $A'$. So there is a non-negative integer $m$ such that $9m\leq k\leq 9m+8$.

First, $A$ must contain at least $m$ tiles of degree 2. Otherwise, $A$ would contain an induced subtree with nine internal tiles all of degree~3. 

If $A$ has more than $m$ degree-2 tiles, then the number of degree-2 tiles is not minimal (equivalently, the number of degree-3 tiles is not maximal), since no matter how large $k$ is, it is possible to construct a fully leafed induced subtree with $k$ internal tiles and exactly $m$ tiles of degree~2 by Theorem \ref{thm 1}. Thus $A$ has exactly $m$ degree-2 tiles.\\

We now complete the proof by induction on $m$. If $m=0$, the result follows from Lemma \ref{lemme 4}.  If $0\leq k < 8$ then we are in  case 1 of the theorem.  If $k=8$ then we are in  case 2 of the theorem.  Now suppose  $m \geq 1$ and consider $A$ as an extension of a fully leafed induced subtree $B$ with $m-1$ degree-2 tiles. So for $I$ the fully leafed induced subtree grafted to $B$ to construct $A$, $I$ does not contain any degree-2 tile and all internal tiles of $I$ are degree-3 tiles, so then $I$ is either a prime caterpillar  or is a proper subcaterpillar  of a prime caterpillar.

If $B$ has no appendix, then wherever the degree-2 tile used to complete $B$ into $A$ is, the conclusion of Theorem \ref{thm2} holds. If $B$ already has an appendix and the addition of a new degree-2 tile to construct $A$ creates a second appendix for $A$, then by Lemma \ref{app 2}, between two and four internal tiles of $A$ belong to appendices. In this case, $A$ is not fully leafed, since it is possible to construct a fully leafed caterpillar of the same order as $A$ (by Theorem \ref{thm 1}) with more leaves (by Corollary \ref{cor 1}). For the same reason, $A$ satisfies case 2 or case 3, but not both.
\end{proof}

\subsection{Saturated fully leafed induced subtrees}

We denote the leaf function of any P2-graph by $L_{P2}$. From \cite{porrier2023leaf}, we have the following recursive formula:

\begin{equation} \label{formule récursive}
  L_{P2}(n) = 
\begin{cases}
0 & \text{ if } 0 \leq n \leq 1, \\
\left\lfloor \frac{n}{2} \right\rfloor + 1 & \text{ if } 2 \leq n \leq 18, \\
L_{P2}(n-17)+8 &\text{ if } n\geq 19.\\
\end{cases}
\end{equation}

\vspace{30pt}

We first correct an error in \cite{porrier2023leaf} (Proposition 3 of \cite{porrier2023leaf}) and present a correct closed formula for $L_{P2}(n)$.
\begin{proposition} \label{prop:formule close} 
$L_{P2}(n)=
\begin{cases}
0 & \text{if } 0 \leq n \leq 1, \\ 
\left\lfloor \dfrac{n}{2} \right\rfloor + 1 & \text{if } 2 \leq n \leq 18, \\ 
8 \left\lfloor \dfrac{n}{17} \right\rfloor + \left\lfloor \dfrac{n \bmod 17}{2} \right\rfloor + 1 + \mathbbm{1}(n \bmod 17 = 1) & \text{if } n \geq 19,
\end{cases}$

where 
\[
 \mathbbm{1}(x) =
\begin{cases}
1 
& \mbox{ if $x$ is true},  \\
0 & \mbox{ otherwise }
\end{cases}
\]
is the usual characteristic function. 
\end{proposition}
\begin{proof}\footnote{The proof is similar to the one in \cite{porrier2023leaf}, with a correction.}
For $0\leq n\leq 18$ the claim is immediate from recursion \eqref{formule récursive}.
For $n\ge 19$, we prove by strong induction on $n$ the recurrence
\begin{equation} \label{pour formule close}
    L_{P2}(n)=L_{P2}\!\bigl(n-17(\lfloor n/17\rfloor -1)\bigr) + 8(\lfloor n/17\rfloor -1).
\end{equation}

For $19\leq n\leq 33$, since $\lfloor 33/17\rfloor=1=\lfloor 19/17\rfloor$, we easily verify that the right-hand side of Equation \eqref{pour formule close} equals $L_{P2}(n)$. For $n=34$ or $n=35$, $\lfloor n/17\rfloor=2$, and by the recursive formula \eqref{formule récursive}, Equation \eqref{pour formule close} holds. Suppose now that Equation \eqref{pour formule close} holds for every $n$ such that $35\leq n \leq k$, for some arbitrary integer $k\geq 35$. We prove that Equation \eqref{pour formule close} holds for $n=k+1$. By the recursion \eqref{formule récursive} and by the induction hypothesis, we have
\begin{align}
 L_{P2}(k+1)&=L_{P2}(k-16)+8\\
 &=L_{P2}\left(k-16-17\left(\left\lfloor \frac{k-16}{17}\right\rfloor-1\right)\right) +8\left(\left\lfloor \frac{k-16}{17}\right\rfloor-1\right)+8.\label{5} 
\end{align}
Since $k+1\geq 2(17)$, we have

\begin{align} \label{6}
    \left\lfloor \frac{k+1}{17}\right\rfloor=\left\lfloor \frac{k-16}{17}\right\rfloor+1. 
\end{align}

From Equations \eqref{5} and \eqref{6}, we obtain

\begin{align*}
    L_{P2}(k+1)&=L_{P2}\left(k-16-17\left(\left\lfloor \frac{k+1}{17}\right\rfloor-2\right)\right)+8\left(\left\lfloor \frac{k+1}{17}\right\rfloor-2\right)+8\\
&=L_{P2}\left(k+1-17\left(\left\lfloor \frac{k+1}{17}\right\rfloor-1\right)\right)+8\left(\left\lfloor \frac{k+1}{17}\right\rfloor-1\right),
\end{align*}
which completes the induction. So Equation \eqref{pour formule close} holds for all $n\ge 19$.\\

Now let $n\geq 19$ and write $n=17q+r$ so that $q=\lfloor n/17\rfloor$ and $r=n \bmod 17$ (the remainder of the division of $n$ by 17). So we have

\begin{equation} \label{eq:reduce-to-r+17}
    L_{P2}(n)=L_{P2}(r+17)+8(q-1).
\end{equation}

If $2\le r\le 16$, then $r+17\ge 19$ and by recursion \eqref{formule récursive} we obtain $L_{P2}(r+17)=L_{P2}(r)+8$. Since $r\le 16$, $L_{P2}(r)=\lfloor r/2\rfloor+1$, hence
\[
L_{P2}(n)=\lfloor r/2\rfloor+1 + 8 + 8(q-1)=\lfloor r/2\rfloor+1 + 8q,
\]
which satisfies the closed formula of Proposition~\ref{prop:formule close}.

If $0\leq r \leq 1$, then $17\leq r+17 \leq 18$ and by the recursion \eqref{formule récursive},
\[
L_{P2}(r+17)=\left\lfloor \tfrac{r+17}{2}\right\rfloor + 1.
\]
Substituting this into Equation \eqref{eq:reduce-to-r+17}, we obtain

\[
L_{P2}(n)=\left\lfloor \tfrac{r+17}{2}\right\rfloor + 1 + 8(q-1),
\]
and that satisfies the formula of Proposition~\ref{prop:formule close}.

\end{proof}

\begin{lemma}
 As $n\to \infty$, $L_{P2}(n) \sim \frac{8n}{17}$.
\end{lemma}
\begin{proof}
    We compute the limit $\lim_{n\to \infty} \frac{L_{P2}(n)}{n}$.
    \begin{align*}
        \lim_{n\to \infty} \frac{L_{P2}(n)}{n}=\lim_{n\to \infty} \left(\frac{8}{n} \left\lfloor \dfrac{n}{17} \right\rfloor + \frac{\left\lfloor \dfrac{n \bmod 17}{2} \right\rfloor + 1 + \mathbbm{1}(n \bmod 17 = 1)}{n}\right)=\frac{8}{17}.
    \end{align*}
\end{proof}

Since $L_{P2}$ has an asymptotic linear growth with slope $\frac{8}{17}$, there exists a linear function $\overline{L_{P2}} = \frac{8}{17}n + b$ and an integer $N\in \mathbb{N}$ such that for all $n\geq N$, $\overline{L_{P2}}(n)\geq L_{P2}(n)$ with minimal intercept $b \in \mathbb{R}$.

\begin{definition}[\cite{masse2018saturated}]
    For $\overline{L_{P2}}$ as described above, a fully leafed induced subtree of order $n$ is \textbf{saturated} if $L_{P2}(n) = \overline{L_{P2}}(n)$.
\end{definition}

\begin{proposition} \label{sup}
    For all integers $n\geq 1$, $\overline{L_{P2}}(n) = \frac{8n+26}{17}$.
\end{proposition}

\begin{proof}
    Let $f(n)=\frac{8n+26}{17}$. We first show that $L_{P2}(n) \leq f(n)$ for all $n \in \mathbb{N}$.

    For $0 \leq n \leq 1$, this is obvious. For $2 \leq n \leq 18$, we have
    \[
        \left\lfloor \tfrac{n}{2} \right\rfloor + 1 \leq \tfrac{n}{2}+1 \leq \tfrac{8n+26}{17},
    \]

    hence the inequality holds.

    For $n \geq 19$, we can write $n=17q+r$ with $q=\lfloor n/17\rfloor$ and $r=n \bmod 17$. Hence, by substituting this into the closed formula of Proposition \ref{prop:formule close}, the inequality to verify is reduced to
    \[
        \left\lfloor \tfrac{r}{2} \right\rfloor + 1 + \mathbbm{1}(r=1) \leq \tfrac{8r+26}{17}.
    \]
    If $r=1$, both sides equal $2$. Otherwise, $\mathbbm{1}(r=1)=0$ and we easily verify that the inequality is valid. Thus $L_{P2}(n) \leq f(n)$ for every $n\in \mathbb{N}$.

    To prove minimality, let $n_0=17k+1$ with $k \geq 2$. Then, by Proposition \ref{prop:formule close},
    \[
        L_{P2}(n_0)=8k+2=\tfrac{8(17k+1)+26}{17}=f(n_0),
    \]
    so the bound $f(n)$ is reached infinitely often. Hence no smaller affine function of slope $\tfrac{8}{17}$ can bound $L_{P2}$, and the result follows.
\end{proof}

\begin{lemma}\label{sat}
    For an integer $n\geq 1$, let $A$ be a fully leafed induced subtree of order $n$. Then $A$ is saturated if and only if $n= 17k+1$ for some integer $k\geq 1$.
\end{lemma}

\begin{proof}
    Write $n=17q+r$ with $q=\left\lfloor \tfrac{n}{17}\right\rfloor$ and $r=n \bmod 17$.
 
   For $n=18$, we have
    \[
        L_{P2}(18)=\left\lfloor \tfrac{18}{2}\right\rfloor+1=10=\tfrac{8\cdot 18+26}{17}=\overline{L_{P2}}(18),
    \]
    hence $A$ is saturated. For larger $n=17k+1$, the conclusion follows from the end of the proof of Proposition \ref{sup}.\\

    Conversely, suppose $A$ is saturated. If $n\leq 17$, the equality
    \begin{align}
        \left\lfloor \tfrac{n}{2}\right\rfloor+1=\tfrac{8n+26}{17} \label{equation a}
    \end{align}
   
    has no solution. Indeed, if $\left\lfloor \tfrac{n}{2}\right\rfloor=\frac{n}{2}$, Equation \eqref{equation a} implies $n=18$ and if $\left\lfloor \tfrac{n}{2}\right\rfloor=\frac{n-1}{2}$, Equation \eqref{equation a} implies $n=35$, and none of these two cases satisfies $n\leq 17$. We now consider $n\geq 18$. If $n=18$, Equation \eqref{equation a} is satisfied. If $n\geq 19$ with $r\neq 1$, then from Propositions \ref{prop:formule close} and \ref{sup}, we need to have
    
    \[
        \left\lfloor \tfrac{r}{2}\right\rfloor+1=\tfrac{8r+26}{17}.
    \]

    But, as we have just checked, this equation has no solution for $0\leq r\leq 16$. Thus $r=1$, and the claim follows.
\end{proof}

The next proposition shows that a fully leafed induced subtree is saturated if and only if it is a prime caterpillar or a caterpillar obtained by successive graftings of prime caterpillars.

\begin{proposition} \label{prop3}
Let $A$ be a fully leafed induced subtree. Then $A$ is saturated if and only if there exists a sequence of prime caterpillars $(C_k)_{k=1}^m$, for some $m \geq 1$, such that 
\[
A=\Diamond_{k=1}^m C_k.
\]

\end{proposition}

\begin{proof}
Let $A$ be a fully leafed induced subtree of order $n$. First, for some $m \geq 1$, suppose that there exists a sequence of prime caterpillars $(C_k)_{k=1}^m$, such that $A=\Diamond_{k=1}^m C_k$. Each prime caterpillar contains 8 internal degree-3 tiles (by definition). Hence, $A$ has $8m$ degree-3 tiles. Moreover, for all $i\in \{1, \ldots, m-1\}$, the grafting $C_i \diamond _t C_{i+1}$ is such that $t$ is a degree-2 tile for $A$. By these remarks and Lemmas \ref{max 3} and \ref{lem2},

\[
n = 8m + (m-1) + (8m+2) = 17m+1.
\]
Thus, by Lemma \ref{sat}, $A$ is saturated.\\

Conversely, suppose that $A$ is saturated. Then $A$ must be a caterpillar. If it is not, by Theorem~\ref{thm2}, it would be the grafting of prime caterpillars with an appendix of at most two internal degree-3 tiles. In this case we would have 
\[
17i+2 \leq n \leq 17i+6
\] 
for some integer $i$, contradicting  Lemma \ref{sat}.

If there is no sequence of prime caterpillars $(C_k)_{k=1}^m$, with $m \geq 1$, that satisfies $A=\Diamond_{k=1}^m C_k$, since $A$ must be a caterpillar, either it is a proper subset of a prime caterpillar, or there exists a sequence of prime caterpillars $(D_k)_{k=1}^{m'}$, with $m'\geq 1$, such that 
\[
A=\Diamond_{k=1}^{m'} D_k \diamond E,
\]
where $E$ is a proper subcaterpillar of a prime caterpillar (by Theorem~\ref{thm2}). In the first case, by Lemma \ref{sat}, $A$ cannot be saturated. In the second case, there exists an integer $j\geq 1$ such that 
\[
17j+2 \leq n \leq 17(j+1),
\]
and then by Lemma \ref{sat}, $A$ cannot be saturated. 
\end{proof}

\section{Caterpillar graftings}
 
In this section, we only consider saturated caterpillars. We know that from a saturated caterpillar, any fully leafed induced subtree can be constructed by adding at most 17 tiles. The main goal of this section is to study how fully leafed induced caterpillars can be grafted together to construct larger ones. We are also interested in the geometry determined by the embedding of fully leafed induced caterpillars in P2 tilings.

\subsection{Prime caterpillars and HBS tilings}

We now study how prime caterpillars can be grafted together.

\begin{proposition} \label{prop: the two graftings}
    The only possible graftings of two prime caterpillars are those whose derived path is one of the path shown in Figure \ref{fig:greffages premiers HBS}.
\end{proposition}
\begin{proof}
    Let $C$ be any prime caterpillar. Then from Figure \ref{fig:chenilles premières}, we see that there are three possible choices of tile for each leaf of $C'$. These three possibilities are illustrated for one such leaf in Figure \ref{fig:Les trois manières de prolonger C'' en une chenille d'ordre 7}.

    \begin{figure}[H]
    \centering
    \begin{subfigure}{0.14\textwidth}
        \centering
        \includegraphics[width=\textwidth]{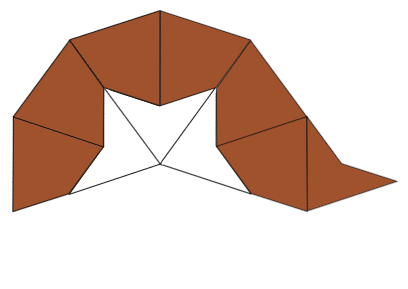}
        \caption{Case 1}
    \end{subfigure} \hspace{25pt}
    \begin{subfigure}{0.14\textwidth}
        \centering
        \includegraphics[width=\textwidth]{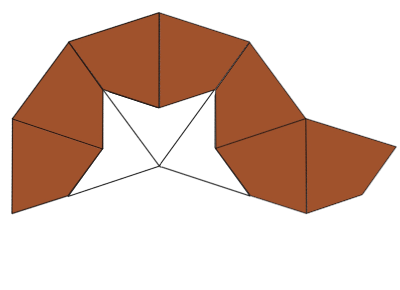}
        \caption{Case 2}
    \end{subfigure} \hspace{25pt}
    \begin{subfigure}{0.14\textwidth}
        \centering
        \includegraphics[width=\textwidth]{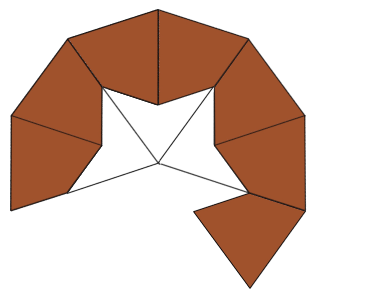}
        \caption{Case 3}
    \end{subfigure}
    \caption{The three ways of extending $C''$ in a caterpillar of order 7 (up to isometry)}
    \label{fig:Les trois manières de prolonger C'' en une chenille d'ordre 7}
\end{figure}

Next, for each of the three ways of adding a seventh tile $t$ to $C''$ to construct a caterpillar of order 7, we consider in Figure \ref{fig:greffages de deux chenilles premières} all possible ways of adding a tile $u$ adjacent to $t$ such that $u$ is a degree-2 tile in every caterpillar obtained by grafting $C$ to another prime caterpillar. In each case, the possible extensions arise by grafting $C$ at the tile $u$.\\

\begin{figure}[H]
    \centering
    \begin{subfigure}{1\textwidth}
        \centering
        \includegraphics[width=\textwidth]{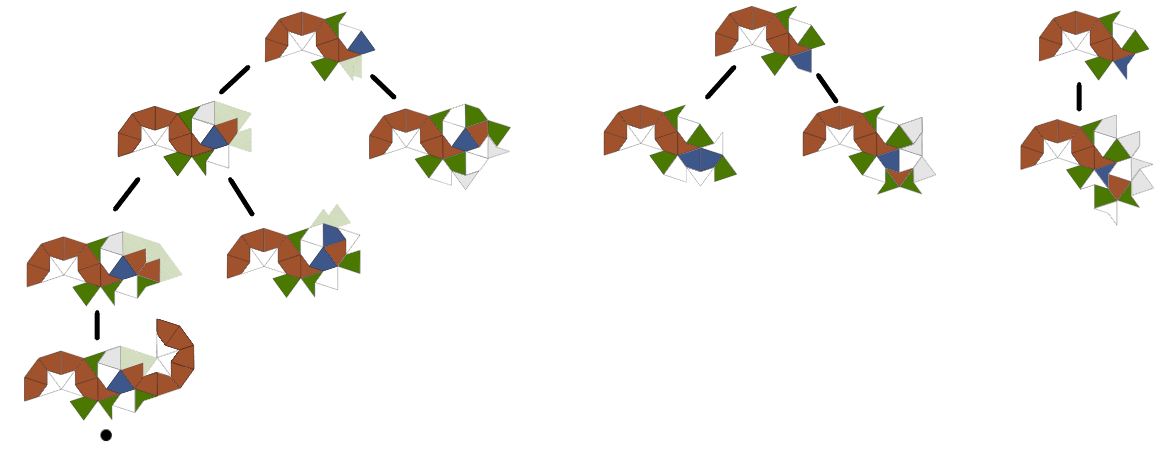}
        \caption{Extensions in Case 1}
    \end{subfigure} 
\end{figure}
\begin{figure}[H]
    \ContinuedFloat
    \centering
    \begin{subfigure}{1\textwidth}
        \centering
        \includegraphics[width=\textwidth]{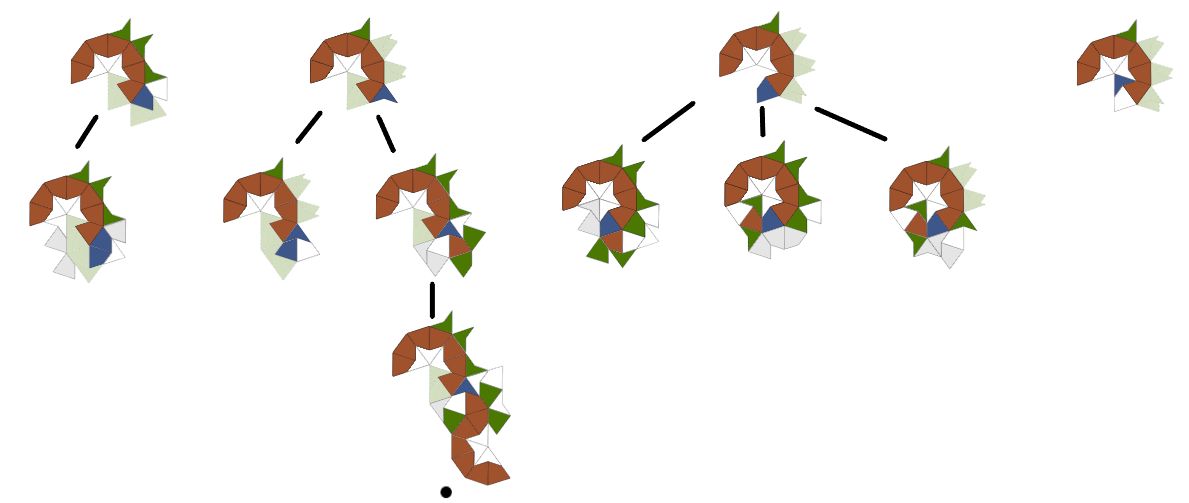}
        \caption*{(b) Extensions in Case 2}
    \end{subfigure}
    \begin{subfigure}{0.4\textwidth}
        \centering
        \includegraphics[width=\textwidth]{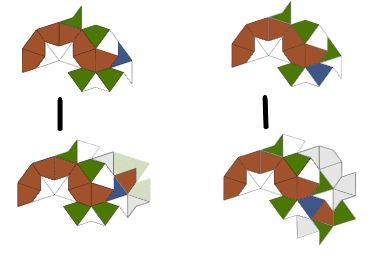}
        \caption*{(c) Extensions in Case 3}
    \end{subfigure}
    \caption{Study of the possible extensions of $C''$ to obtain a grafting of two prime caterpillars. Light-grey tiles are arbitrary tiles of the tiling that do not belong to the fully leafed induced subtree. For each case, the first row shows all possible placements of tile $u$ (in blue). Below each fully leafed induced subtree, we present its possible extensions that may lead to a grafting of two prime caterpillars. When the caterpillar grafted at $u$ to $C'' \cup \{t\} \cup \{u\}$ cannot be extended into a prime caterpillar, we stop the extension attempts. We also omit graftings of two prime caterpillars that are isometric to an already identified one.}
    \label{fig:greffages de deux chenilles premières}
\end{figure}

Finally, the two retained graftings (those marked with a black dot in Figure \ref{fig:greffages de deux chenilles premières}) do belong to a P2 tiling (see Figure \ref{fig:Chenilles de Porrier}).\\

Thus, the only possible graftings of two prime caterpillars are those whose derived caterpillar is one of the caterpillars shown in Figure \ref{fig:greffages premiers HBS}.
\begin{figure}[H]
    \centering
    \begin{subfigure}{0.3\textwidth}
        \centering
        \includegraphics[width=\textwidth]{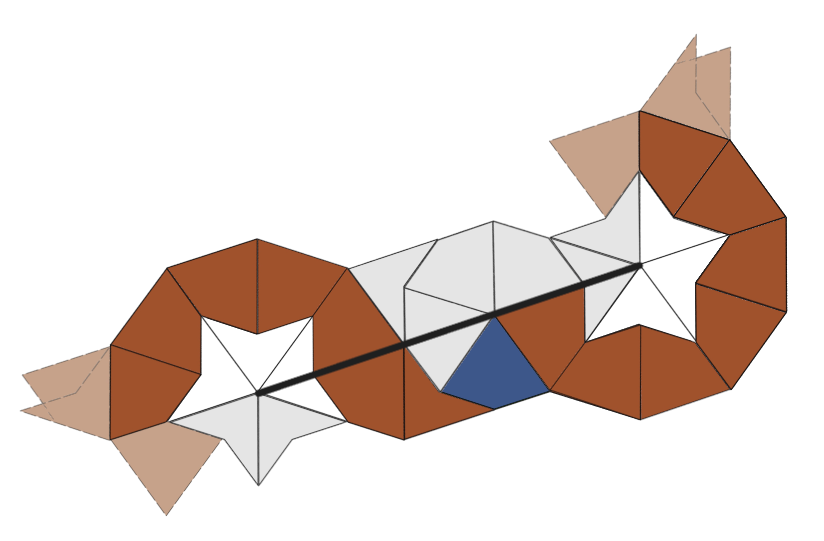}
        \caption{}
    \end{subfigure} 
    \begin{subfigure}{0.25\textwidth}
        \centering
        \includegraphics[width=0.96\textwidth]{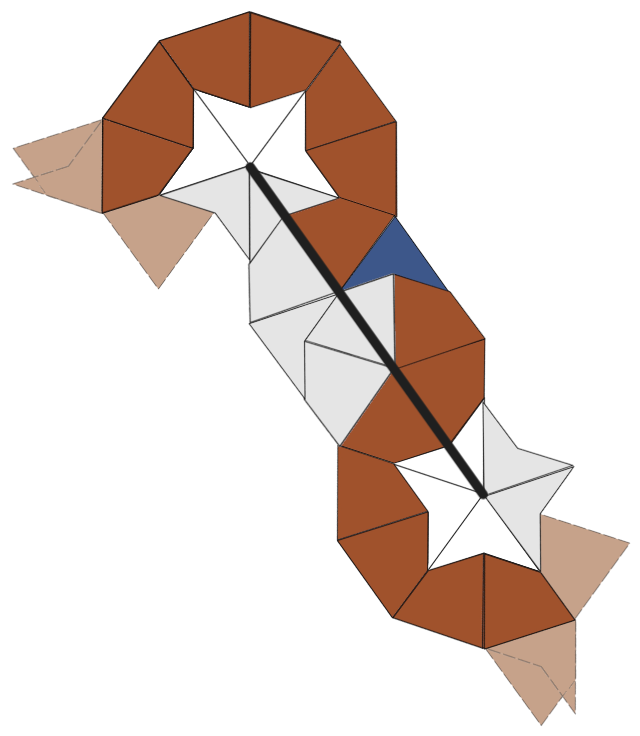}
        \caption{}
    \end{subfigure}
    \caption{The two derived paths of graftings of two prime caterpillars. Light-brown tiles are the possible degree-3 tiles for the grafting. Grey tiles are tiles that could belong to the caterpillar but not to its derived path. Brown, blue and white tiles are interpreted as before. We added a segment that joins the centers of the stars adjacent to the prime caterpillars.}
    \label{fig:greffages premiers HBS}
\end{figure}
\end{proof}
If we link all the stars of a P2 tiling as shown in Figure \ref{fig:greffages premiers HBS}, we obtain a new tiling whose tiles have these line segments as edges. This new family of tilings is called the HBS tilings (\textit{hexagons, boats and stars} tilings). For further details on these tilings, we refer the reader to \cite{porrier2023hbs}. Figure \ref{fig:HBS} illustrates a region of such a tiling, and Figure \ref{fig:protuiles HBS} shows the types of tiles in this family.

\begin{figure}[H]
    \centering
    \includegraphics[width=0.8\textwidth]{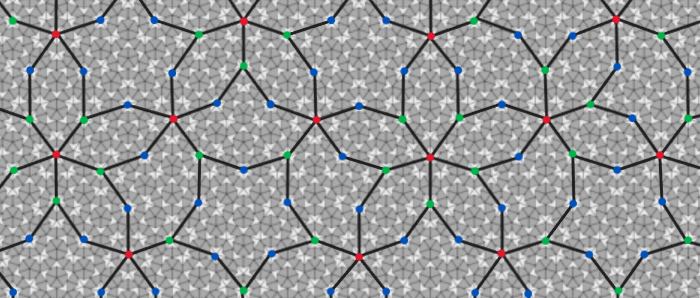}
    \caption{A region of an HBS tiling superimposed on a P2 tiling. The thick black edges are the edges of the HBS tiling and they correspond to the segments that join two stars as in Figure \ref{fig:greffages premiers HBS}. Vertices at the center of a P2 star are color-coded red, green, and blue depending on whether the star is adjacent to 0, 1, or 2 suns, respectively.}
    \label{fig:HBS}
\end{figure}

\begin{figure}[H]
    \centering
    \begin{tabular}{@{}ccccc@{}}
    \includegraphics[width=0.15\textwidth]{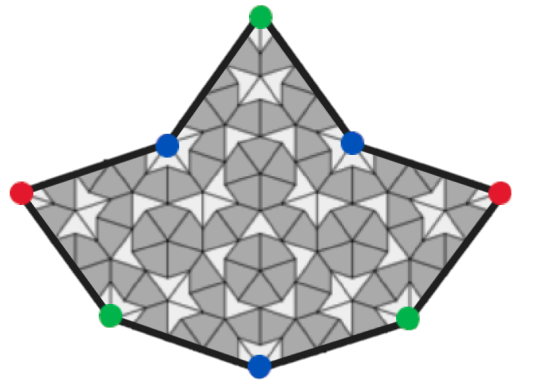} &
    \includegraphics[width=0.08\textwidth]{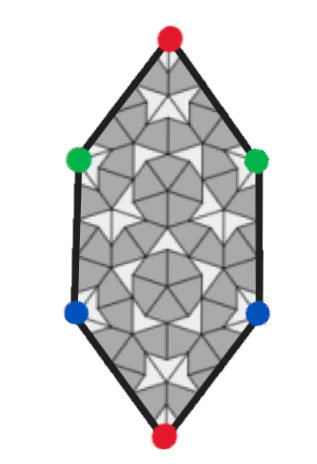} &
    \includegraphics[width=0.15\textwidth]{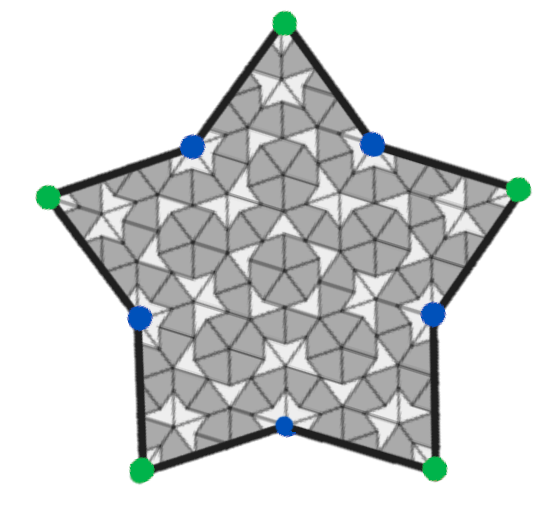} &
     \label{fig:protuiles_S1}
    \raisebox{-0.2mm}{\includegraphics[width=0.15\textwidth]{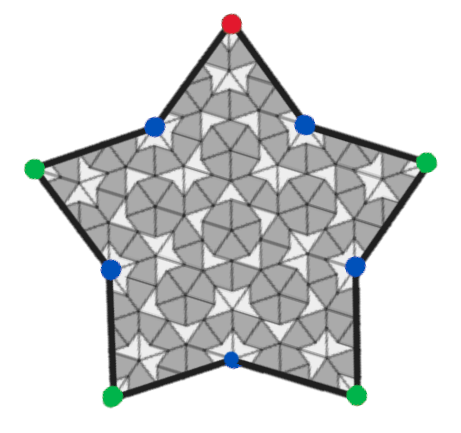}} &
    \raisebox{-0.2mm}{\includegraphics[width=0.15\textwidth]{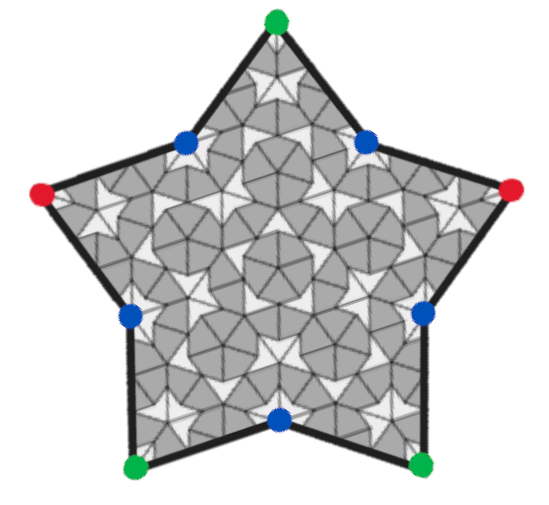}} \\
    \small{(a) Boat} &
    \small{(b) Hexagon} &
    \small{(c) Star of type 0 (S0)} &
    \small{(d) Star of type 1 (S1)} &
    \small{(e) Star of type 2 (S2)}
    \end{tabular}
    \caption{Tile types of HBS tilings. We have filled the interior of each tile with tiles or parts of tiles from P2 tilings. To match HBS tiles together, the colors of the vertices must match. An HBS star is of type 0, 1, or 2 depending on whether it has 0, 1, or 2 red vertices, respectively, on its border.}
    \label{fig:protuiles HBS}
\end{figure}

\begin{figure}[H]
    \centering
    \begin{subfigure}{0.15\textwidth}
        \centering
        \hspace*{1em}\includegraphics[width=\textwidth, trim=0cm 0.3cm 0cm 0cm, clip]{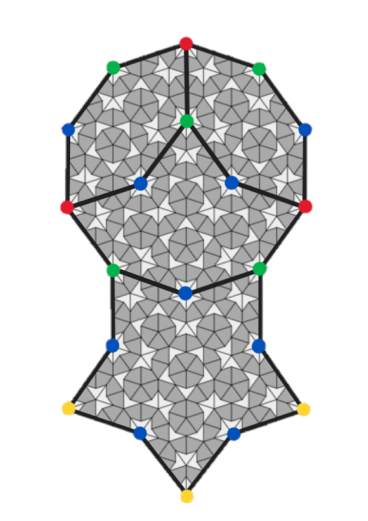}
        \captionsetup{
    width=6\textwidth, justification=justified,
    singlelinecheck=false
  }
        \caption{Kingdom of boat}
    \end{subfigure} \hspace{30pt}
    \begin{subfigure}{0.07\textwidth}
        \centering
        \hspace*{3.2em}\includegraphics[width=\textwidth, trim=0cm 0.2cm 0cm 0cm, clip]{Figures/HBS_tiles_and_kingdoms/Hexagone.PNG}
         \captionsetup{
    width=6\textwidth, justification=justified,
    singlelinecheck=false
  }
        \caption{Kingdom of hexagon}
    \end{subfigure} \hspace{60pt}
    \begin{subfigure}{0.28\textwidth}
        \centering
        \includegraphics[width=\textwidth]{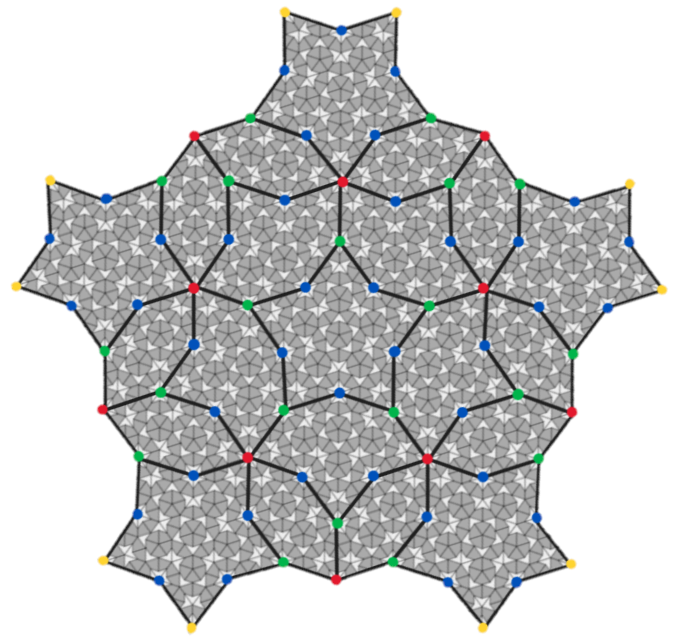}
        \caption{Kingdom of S0}
    \end{subfigure}\hfill
    \begin{subfigure}{0.41\textwidth}
        \centering
        \includegraphics[width=\textwidth]{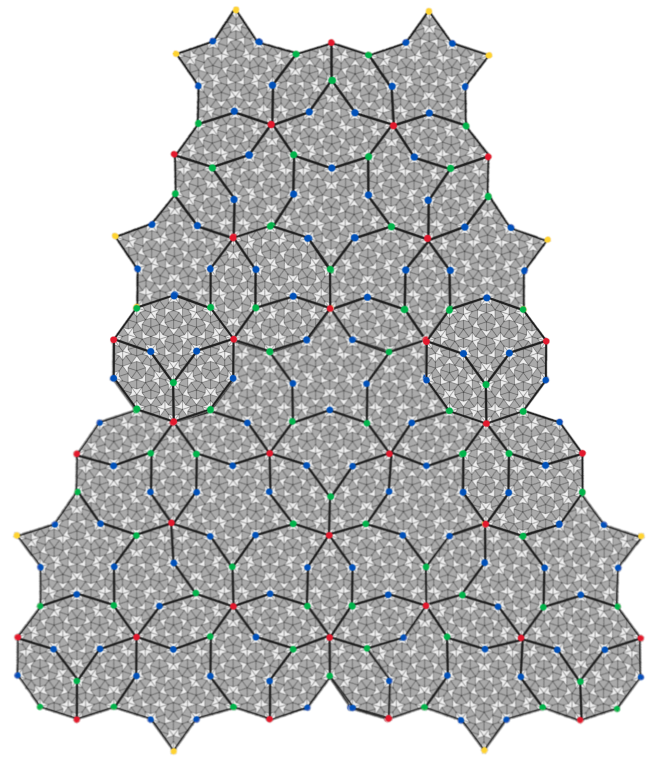}
        \caption{Kingdom of S1}
    \end{subfigure}\hspace{20pt}
    \begin{subfigure}{0.33\textwidth}
        \centering
        \includegraphics[width=\textwidth]{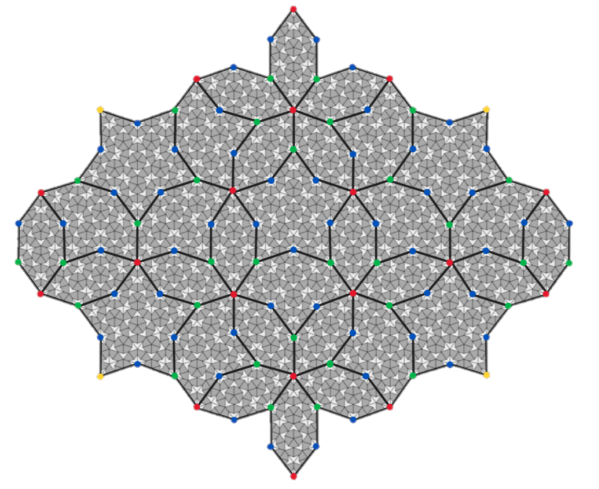}
        \caption{Kingdom of S2}
    \end{subfigure}
    \caption{The five kingdoms of tile types in HBS tilings. Yellow vertices are vertices whose actual color is undetermined. We will use this color code for the rest of the paper.}
    \label{fig:royaumes HBS}
\end{figure}

Figure \ref{fig:royaumes HBS} shows the kingdoms of those tiles. We will use these kingdoms to build large HBS patches.

\begin{definition}
The graph with vertices at the center of stars in a P2 tiling and with edges that are line segments joining them as in Figure \ref{fig:greffages premiers HBS} is called a \textbf{star-graph}.
\end{definition}

From Figure \ref{fig:greffages premiers HBS}, we can see that every prime caterpillar determines two connected edges of a star-graph. We illustrate this in Figure \ref{fig:chenilles premières - angles}. In geometry, an angle is determined by two segments that have the same origin. Since each prime caterpillar determines such two segments, we can directly define the angle determined by a prime caterpillar.

\begin{definition} \label{def:angle}
    For a given prime caterpillar $PC_k$, we call the \textbf{angle of} $PC_k$ the pair of two star-graph adjacent edges it determines, as illustrated in Figure \ref{fig:chenilles premières - angles}.
\end{definition}

For a given prime caterpillar $C$, we see that the angle $A$ of $C$ splits the tiles of $C'$ into two sets of tiles: the ones that are on one given side of the angle $A$, and the other ones. For any prime caterpillar $C$, we can see in Figure \ref{fig:chenilles premières - angles} that the internal tiles of $C''$ are all on the same side of the angle. According to this remark, we can introduce the following definition.

\begin{definition}
    For a given prime caterpillar $C$ and $A$ its angle, we say that $C$ \textbf{lies on side} $s$ of $A$, where $s$ is a given side of $A$, if the internal tiles of $C''$ are on side $s$.
\end{definition}

Since an angle admits two possible measures, we always measure the angle of a prime caterpillar on the side on which the prime caterpillar lies. All angles are measured in radians.

\begin{figure}[H]
    \centering
    \begin{subfigure}{0.16\textwidth}
        \centering
        \includegraphics[width=\textwidth]{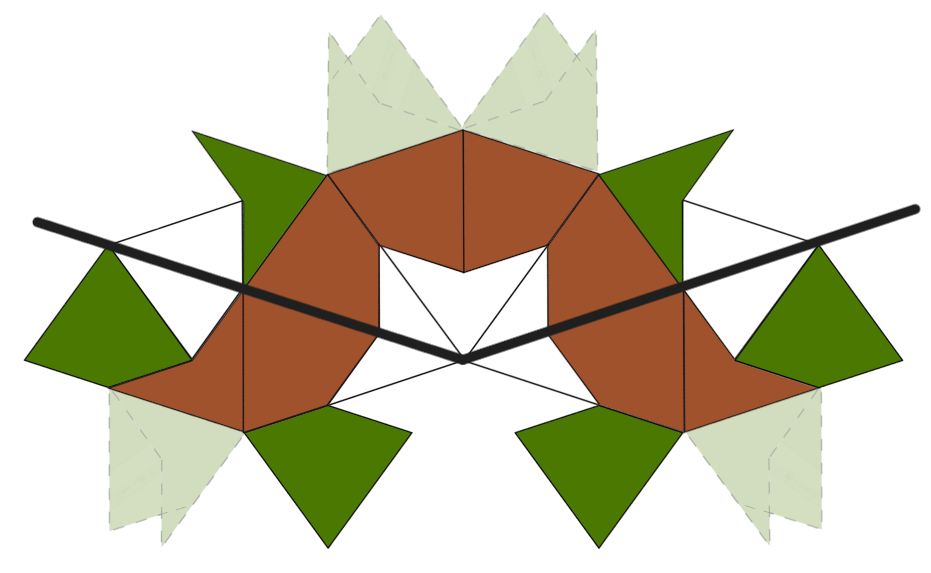}
        \caption{$PC_1$}
        \label{fig:CP1 - angle}
    \end{subfigure} \hfill
    \begin{subfigure}{0.15\textwidth}
    \includegraphics[width=\textwidth, trim=0cm 0.4cm 0cm 0cm, clip]{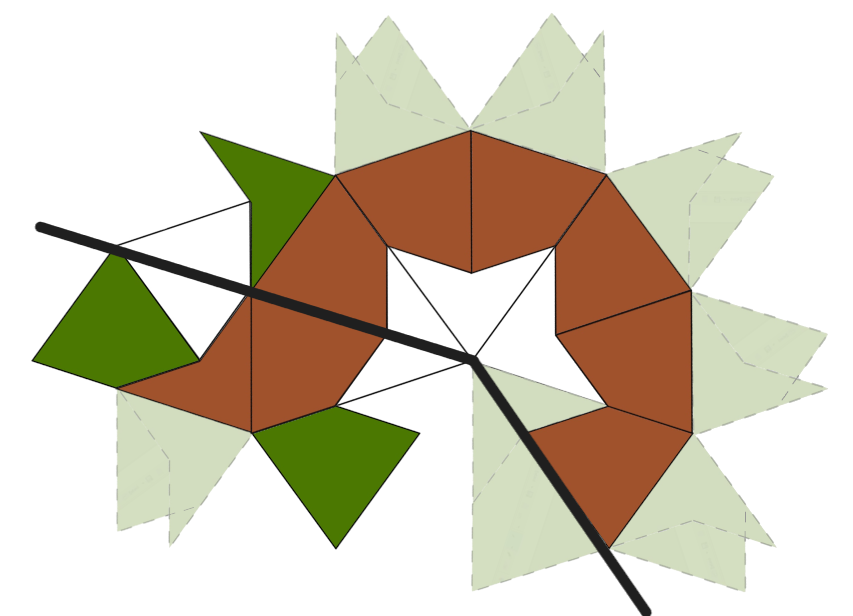}
    \caption{$PC_2$}
    \label{fig:CP2 - angle}
    \end{subfigure} \hfill
    \begin{subfigure}{0.16\textwidth}
    \includegraphics[width=\textwidth, trim=0cm 0cm 0cm 0cm, clip]{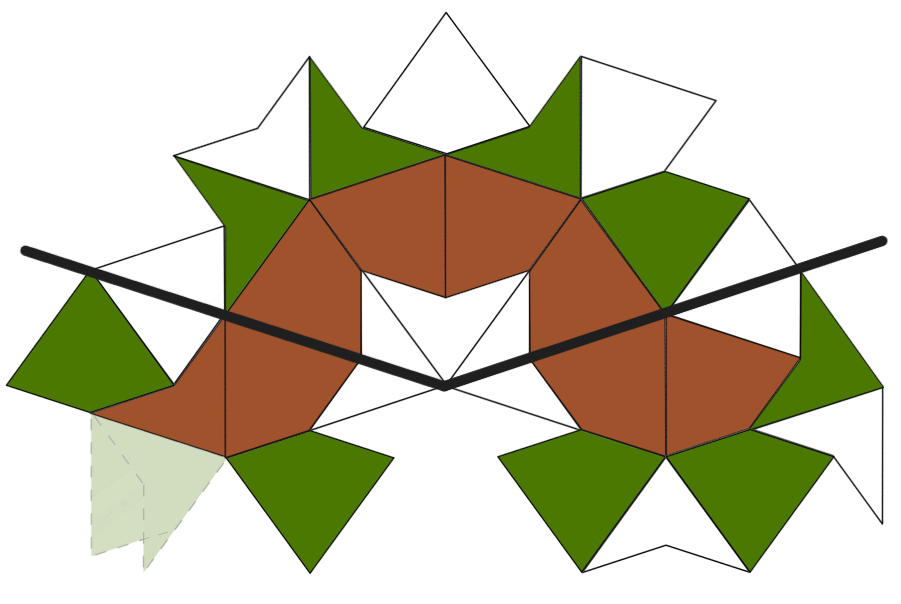}
    \caption{$PC_3$}
    \label{fig:CP3 - angle}
    \end{subfigure} \hfill
    \begin{subfigure}{0.15\textwidth}
    \includegraphics[width=\textwidth, trim=0cm 0.5cm 0cm 0cm, clip]{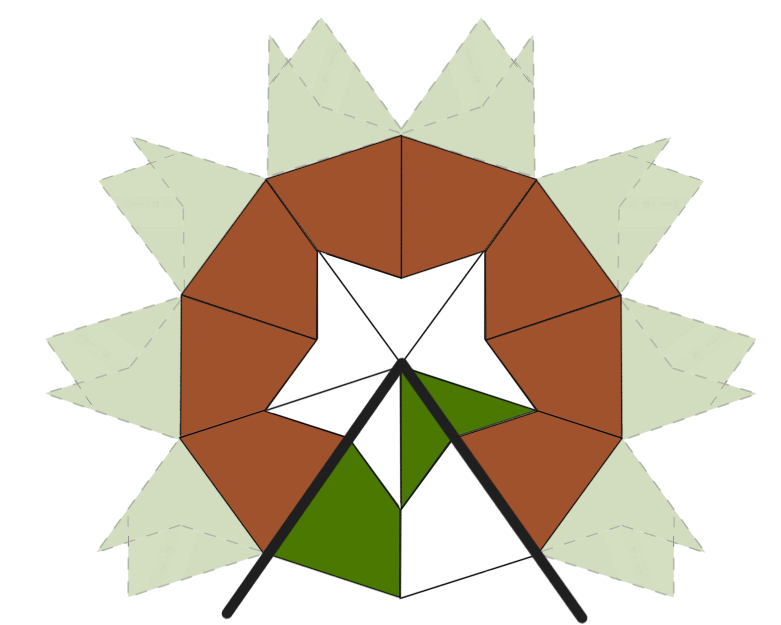}
    \caption{$PC_4$}
    \label{fig:CP4 - angle}
    \end{subfigure} \hfill
    \begin{subfigure}{0.16\textwidth}
    \includegraphics[width=\textwidth]{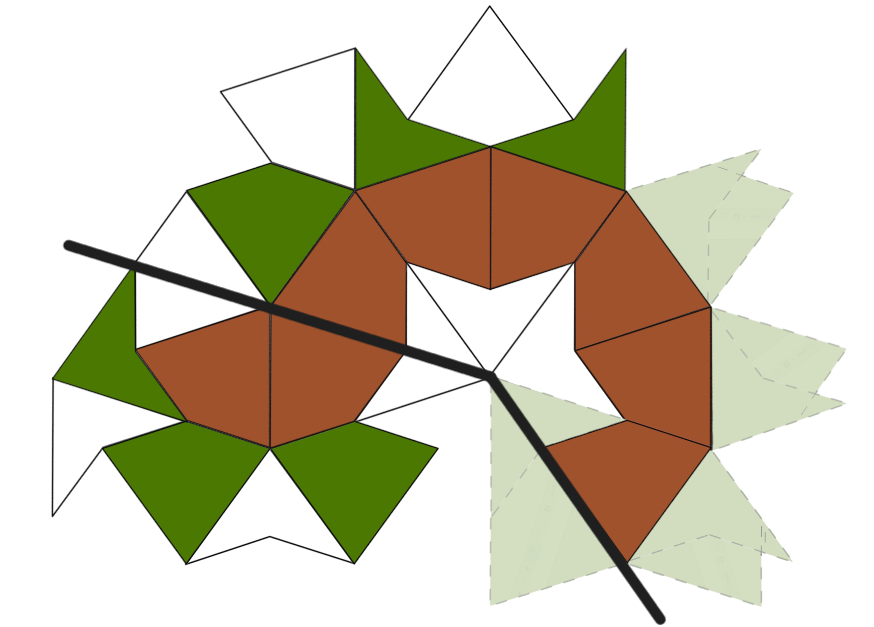}
    \caption{$PC_5$}
    \label{fig:CP5 - angle}
    \end{subfigure} \hfill
    \begin{subfigure}{0.16\textwidth}
    \includegraphics[width=\textwidth, trim=0.5cm 0cm 0cm 0cm, clip]{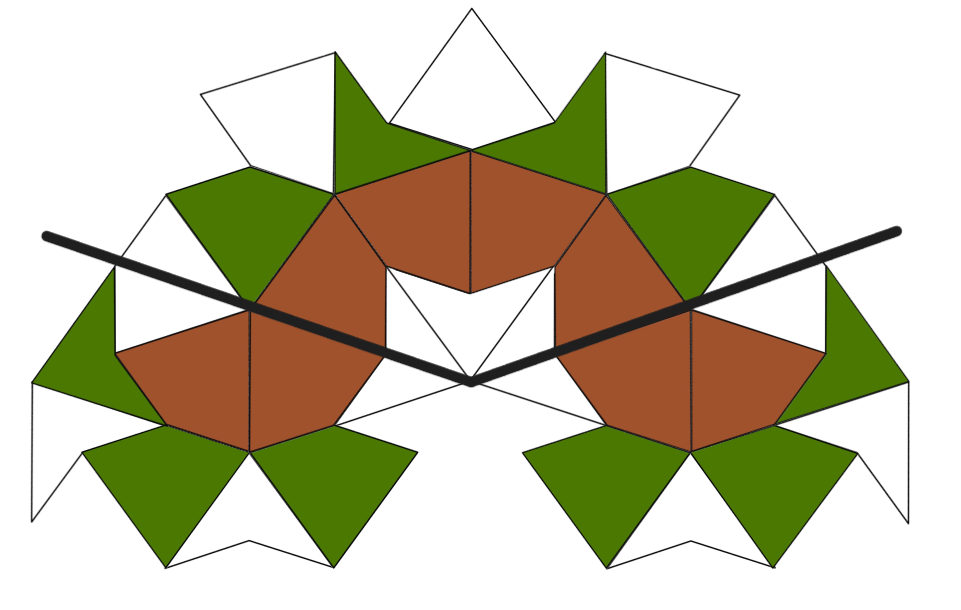}
    \caption{$PC_6$}
    \label{fig:CP6 - angle}
    \end{subfigure} \hfill
    \caption{The six prime caterpillars decorated with the angle they determine}
    \label{fig:chenilles premières - angles}
\end{figure}

\begin{lemma}\label{prop:angle}
The measures of the angles of the prime caterpillars are given by the following claims:
\begin{enumerate}
    \item $PC_1$, $PC_3$ and $PC_6$ determine an angle of $\frac{4\pi}{5}$;
    \item $PC_2$ and $PC_5$ determine an angle of $\frac{6\pi}{5}$;
    \item $PC_4$ determines an angle of $\frac{8\pi}{5}$.
\end{enumerate}
\end{lemma}

\begin{proof}
The boundary of a star kingdom in a P2 tiling is a regular decagon. Figure \ref{fig:chenilles premières - angles} then completes the proof.
\end{proof}

\begin{definition}
    If $\mathcal{C}$ is a fully leafed caterpillar obtained by grafting prime caterpillars, the \textbf{HBS chain of} $\mathcal{C}$ is the path (in the star-graph) obtained by concatenating the star-graph edges determined by the prime caterpillars in $\mathcal{C}$.
\end{definition}

Let $\mathcal{C}$ be a fully leafed caterpillar in a $P2$ tiling $P$ obtained by grafting prime caterpillars, let $h$ be the HBS chain of $\mathcal{C}$ and let $H$ be any bi-infinite path extension of $h$ in the star-graph of $P$ that does not self-intersect. Then $H$ splits $\mathbb{R}^2$ in two regions. Thus, we can easily speak of one side or the other of $H$, whereas there is a priori an ambiguity to identify one side or the other of $h$, since $h$ can be a finite chain in a star-graph, which is an infinite graph. But since $H$ is an extension of $h$ that does not self-intersect, no matter how $H$ is constructed from $h$, for any prime caterpillar $C$ in $\mathcal{C}$, the side of $C$ determined from its angle is always included in the side of $H$ where the internal tiles of $C''$ lie. Thus, the choice of $H$ does not matter to determine on which side of $h$ the prime caterpillar $C$ is, so we can present the following definition without ambiguity.

\begin{definition} \label{def: côté}
    Let $\mathcal{C}$ be a fully leafed caterpillar obtained by grafting prime caterpillars, and let $h$ the HBS chain of $\mathcal{C}$. For $C$ a prime caterpillar in $\mathcal{C}$, we say that $C$ \textbf{lies on side} $s$ of $h$ if the internal tiles of $C''$ are on side $s$ of $H$, where $s$ is a given side of a given bi-infinite chain extension $H$ of $h$ that does not self-intersect, in a given star-graph.
\end{definition}

To prevent ambiguity, we had to add details to the previous definitions which might make them more difficult to understand. But in fact, these definitions are quite intuitive. Example \ref{exemple} illustrates the preceding notions we have just introduced.

\begin{example} \label{exemple}
    Let $\mathcal{C}$ be the fully leafed caterpillar presented in Figure \ref{fig:0}. This caterpillar is obtained by grafting 7 prime caterpillars. We denote them $C_{-3}$, $C_{-2}$, $\ldots$, $C_2$ and $C_3$ respectively in the order from left to right in which they appear in Figure \ref{fig:0}. The HBS chain of $\mathcal{C}$ is the purple chain. The prime caterpillars $C_{-3}$, $C_{-1}$, $C_1$ and $C_3$ lie on the same side of the HBS chain of $\mathcal{C}$, while the prime caterpillars $C_{-2}$, $C_0$ and $C_2$ lie on the other side. The prime caterpillar $C_0$ determines an angle of $\frac{8\pi}{5}$, the prime caterpillars $C_{-2}$, $C_{-1}$, $C_1$ and $C_2$ determine an angle of $\frac{6\pi}{5}$ and the prime caterpillars $C_{-3}$ and $C_3$ determine an angle of $\frac{4\pi}{5}$.

    \begin{figure}[H]
    \centering
    \includegraphics[width=0.4\textwidth]{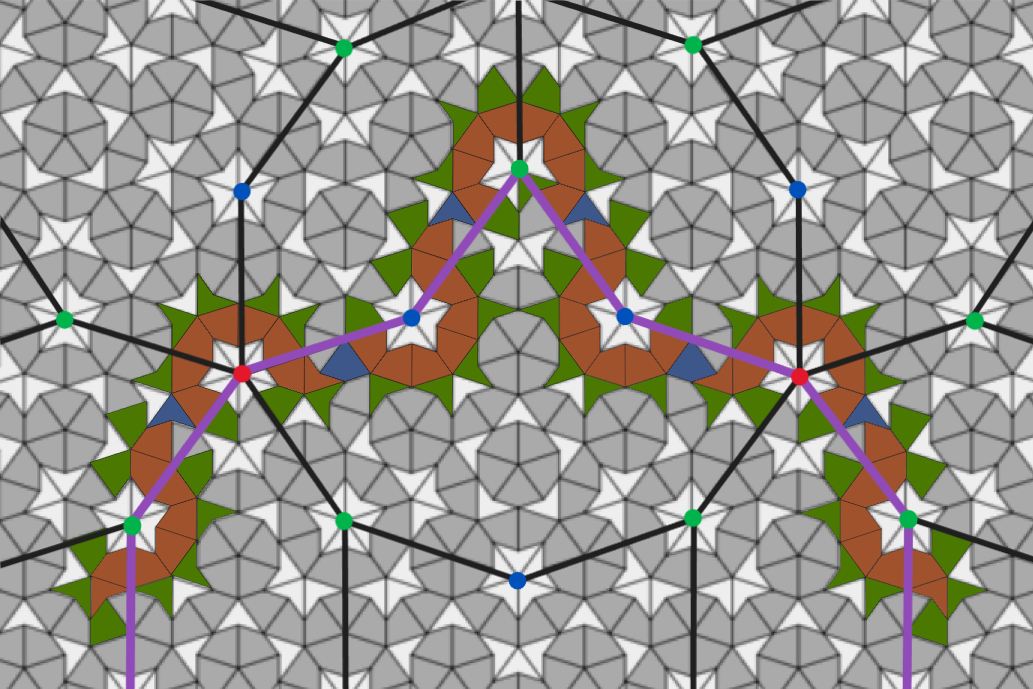}
    \caption{A fully leafed caterpillar embedded in an HBS tiling. The HBS chain of the caterpillar is shown in purple.}
    \label{fig:0}
\end{figure}
\end{example}

We now have a corollary of Proposition \ref{prop: the two graftings} that will be useful later to simplify representations of large fully leafed caterpillars.

\begin{corollary} \label{prop:serpente} 
Let $\mathcal{C}=\Diamond_{i=n}^m C_i$ be a fully leafed caterpillar, with $n,m\in \mathbb{Z}$ and $m>n$, where each $C_k$ for $k\in\{n,\dots,m\}$ is a prime caterpillar. Then for every $j\in\{n,\dots,m-1\}$, $C_j$ and $C_{j+1}$ lie on opposite sides of the HBS chain of $\mathcal{C}$.
\end{corollary}
 
\begin{proof}
This follows directly from Figure~\ref{fig:greffages premiers HBS}.
\end{proof}

\begin{lemma}\label{lem:4 pi sur 5}
Let $\mathcal{C}=\Diamond_{i=n}^m C_i$ be a fully leafed caterpillar, where $n,m\in \mathbb{Z}$, $m>n+1$ and for every $k\in\{n,\ldots,m\}$, $C_k$ is a prime caterpillar. If $j\in\{n,\ldots, m\}$ is such that $C_j$ determines an angle of $\frac{4\pi}{5}$, then the leaves of $C_j$ that are adjacent to the leaves of $C_j'$, (i.e. the tiles where a grafting of prime caterpillars can occur), both lie on the side of the HBS chain where $C_j$ does not lie.
\end{lemma}

\begin{proof}
This follows from Figure~\ref{fig:chenilles premières - angles} and Lemma \ref{prop:angle}.
\end{proof}

\begin{lemma} \label{prop:4 pi sur 5}
Let $\mathcal{C}=\Diamond_{i=n}^m C_i$ be a fully leafed caterpillar, where $n,m\in \mathbb{Z}$, $m>n+1$ and where for every $k\in\{n,\ldots,m\}$, $C_k$ is a prime caterpillar. Then for every $j\in\{n,\ldots,m-1\}$, $C_j$ and $C_{j+1}$ cannot both determine an angle of $\frac{4\pi}{5}$.
\end{lemma}
\begin{proof}
Let $\mathcal{C}=\Diamond_{i=n}^m C_i$ be as in the hypothesis. Let $j\in\{n,\ldots,m-1\}$ be such that $C_j$ determines an angle of $\frac{4\pi}{5}$. From Lemma~\ref{lem:4 pi sur 5}, the degree-2 tile for $\mathcal{C}$ through which $C_j$ is grafted onto $C_{j+1}$ lies on the same side of the HBS chain as the side where the angle of $C_{j+1}$ is formed. Hence, again by Lemma~\ref{lem:4 pi sur 5}, $C_{j+1}$ cannot form an angle of $\frac{4\pi}{5}$.
\end{proof}

Figure~\ref{fig:Voisinages HBS} shows vertex neighborhoods in the star graphs.

\begin{figure}[H]
    \centering
    \begin{tabular}{@{}ccc@{}}
    \includegraphics[width=0.27\textwidth, trim=0cm 4cm 0cm 0cm, clip]{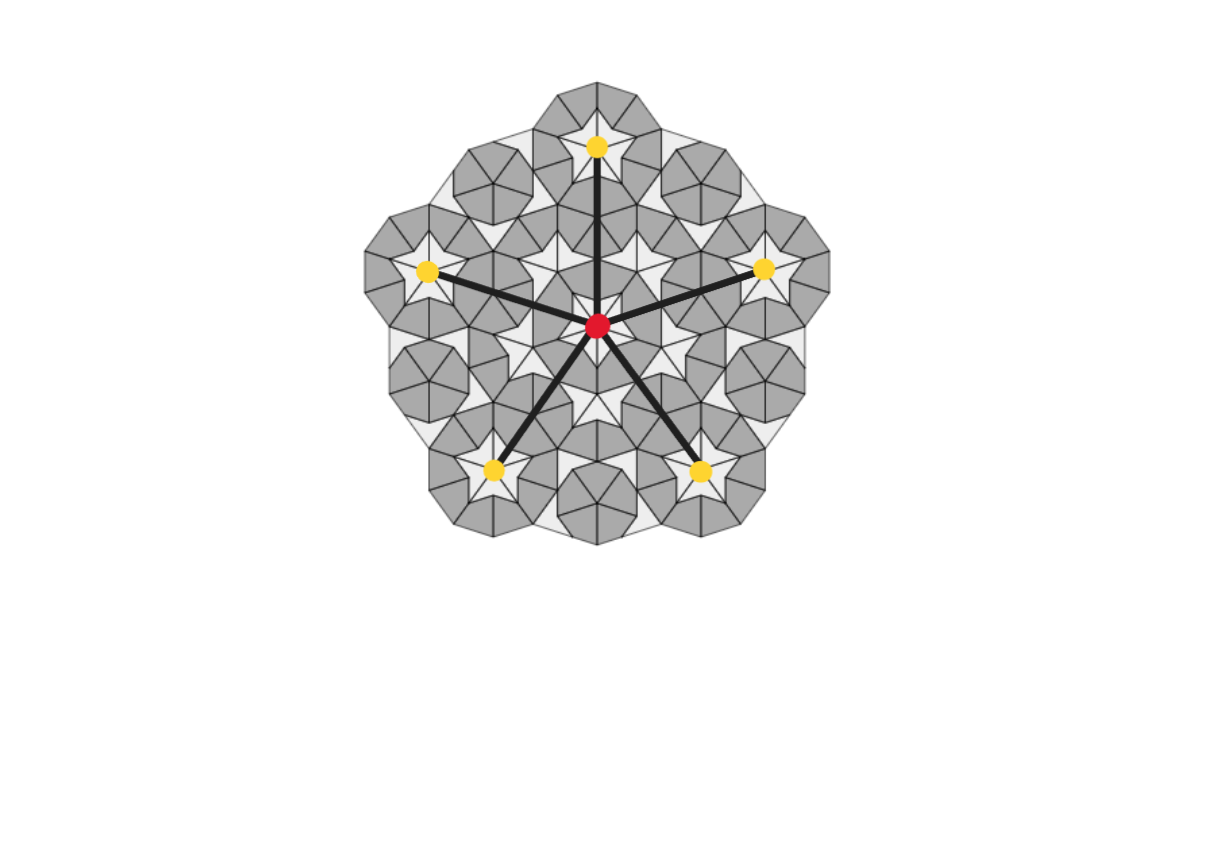} &
    \includegraphics[width=0.17\textwidth]{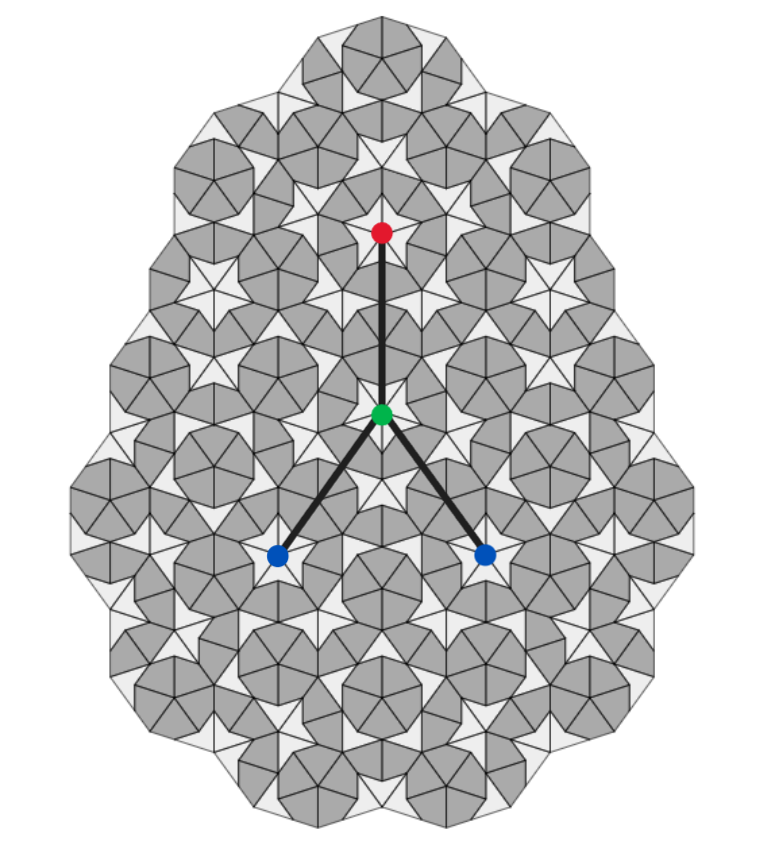} &
    \includegraphics[width=0.15\textwidth]{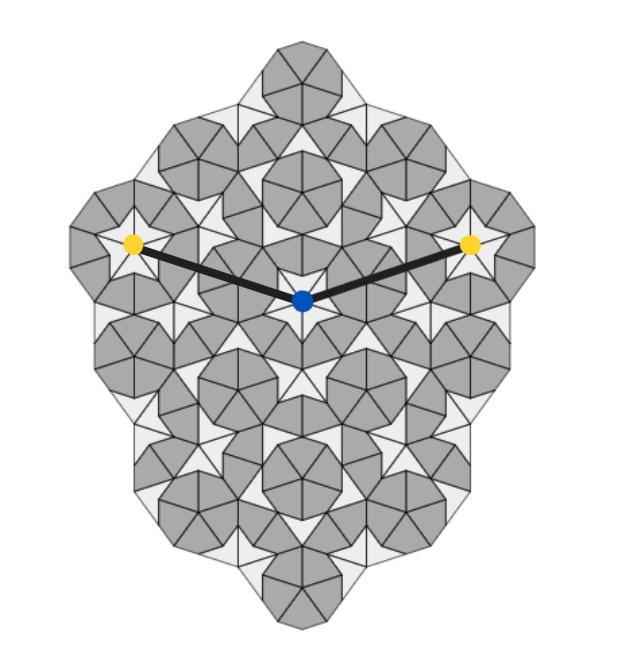}\\
    \small \makecell{(a) Neighborhood of any red\\ vertex in any star graph} &
    \small \makecell{(b) Neighborhood of any green\\ vertex in any star graph} &
    \small \makecell{(c) Neighborhood of any blue\\ vertex in any star graph}
    \end{tabular}
    \caption{Neighborhoods of the red, green, and blue vertices in any star graph. We recall that the yellow vertices are vertices with an undetermined color (between the colors red, green and blue). In case (a), note that yellow vertices cannot be red. In case (c), yellow vertices cannot be blue.}
    \label{fig:Voisinages HBS}
\end{figure}

\begin{proposition} \label{prop:CP1}
The prime caterpillar $PC_1$ does not belong to any bi-infinite fully leafed caterpillar.
\end{proposition}
\begin{proof}
First, from Figures \ref{fig:chenilles premières - angles} and \ref{fig:Voisinages HBS}, the star adjacent to a $PC_1$ caterpillar must have a red vertex at its center. We also know from Lemma \ref{prop:angle} that $PC_1$ determines an angle of $\frac{4\pi}{5}$.

Let $C$ be a prime caterpillar of type $PC_1$. We consider the grafting $C_2 \diamond C \diamond C_1$, where $C_2$ and $C_1$ are prime caterpillars. Suppose that the star adjacent to $C_1$ has a green vertex at its center. From Lemma \ref{prop:4 pi sur 5} and Figure \ref{fig:Voisinages HBS}, the caterpillar $C_1$ must form an angle of $\frac{6\pi}{5}$, and therefore, by Figures \ref{fig:greffages premiers HBS} and \ref{fig:chenilles premières - angles} and Lemma \ref{prop:angle}, $C_1$ is of type $PC_2$. From Figure \ref{fig:Voisinages HBS}, this prime caterpillar should be grafted to a prime caterpillar adjacent to a star with blue center, and then it forms an angle of $\frac{4\pi}{5}$. However, by Lemma \ref{lem:4 pi sur 5}, this is not possible.

Now suppose that the star adjacent to $C_1$ has a blue center. Since $C$ is adjacent to a star of red center, from Figure \ref{fig:Voisinages HBS}, $C_2$ is adjacent to a star with green center or with blue center. In the first case, from the previous paragraph, this does not allow $C$ to be extended in a bi-infinite fully leafed caterpillar. In the second case, from Figure \ref{fig:Voisinages HBS}, $C_2$ would determine an angle of $\frac{4\pi}{5}$. But that is impossible, from Lemma \ref{prop:4 pi sur 5}, since $C$ also determines an angle of $\frac{4\pi}{5}$.

\end{proof}

\subsection{Sea caterpillars}
To facilitate the study of bi-infinite fully leafed caterpillars, we introduce a new class of caterpillars that we call \textbf{sea caterpillars}. All sea caterpillars are shown in Figure~\ref{fig:chenilles marines}. From Figure \ref{fig:royaumes HBS}, it can be readily verified that every bi-infinite fully leafed caterpillar can be obtained by successive graftings of sea caterpillars. Thus, as sea caterpillars are larger than prime caterpillars, using them may facilitate the identification of bi-infinite fully leafed caterpillars.

\begin{figure}[H]
    \centering
    \begin{subfigure}{0.27\textwidth}
        \centering
        \includegraphics[width=\textwidth, trim=0cm 0.5cm 0cm 0.3cm, clip]{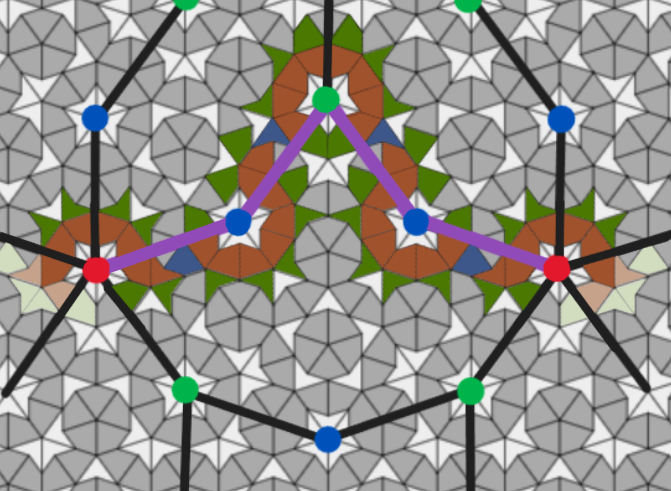}
        \caption{Sail}
    \end{subfigure}
    \begin{subfigure}{0.27\textwidth}
        \centering
        \includegraphics[width=\textwidth, trim=0cm 0cm 0cm 0cm, clip]{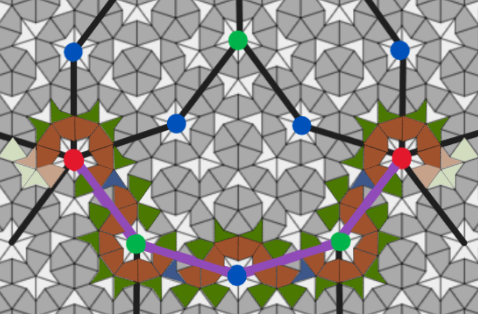}
        \caption{Hull}
    \end{subfigure} 
    \begin{subfigure}{0.23\textwidth}
        \centering
        \includegraphics[width=\textwidth, trim=0cm 0.4cm 0cm 0.1cm, clip]{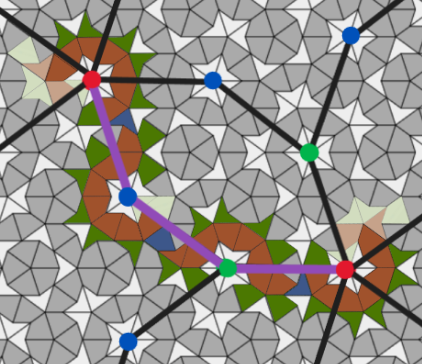}
        \caption{Dock 1}
    \end{subfigure}  
    \begin{subfigure}{0.23\textwidth}
        \centering
        \includegraphics[width=\textwidth, trim=0cm 0.9cm 0cm 0.05cm, clip]{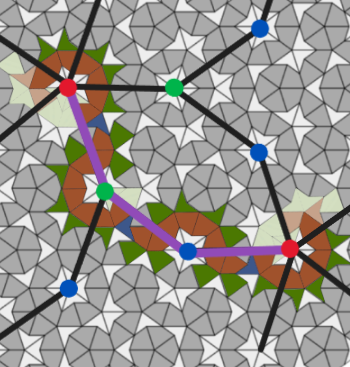}
        \caption{Dock 2}
    \end{subfigure}
    \begin{subfigure}{0.23\textwidth}
        \centering
        \includegraphics[width=\textwidth, trim=0cm 0.2cm 0cm 0cm, clip]{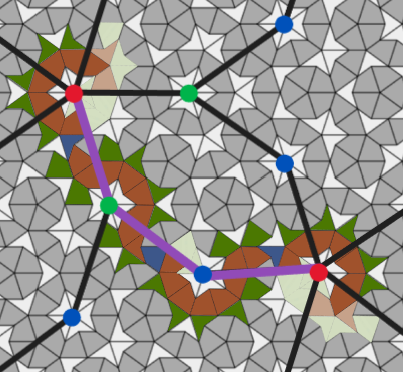}
        \caption{Dock 3}
    \end{subfigure}
    \begin{subfigure}{0.23\textwidth}
        \centering
        \includegraphics[width=\textwidth]{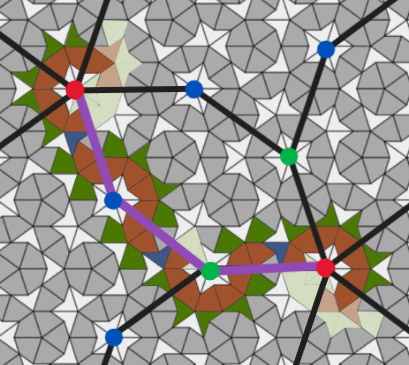}
        \caption{Dock 4}
    \end{subfigure} 
    \begin{subfigure}{0.29\textwidth}
        \centering
        \includegraphics[width=\textwidth, trim=0cm 0cm 0cm 0cm, clip]{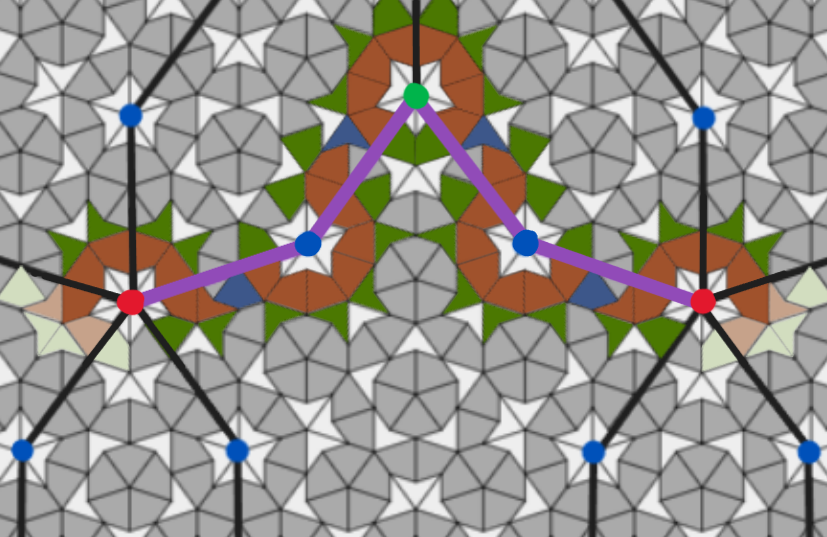}
        \caption{Cape 1}
    \end{subfigure} 
    \begin{subfigure}{0.40\textwidth}
        \centering
        \includegraphics[width=\textwidth, trim=0cm 0.7cm 0cm 0cm, clip]{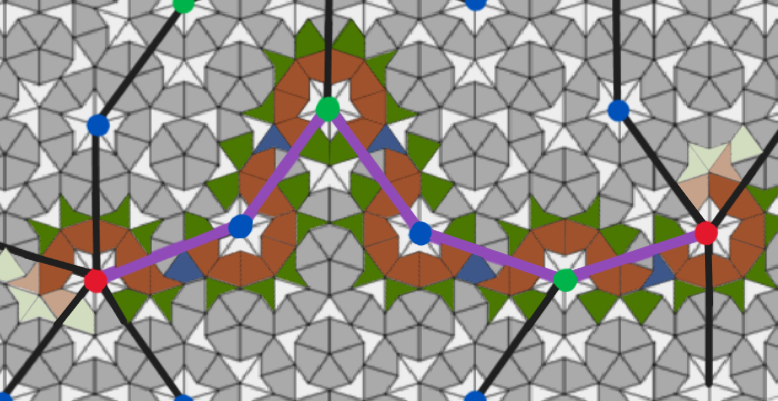}
        \caption{Cape 2}
    \end{subfigure} 
    \begin{subfigure}{0.40\textwidth}
        \centering
        \includegraphics[width=\textwidth, trim=0cm 0.4cm 0cm 0cm, clip]{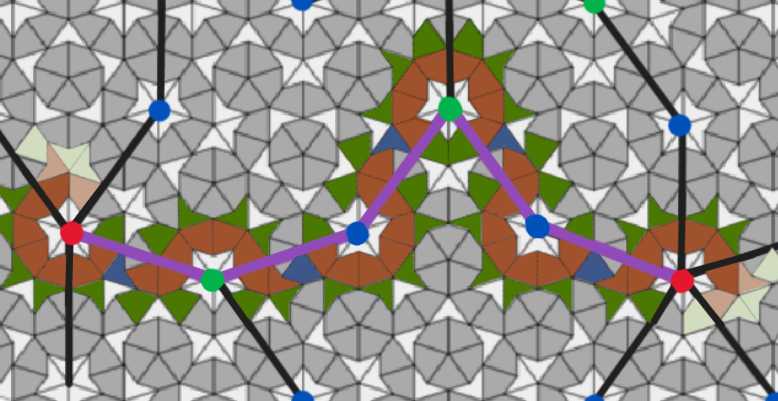}
        \caption{Cape 3}
    \end{subfigure} 
    \begin{subfigure}{0.445\textwidth}
        \centering
        \includegraphics[width=\textwidth, trim=0cm 0cm 0cm 0cm, clip]{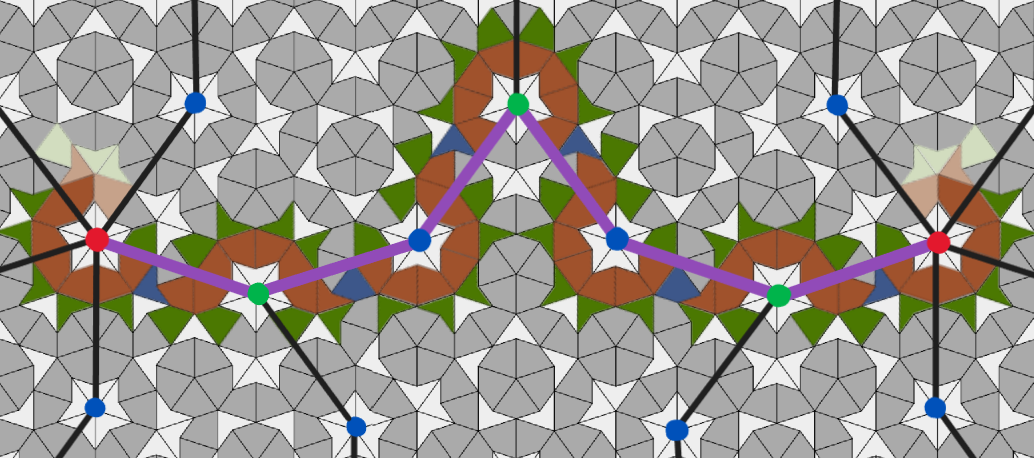}
        \caption{Cape 4}
    \end{subfigure}
    \caption{Sea caterpillars (up to rotation). There are two choices of prime caterpillars (forming different angles) at each end of a sea caterpillar. The choice depends on the degree-3 chosen tile (in light brown). Sail and hull are constructed on a boat, docks are constructed on a hexagon and capes are constructed on a star (of an HBS tiling).}
    \label{fig:chenilles marines}
\end{figure}

For the rest of the paper, we will only use the purple chains that appear in Figure \ref{fig:chenilles marines} to identify sea caterpillars. We distinguish reflections between sea caterpillars because, ultimately, our goal is to characterize all bi-infinite fully leafed caterpillars by words. Notice that no prime caterpillar determines an angle of $\frac{2\pi}{5}$, so from an angle of $\frac{8\pi}{5}$ in an HBS chain, the side on which the prime caterpillar associated to this angle lies is determined, and from Lemma \ref{prop:angle}, this prime caterpillar is necessarily $PC_4$. By Corollary \ref{prop:serpente}, the rest of the caterpillar along the HBS chain is then completely determined. It is then easy to reconstruct a caterpillar (up to the choice of the prime caterpillars at the two ends of the caterpillar) from a given HBS chain containing an angle of $\frac{8\pi}{5}$.

\section{Beginnings of the research of all bi-infinite fully leafed caterpillars}

The proof of Theorem~\ref{thm 1} shows how to construct a bi-infinite fully leafed caterpillar in P2-graphs. It is then natural to ask whether other such bi-infinite caterpillars exist. Theorem~\ref{thm:deuxième chenille infinie} answers this question in the affirmative, thereby refuting Conjecture~2 in \cite{porrier2023leaf}.

\subsection{Elimination of sea caterpillars}

Our goal is to find all bi-infinite fully leafed caterpillars in P2-graphs. The approach adopted here consists of  eliminating graftings of prime and sea caterpillars until the set of admissible graftings becomes sufficiently restricted to yield all possible constructions of bi-infinite fully leafed caterpillars. We present here the initial steps in identifying such forbidden graftings.

\begin{lemma} \label{BRB hexagone}
Consider the HBS chain of a dock~2 (shown in purple in Figure~\ref{fig:chenilles marines}~(d)), and let $R$ be the red vertex of this chain adjacent in the star-graph to its blue vertex. If the dock~2 extends into a bi-infinite fully leafed caterpillar, then the prime caterpillar adjacent to the star of center $R$ must be of type~$PC_2$.
\end{lemma}

\begin{proof}
From Figure~\ref{fig:chenilles marines}~(d), the prime caterpillar adjacent to the star of center $R$ is either of type $PC_2$ or $PC_4$. Suppose it is $PC_4$ and we attempt to extend the dock 2 into a bi-infinite fully leafed caterpillar $\mathcal{C}$. On the same hexagon supporting the dock~2, the extension will then form an angle of $\frac{4\pi}{5}$ at the blue vertex not belonging to the HBS chain of the initial dock 2. Since no prime caterpillar forms an angle of $\frac{2\pi}{5}$, the extension would form an angle of $\frac{6\pi}{5}$ at the green vertex not already belonging to the HBS chain of the initial dock 2 (the one adjacent in the star-graph to the blue vertex we just considered). This configuration then creates a cycle, implying that $\mathcal{C}$ is not a tree.
\end{proof}

\begin{lemma}\label{a}
    If a sail is extended into a bi-infinite fully leafed caterpillar, then the red vertices of the HBS chain of this sail are the center of a star adjacent to a prime caterpillar $PC_2$.
\end{lemma}
\begin{proof}
    From Figure~\ref{fig:chenilles marines}, the red vertices of an HBS chain of a sail extended into a bi-infinite fully leafed caterpillar are the center of a star adjacent to a prime caterpillar $PC_1$ or $PC_2$. From Proposition~\ref{prop:CP1}, the conclusion follows.
\end{proof}

\begin{lemma}\label{b}
    If a hull is extended into a bi-infinite fully leafed caterpillar, the red vertices of the HBS chain of the hull are the center of a star adjacent to a prime caterpillar $PC_4$.
\end{lemma}
\begin{proof}
    Let $C$ be a hull. From Figure~\ref{fig:chenilles marines}, the red vertices of the HBS chain of $C$ are the center of a star adjacent to a prime caterpillar $PC_4$ or $PC_2$. Suppose $PC_2$ is chosen for one of these caterpillars and we attempt to extend it. Let $C_1$ be the prime caterpillar adjacent to the star centered at the yellow vertex in Figure~\ref{fig:prolongements de coque avec CP2}~(a). Clearly, by Lemma~\ref{lem:4 pi sur 5}, the yellow vertex cannot be blue. Hence, $C_1$ is adjacent to a star of green center and forms an angle of $\frac{6\pi}{5}$. Figure~\ref{fig:prolongements de coque avec CP2}~(b) shows a region of the kingdom of this configuration. Again, by Lemma~\ref{lem:4 pi sur 5}, it is impossible to extend the caterpillar into a bi-infinite fully leafed caterpillar.

    \end{proof}

\begin{figure}[H]
    \centering
    \begin{subfigure}{0.23\textwidth}
        \centering
        \includegraphics[width=\textwidth]{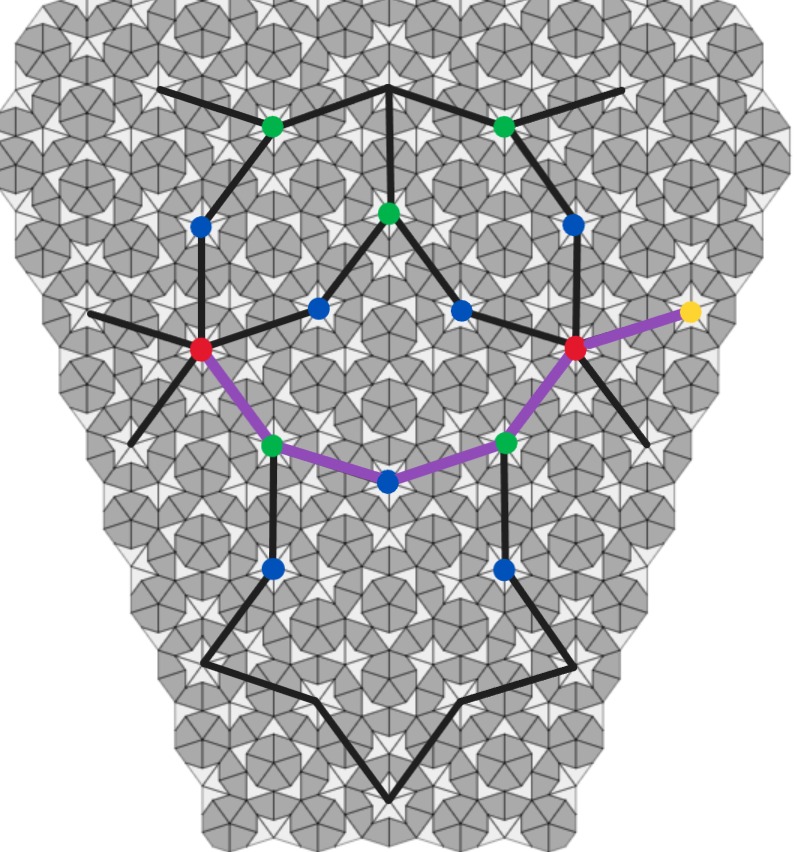}
        \caption{}
    \end{subfigure}
    \begin{subfigure}{0.23\textwidth}
        \centering
        \includegraphics[width=\textwidth]{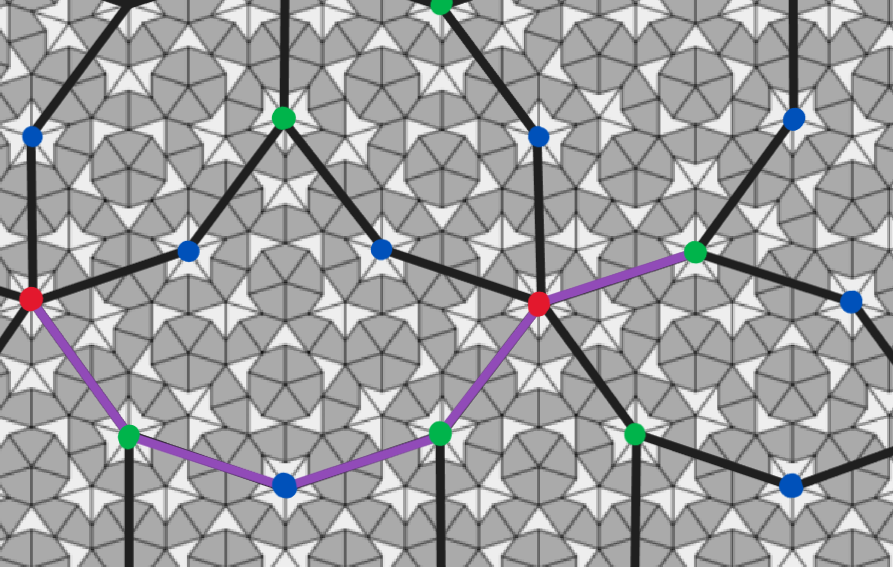}
        \caption{}
    \end{subfigure}
    \caption{Attempted extension of a hull where a prime caterpillar adjacent to a star of red center is $PC_2$}
    \label{fig:prolongements de coque avec CP2}
\end{figure}

\begin{lemma}
    Every sail is equivalent to the hull on the same boat, in the sense that in a fully leafed caterpillar $\mathcal{C}$, these two sea caterpillars can be interchanged without modifying the rest of the construction of $\mathcal{C}$.
\end{lemma}
\begin{proof}
This follows directly from Lemmas \ref{a} and \ref{b}.
\end{proof}

In the figures of this section, whenever we can use both the hull and the sail of the same boat, these sea caterpillars being equivalent, we will always use the hull to avoid overloading the figures unnecessarily (while considering both equivalent cases each time).

\begin{proposition} \label{prop:élimination de cap 2 et cap 3}
No bi-infinite fully leafed caterpillar contains a cape 2 or a cape 3.
\end{proposition}
\begin{proof}
It is enough to examine the possible extensions of a cape 2 and a cape 3, and observe that none of them can be bi-infinite. See Appendix A.
\end{proof}

\begin{proposition}\label{c}
Any cape 4 sea caterpillar built on a S1 star does not belong to any bi-infinite fully leafed caterpillar.
\end{proposition}
\begin{proof}

It is enough to examine the possible extensions of a cape 4, and observe that none of them can be bi-infinite. See Appendix B.
\end{proof}

\begin{corollary}
A bi-infinite fully leafed caterpillar never contains a cape built on a S1 star (see Figure \ref{fig:protuiles HBS}).
\end{corollary}
\begin{proof}
A cape 1 is always built on a S2 star. Next, from Proposition \ref{c}, if a bi-infinite fully leafed caterpillar contains a sea caterpillar cape 3, it is built on a S0 star. Finally, cape 2 and cape 3 sea caterpillars cannot belong to a bi-infinite fully leafed caterpillar (from Proposition~\ref{prop:élimination de cap 2 et cap 3}).
\end{proof}

\begin{theorem}\label{thm:deuxième chenille infinie}
There exists a bi-infinite fully leafed caterpillar containing a cape 4.
\end{theorem}
\begin{proof}
We first define two new caterpillars, a \textit{Super cape 1} and a \textit{Super cape 4}, illustrated in Figure~\ref{Super capes}. They are obtained by grafting two prime caterpillars (one on each end of the sea caterpillar) onto a cape 1 (resp. cape 4) caterpillar.

\begin{figure}[H]
\centering
\begin{subfigure}{0.4\textwidth}
\centering
\includegraphics[width=\textwidth, trim=0cm 2cm 0cm 0cm, clip]{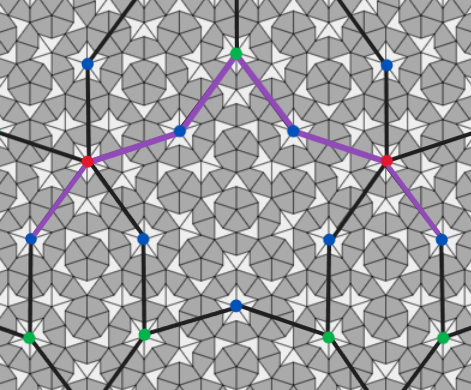}
\caption{Super cape 1}
\end{subfigure}
\begin{subfigure}{0.4\textwidth}
\centering
\includegraphics[width=\textwidth]{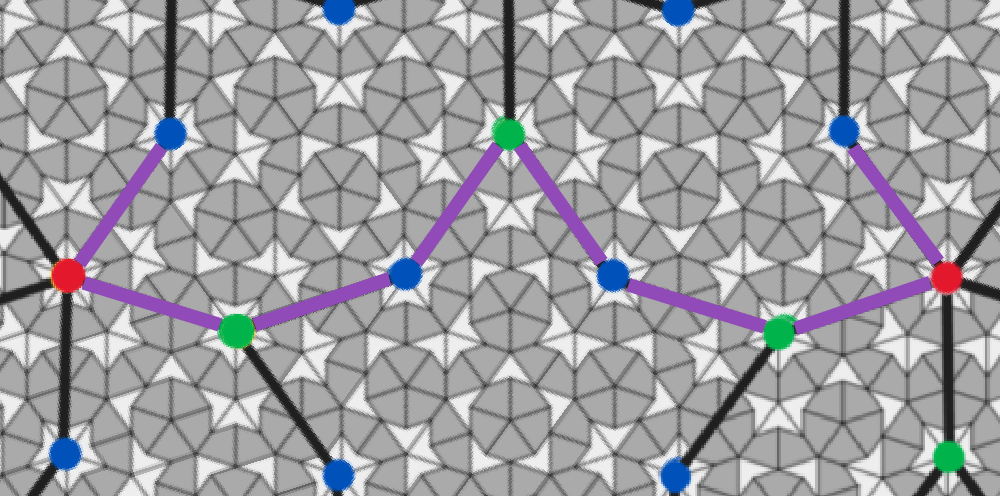}
\caption{Super cape 4}
\end{subfigure}
\caption{New caterpillars Super cape 1 and Super cape 4}
\label{Super capes}
\end{figure}

Next, we have to consider different possible kingdoms for a Porrier caterpillar B and for a Porrier caterpillar C. If a Porrier caterpillar B (resp. Porrier caterpillar C) is grafted to a Super cape 1 onto the prime caterpillar adjacent to the star of green center, at the end of the Porrier caterpillar, then we denote this Porrier caterpillar by \textit{Porrier caterpillar $B_1$} (resp. \textit{Porrier caterpillar $C_1$}). In all other cases, the Porrier caterpillar is denoted by \textit{Porrier caterpillar $B_2$} (resp. \textit{Porrier caterpillar $C_2$}). We denote a Porrier caterpillar D built on a boat by \textit{Porrier caterpillar $D_1$}.

Let $E$ be the set \{Super cape 1, Porrier caterpillar A, Porrier caterpillar $B_1$, Porrier caterpillar $B_2$, Porrier caterpillar $C_1$, Porrier caterpillar $C_2$, Porrier caterpillar $D_1$, Super cape 4\}. Recall that if a Porrier caterpillar is built on a boat, there exists an equivalent caterpillar that follows the hull instead of the sail. In the following figures, each Porrier caterpillar built on a boat is replaced by its equivalent version. Thus, there is no confusion between a Porrier caterpillar $D_1$ and a Super cape 1.

For each caterpillar in $E$, we apply three successive inflations to its kingdom and define a new caterpillar in the resulting patch (see Figure~\ref{fig:nouvelle chenille infinie}). Let $\psi$ be the function associating  each caterpillar in $E$ to the caterpillar defined in Figure~\ref{fig:nouvelle chenille infinie} after three inflations.\\

Each image of $\psi$ is thus a grafting of caterpillars from $E$. Now, for all caterpillars $C_1, C_2 \in E$ such that $C_1 \diamond C_2$ belongs to a caterpillar in the image of $\psi$, we show that $\psi(C_1)\diamond \psi(C_2)$ is a caterpillar obtained by grafting caterpillars from $E$. All possible cases (up to reflection) are shown in Figure~\ref{fig:valide pour greffages:nouvelle chenille infinie}.

\begin{figure}[H]
\centering
\begin{subfigure}{0.5\textwidth}
\centering
\includegraphics[width=\textwidth]{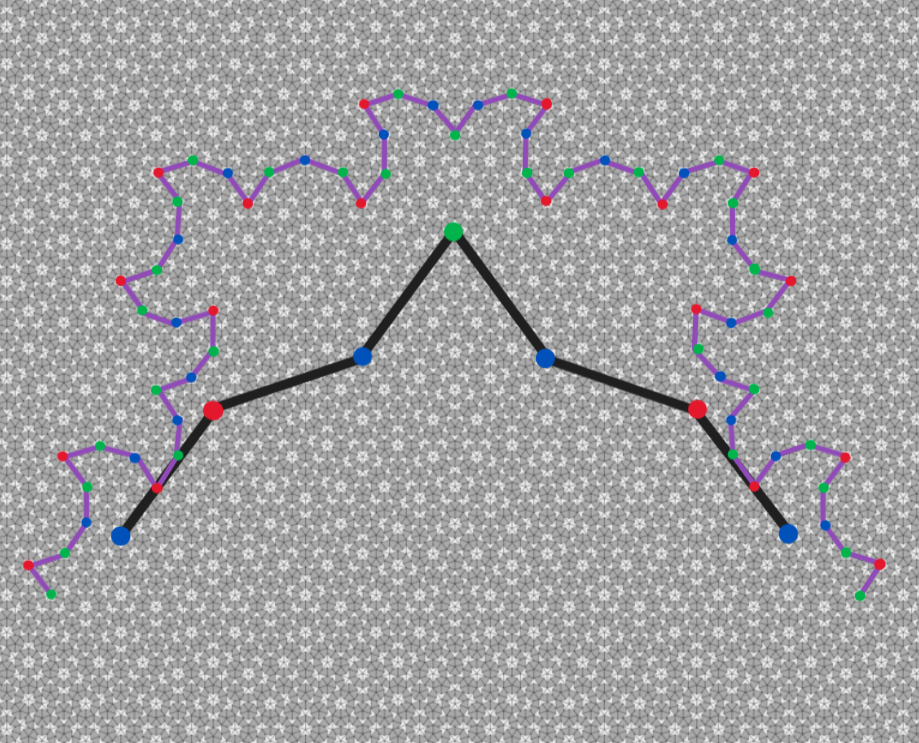}
\caption{$\psi(\text{Super cape 1})$}
\end{subfigure}
\begin{subfigure}{0.5\textwidth}
\centering
\includegraphics[width=\textwidth]{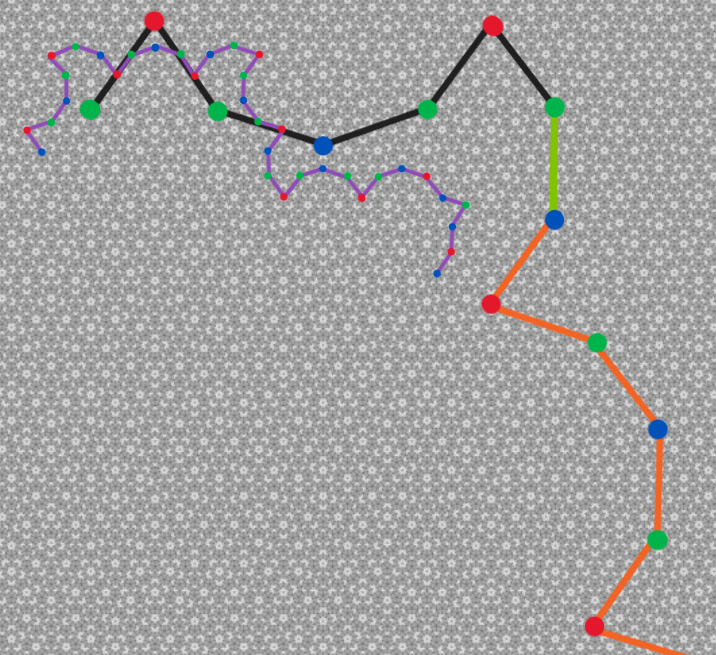}
\caption{$\psi(\textit{Porrier caterpillar A})$}
\end{subfigure}
\begin{subfigure}{0.5\textwidth}
\centering
\includegraphics[width=\textwidth]{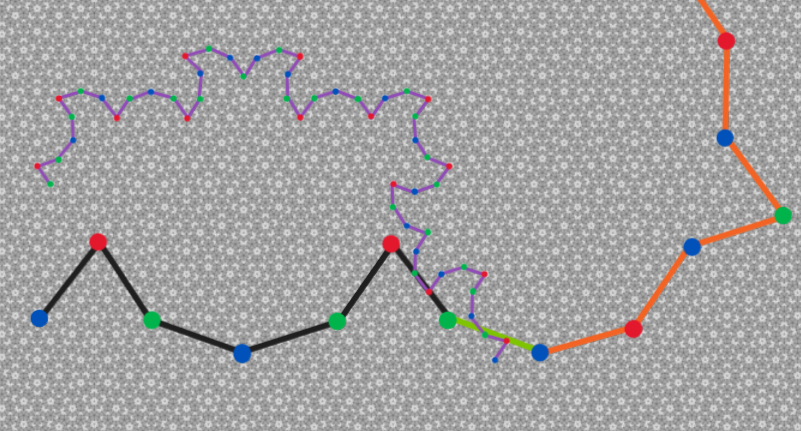}
\caption{$\psi(\textit{Porrier caterpillar $B_1$})$}
\end{subfigure}

\end{figure}

\begin{figure}[H]
\ContinuedFloat
\centering
\begin{subfigure}{0.5\textwidth}
\centering
\includegraphics[width=\textwidth]{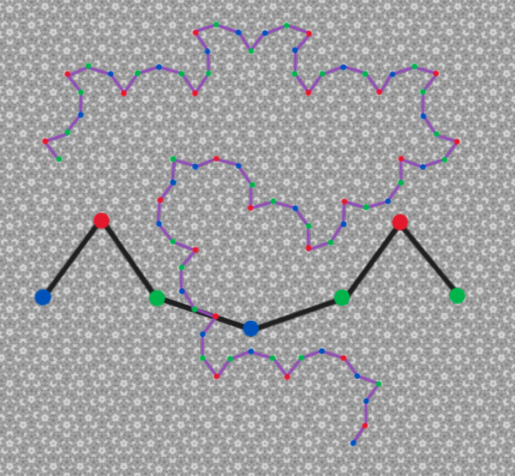}
\caption{$\psi(\textit{Porrier caterpillar $B_2$})$}
\end{subfigure}
\begin{subfigure}{0.5\textwidth}
\centering
\includegraphics[width=\textwidth]{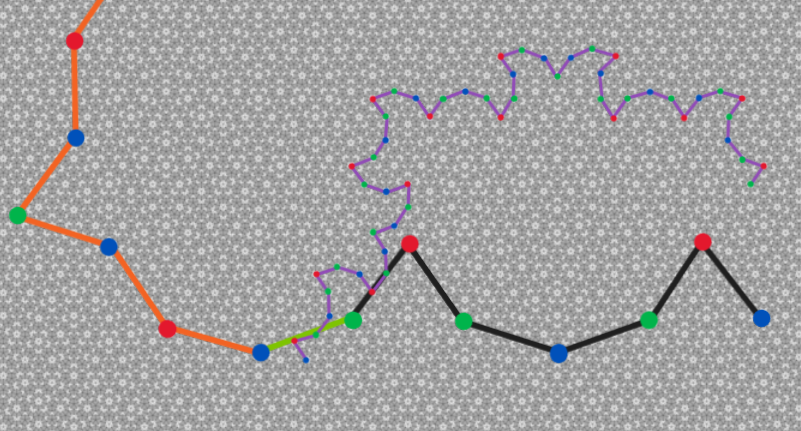}
\caption{$\psi(\textit{Porrier caterpillar $C_1$})$}
\end{subfigure}
\begin{subfigure}{0.5\textwidth}
\centering
\includegraphics[width=\textwidth]{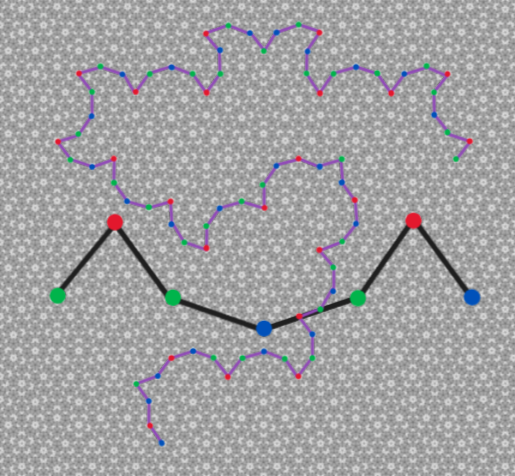}
\caption{$\psi(\textit{Porrier caterpillar $C_2$})$}
\end{subfigure}
\end{figure}

\begin{figure}[H]
\ContinuedFloat
\centering
\begin{subfigure}{0.5\textwidth}
\centering
\includegraphics[width=\textwidth]{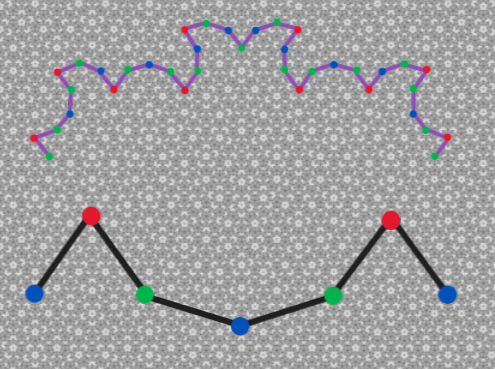}
\caption{$\psi(\textit{Porrier caterpillar $D_1$})$}
\end{subfigure}
\begin{subfigure}{0.5\textwidth}
\centering
\includegraphics[width=\textwidth]{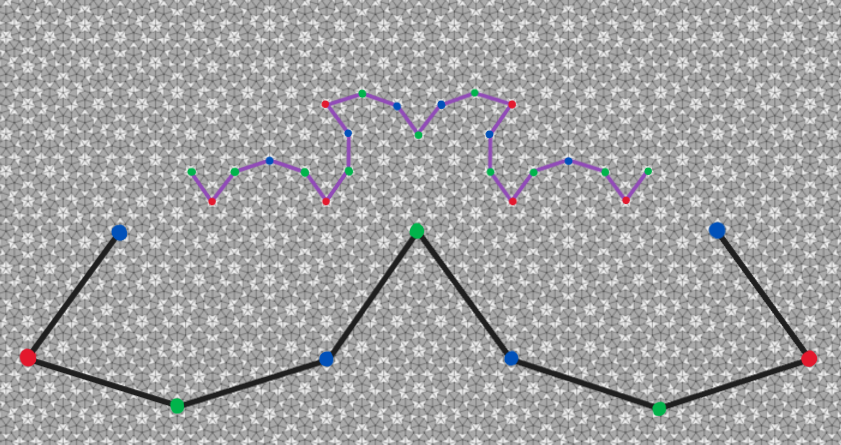}
\caption{$\psi(\text{Super cape 4})$}
\end{subfigure}
\caption{Images of $\psi$ for elements of $E$. The black chain represents the HBS chain of the preimage caterpillar (before applying three inflations). The image of $\psi$ is the caterpillar with HBS chain in purple. Some black chains have an extension (an orange chain and a light-green chain connecting the black and orange ones) used as a reference. More precisely, since a Porrier caterpillar A is symmetric by reflection but its image by $\psi$ is not, we fix a reference. In our constructions, a Porrier caterpillar A is always grafted to exactly one such reference, ensuring no ambiguity in the images of $\psi$. For other subfigures with such an extension, this serves to distinguish two different preimages for the same Porrier caterpillar (the image by $\psi$ differs).}
\label{fig:nouvelle chenille infinie}
\end{figure}

Starting from a Super cape 4, we apply $3n$ inflations to the patch containing it, for $n \in \mathbb{N}$. For every $n$, the resulting caterpillar necessarily contains a Super cape 4, since $\psi(\text{Super cape 4})$ contains a Super cape 4. Letting $n$ tend to infinity yields a bi-infinite fully leafed caterpillar.  Figure \ref{fig:valide pour greffages:nouvelle chenille infinie} guarantees that the graph obtained after $3n$ inflations is connected. It is also acyclic, since no cycles are formed in Figure \ref{fig:valide pour greffages:nouvelle chenille infinie}. Moreover, the kingdoms of two grafted caterpillars from $E$ (or their thrice-inflated kingdoms) intersect locally, in regions where we have verified that no cycle arises.
\end{proof}

\begin{figure}[H]
\centering
\begin{subfigure}{0.6\textwidth}
\centering
\includegraphics[width=\textwidth]{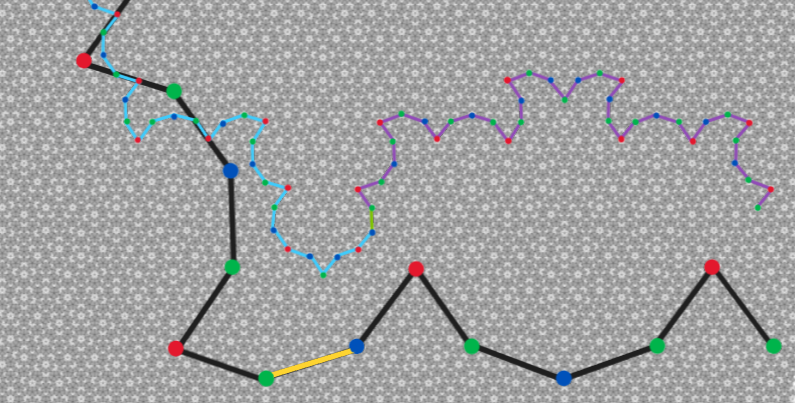}
\caption{$\psi(\textit{Porrier Caterpillar A}) \diamond \psi(\textit{Porrier Caterpillar $B_i$})$, $i\in \{1,2\}$}
\end{subfigure}
\begin{subfigure}{0.6\textwidth}
\centering
\includegraphics[width=\textwidth]{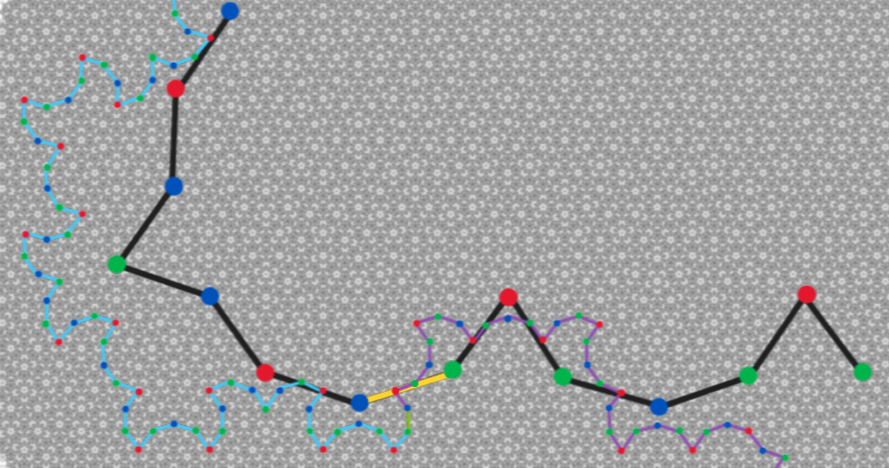}
\caption{$\psi(\textit{Porrier caterpillar A}) \diamond \psi(\textit{Super cape 1})$}
\end{subfigure}
\begin{subfigure}{0.6\textwidth}
\centering
\includegraphics[width=\textwidth]{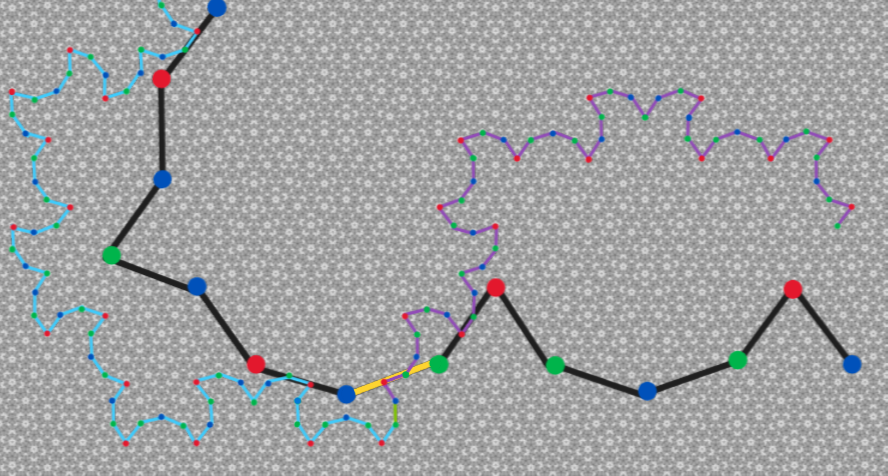}
\caption{$\psi(\textit{Super cape 1}) \diamond \psi(\textit{Porrier caterpillar $C_1$})$}
\end{subfigure}
\end{figure}

\begin{figure}[H]
\ContinuedFloat
\centering
\begin{subfigure}{0.6\textwidth}
\centering
\includegraphics[width=\textwidth]{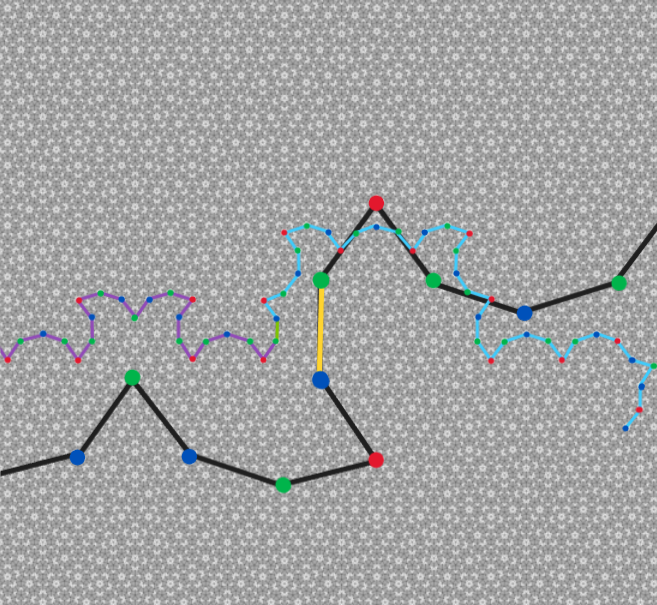}
\caption{$\psi(\textit{Super cape 4}) \diamond \psi(\textit{Porrier Caterpillar A})$}
\end{subfigure}
\end{figure}

\begin{figure}[H]
\ContinuedFloat
\centering
\begin{subfigure}{0.6\textwidth}
\centering
\includegraphics[width=\textwidth]{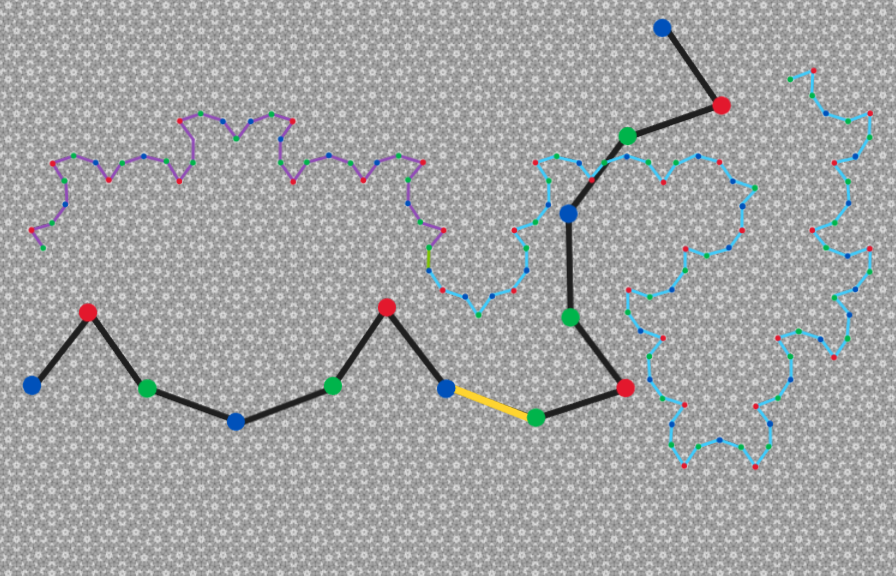}
\caption{$\psi(\textit{Porrier caterpillar D}) \diamond \psi(\textit{Porrier caterpillar $B_2$})$}
\end{subfigure}
\caption{Graftings $\psi(C_1)\diamond \psi(C_2)$ where $C_1, C_2 \in E$ are such that $C_1 \diamond C_2$ belongs to a caterpillar in the possible images of $\psi$. The yellow edges show where the graft is done between the two caterpillars that are preimages of $\psi$, and the light-green edges show where the graft is done between the two caterpillars that are images of $\psi$.}
\label{fig:valide pour greffages:nouvelle chenille infinie}
\end{figure}

\newpage
\begin{corollary}
The bi-infinite fully leafed caterpillar described in Theorem~\ref{thm 1} is not the only bi-infinite fully leafed caterpillar.
\end{corollary}
\begin{proof}
The constructions in Theorem~\ref{thm 1} and in  Theorem \ref{thm:deuxième chenille infinie} are clearly distinct. For instance, no bi-infinite caterpillar obtained from the construction of Theorem~\ref{thm 1} contains a cape 4, whereas the construction of Theorem \ref{thm:deuxième chenille infinie} does.
\end{proof}

\subsection{Elimination of sea graftings}

We continue the elimination of patterns in bi-infinite fully leafed caterpillars by ruling out graftings of sea caterpillars.

\begin{definition}
A \textbf{sea grafting} is defined as the grafting of two sea caterpillars.
\end{definition}

From Figures \ref{fig:chenilles marines} and \ref{fig:royaumes HBS} and from the previous results, we can list all the sea graftings that, at this point in the paper, are known to belong to a bi-infinite fully leafed caterpillar and those for which we have not yet proved that they cannot. This enumeration is presented in Figure \ref{fig:greffages marins}. Among the graftings shown in Figure \ref{fig:greffages marins}, those for which it is still unknown whether they belong to a bi-infinite fully leafed caterpillar are graftings (g), (h), (i), (j), (k), and (l). All the others do belong to a bi-infinite fully leafed caterpillar already studied in Theorem \ref{thm 1} or Theorem \ref{thm:deuxième chenille infinie}. Let us therefore examine the remaining graftings.

\begin{proposition}\label{d}
If the sea graftings (i) and (k) from Figure \ref{fig:greffages marins} are such that the hexagons of these docks are adjacent to the same boat, which itself is not adjacent to exactly two other boats (see Figure \ref{fig:configuration possiblement bi-infiniment prolongeable}), then these graftings cannot belong to a bi-infinite fully leafed caterpillar.
\end{proposition}

\begin{figure}[H]
\centering
\includegraphics[width=0.5\textwidth, trim=0cm 17cm 0cm 4cm, clip]{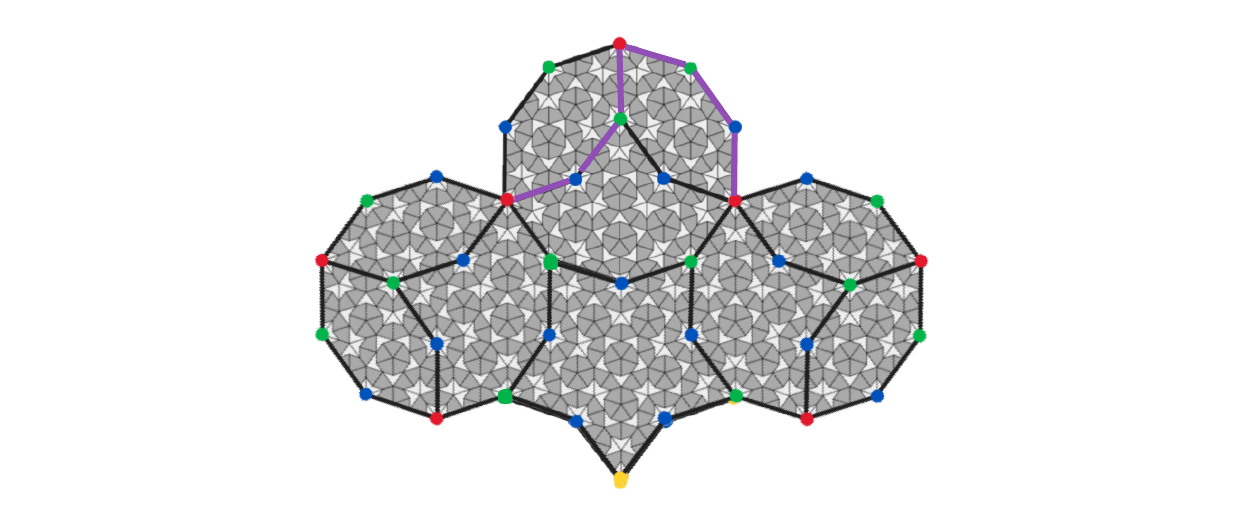}
\caption{Proposition \ref{d} states that if the sea grafting (k) does not occur on this patch (i.e. the boat is not adjacent to two other boats), then it cannot be extended into a bi-infinite fully leafed caterpillar. No claim is made for the case in which grafting (k) is constructed on this patch.}
\label{fig:configuration possiblement bi-infiniment prolongeable}
\end{figure}

\begin{proof}
It is enough to examine the possible extensions of these sea graftings and observe that none of them can be bi-infinite. See Appendix C.
\end{proof}

We were unable to determine whether the configuration not excluded in Proposition \ref{d} (see Figure \ref{fig:configuration possiblement bi-infiniment prolongeable}) could be such that the sea grafting (k) it contained can be extended into a bi-infinite fully leafed caterpillar. We believe that this might be the case, but that it would require a construction similar to that used in Theorem \ref{thm:deuxième chenille infinie} to establish it.

\begin{figure}[H]
\centering
\setlength{\tabcolsep}{0pt}
\begin{tabular}{@{}cccc@{}}
\includegraphics[width=0.2\textwidth, trim=0cm 6cm 0cm 0cm, clip]{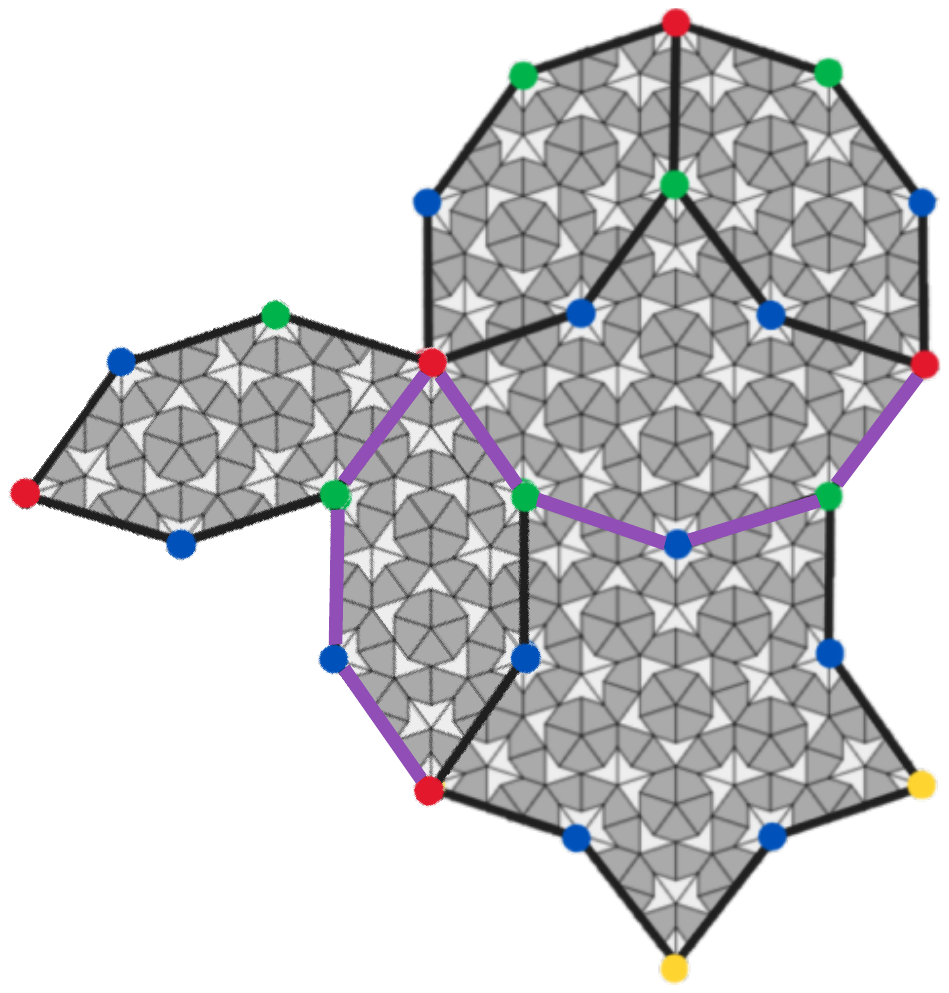} &

\includegraphics[width=0.3\textwidth, trim=0cm 13cm 0cm 0cm, clip]{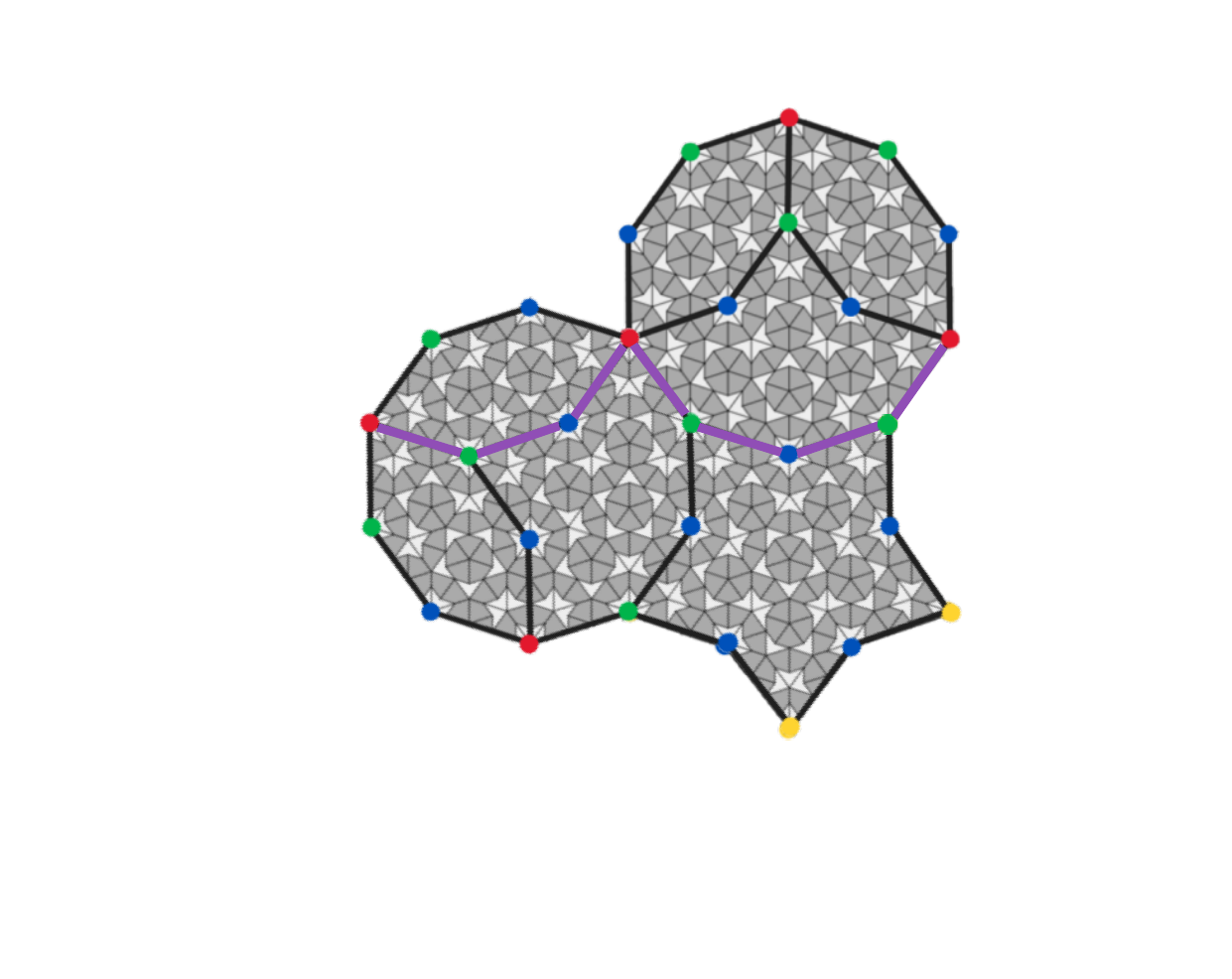} &

\includegraphics[width=0.2\textwidth, trim=0cm 6cm 0cm 0cm, clip]{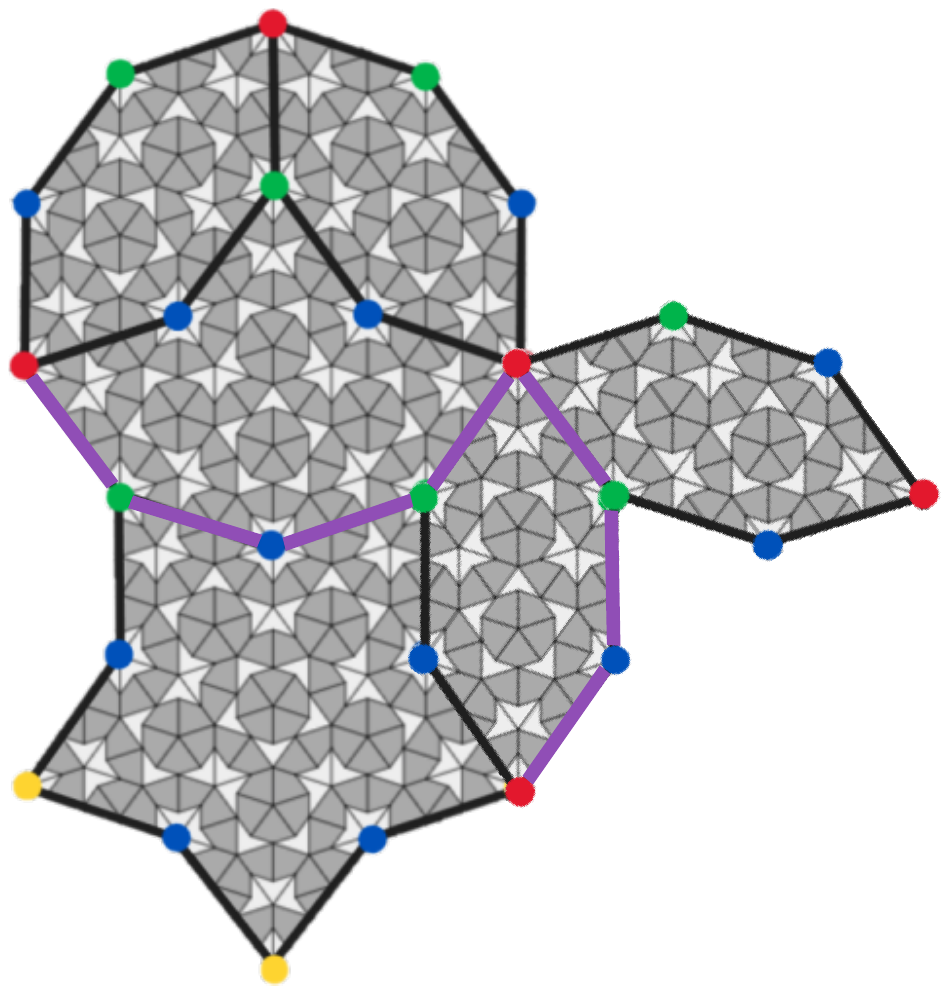} &

\includegraphics[width=0.3\textwidth, trim=0cm 15cm 0cm 0cm, clip]{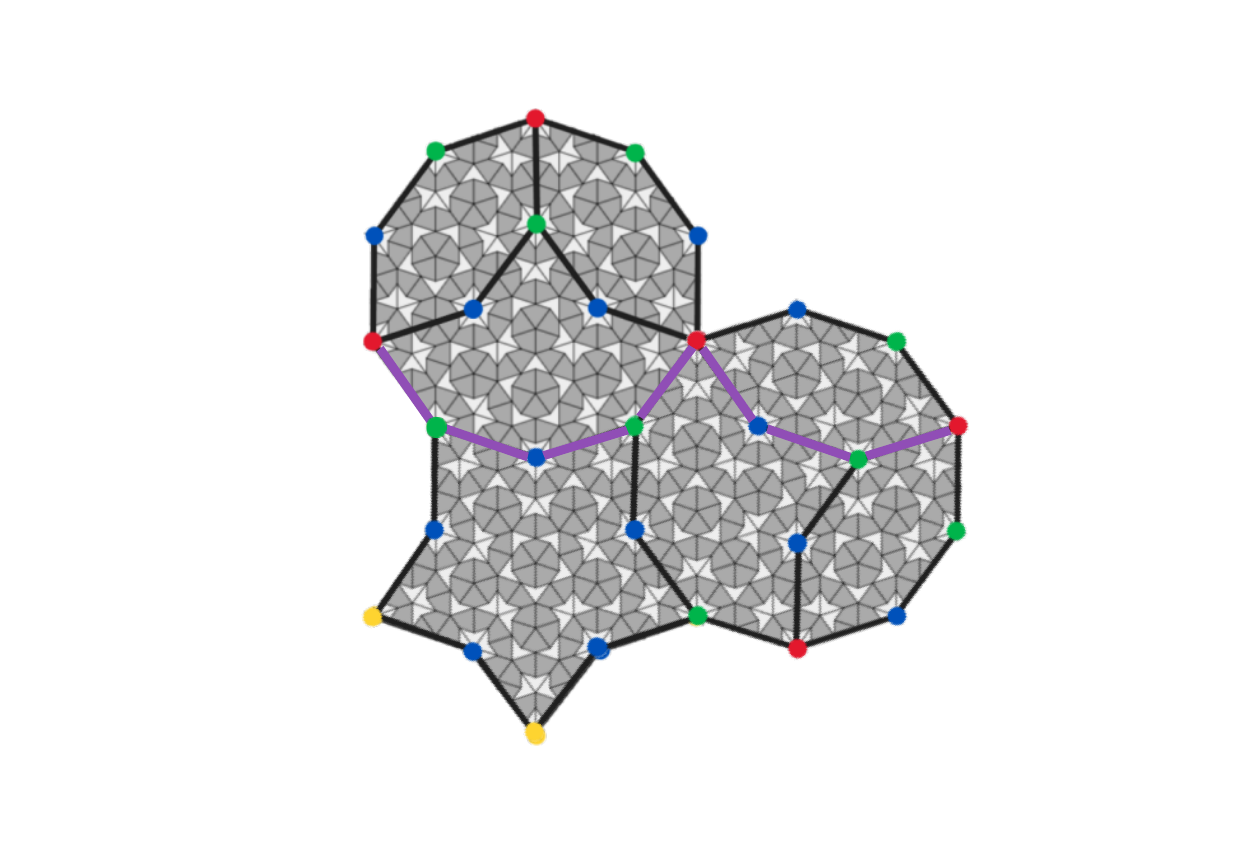} \\

    \small\shortstack{(a) dock 3 $\diamond$ hull\\ \hspace{14pt}(case 1)} &
    \small\shortstack{(b) dock 3 $\diamond$ hull\\ \hspace{14pt} (case 2)} &
    \small\shortstack{(c) hull $\diamond$ dock 1 \\ \hspace{3pt} (case 1)} &
    \small\shortstack{(d) hull $\diamond$ dock 1 \\ \hspace{3pt} (case 2)}\\

\includegraphics[width=0.2\textwidth, trim=0cm 7cm 0cm 0cm, clip]{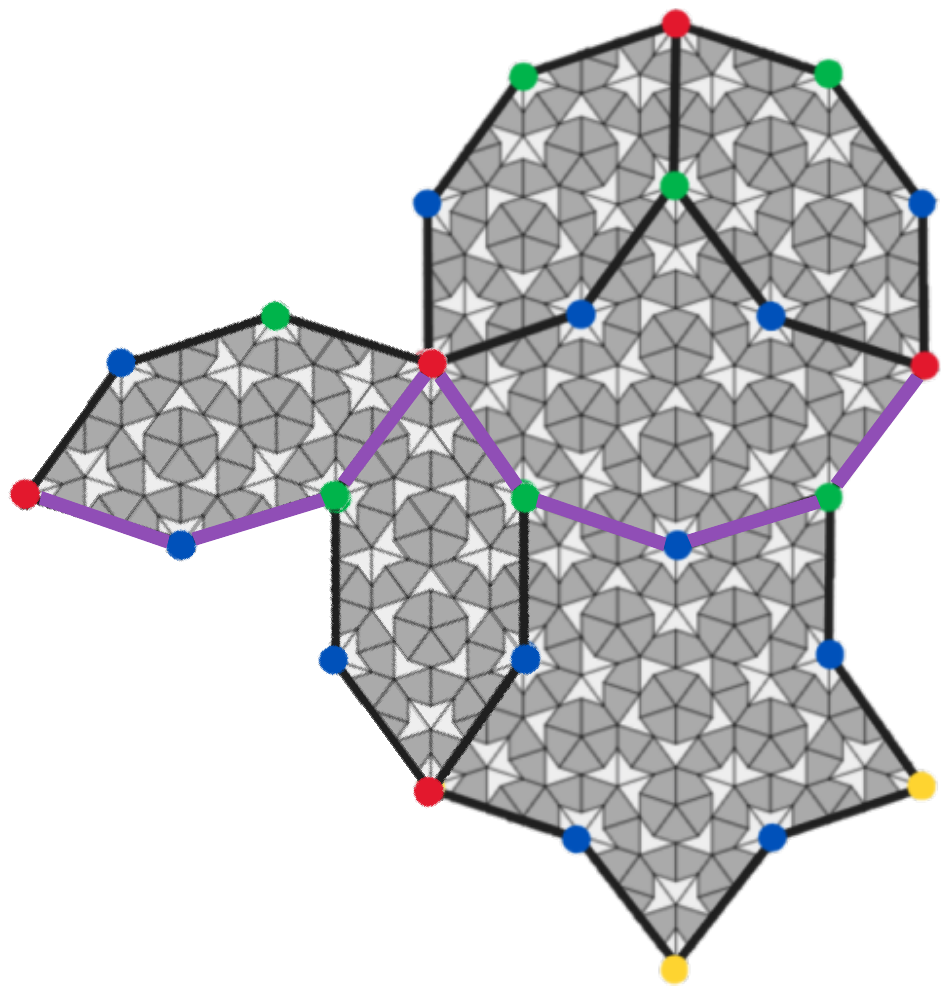} &

\includegraphics[width=0.2\textwidth, trim=0cm 4cm 0cm 0cm, clip]{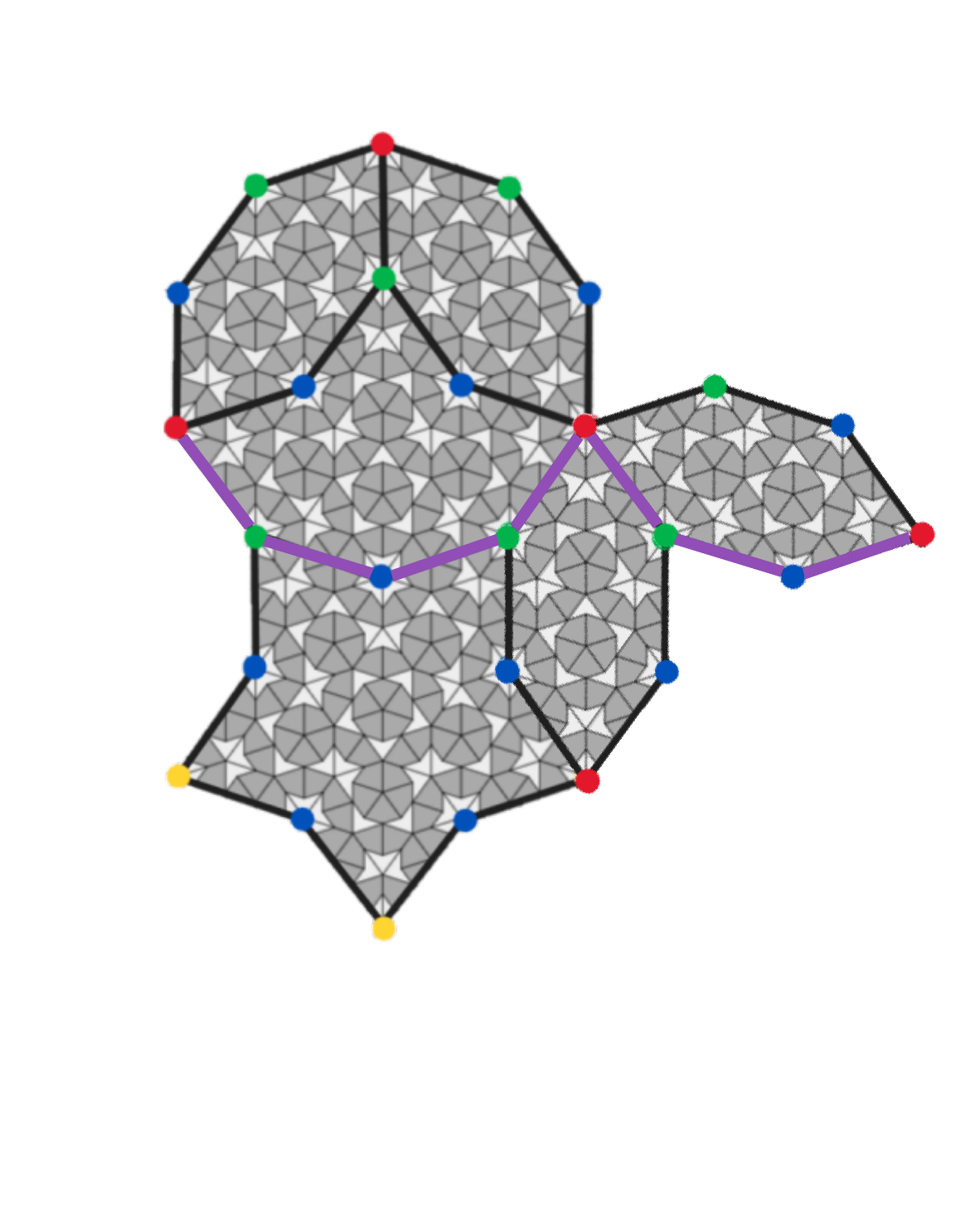} &

\includegraphics[width=0.2\textwidth, trim=0cm 17cm 0cm 0cm, clip]{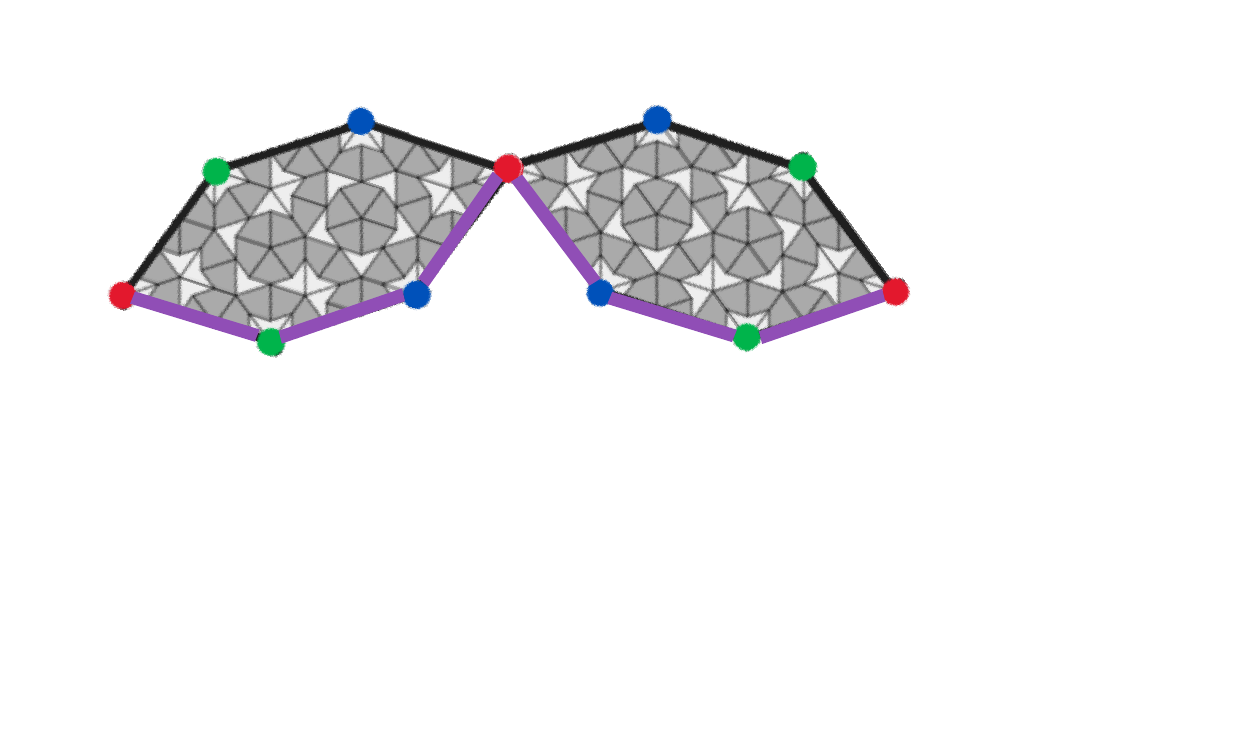} &
\label{fig27g}

\includegraphics[width=0.2\textwidth, trim=0cm 17cm 0cm 0cm, clip]{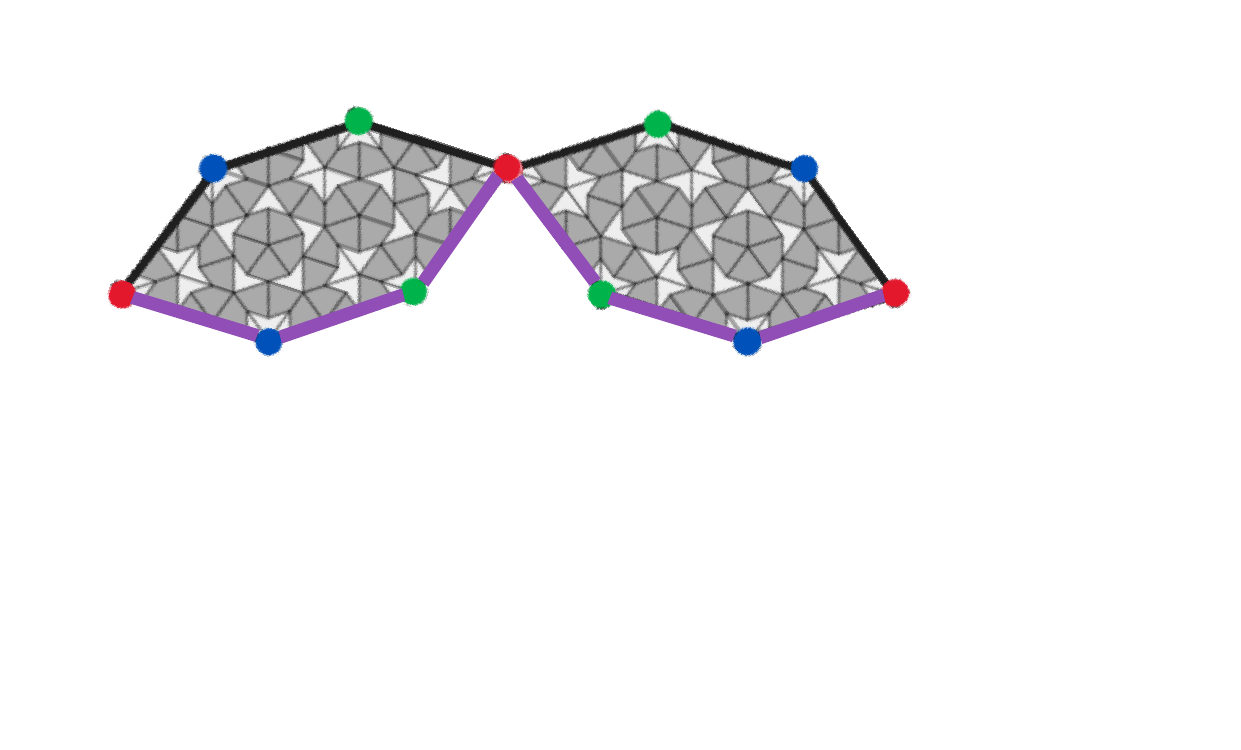} \\

    \small{(e) dock 4 $\diamond$ hull} &
    \small{(f) hull $\diamond$ dock 1} &
    \small{(g) dock 3 $\diamond$ dock 1} &
    \small{(h) dock 4 $\diamond$ dock 2}\\

\includegraphics[width=0.2\textwidth, trim=0cm 7cm 0cm 0cm, clip]{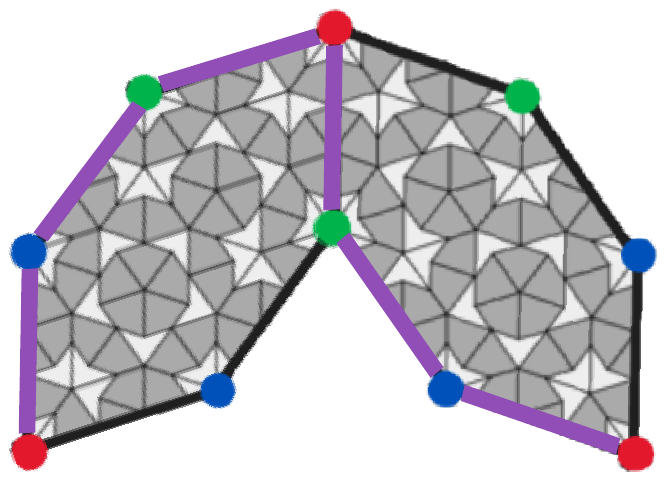} &

\includegraphics[width=0.2\textwidth, trim=0cm 18cm 0cm 0cm, clip]{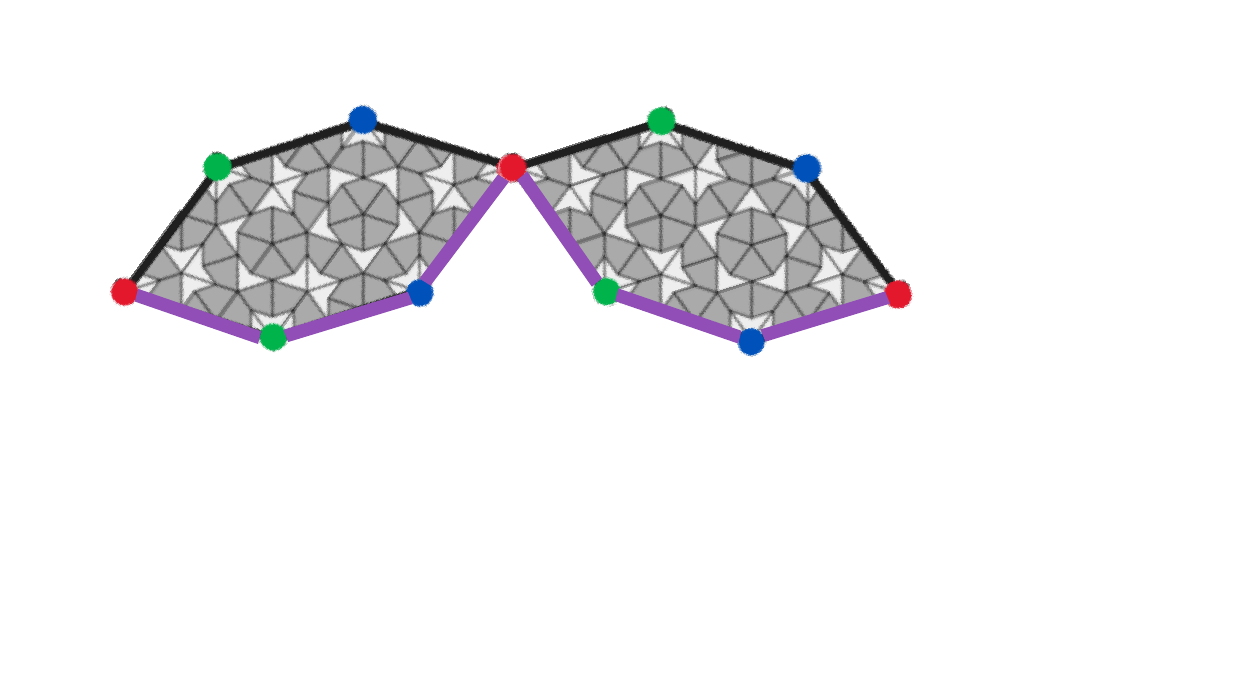} &

\includegraphics[width=0.2\textwidth, trim=0cm 7cm 0cm 0cm, clip]{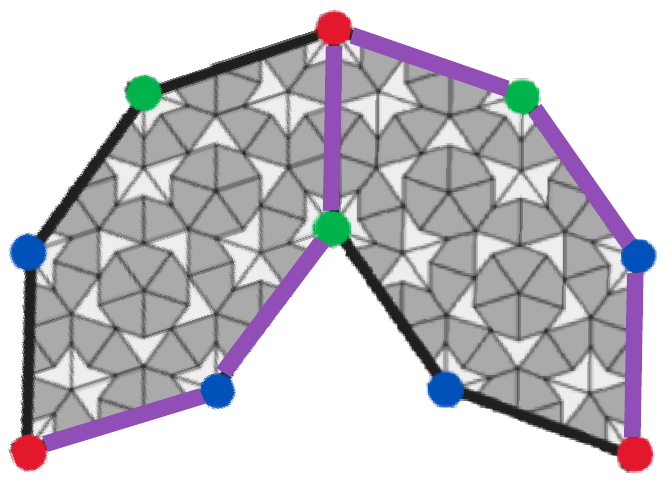} &

\includegraphics[width=0.2\textwidth, trim=0cm 18cm 0cm 0cm, clip]{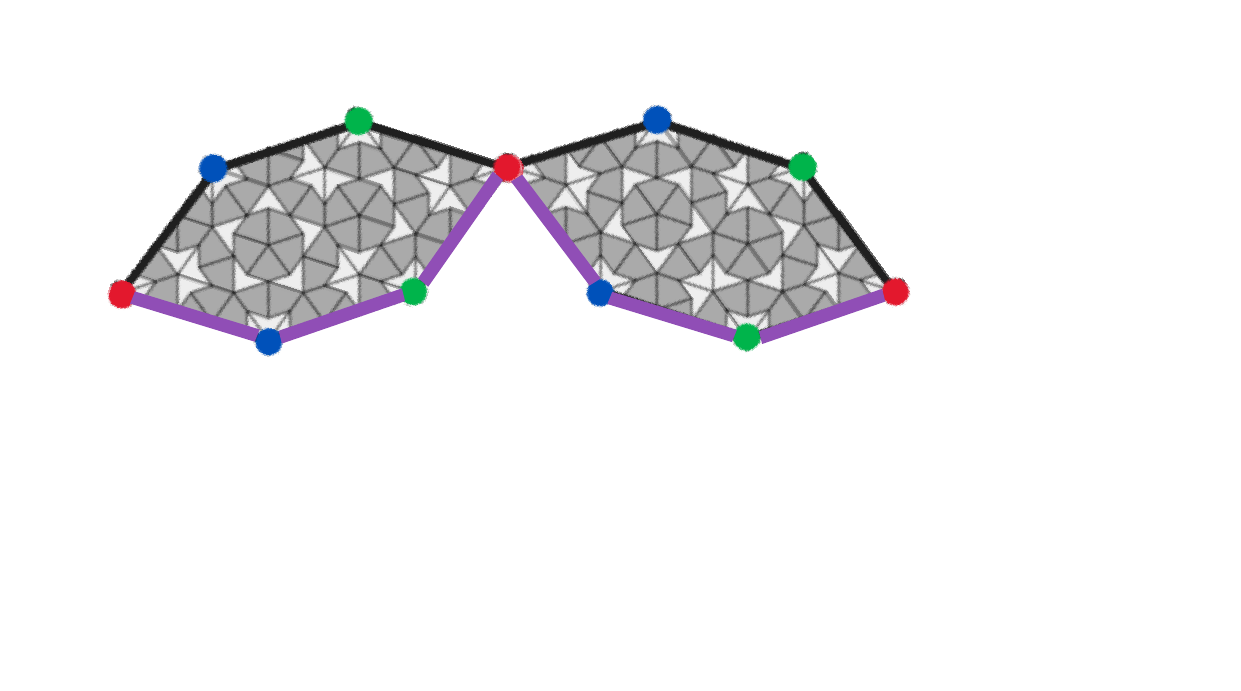} \\

    \small\shortstack{(i) dock 3 $\diamond$ dock 2 \\ \hspace{3pt} (case 1)} &
    \small\shortstack{(j) dock 3 $\diamond$ dock 2 \\ \hspace{3pt} (case 2)} &
    \small\shortstack{(k) dock 4 $\diamond$ dock 1 \\ \hspace{3pt} (case 1)} &
    \small\shortstack{(l) dock 4 $\diamond$ dock 1 \\ \hspace{3pt} (case 2)}\\

\includegraphics[width=0.2\textwidth, trim=0cm 10cm 0cm 0cm, clip]{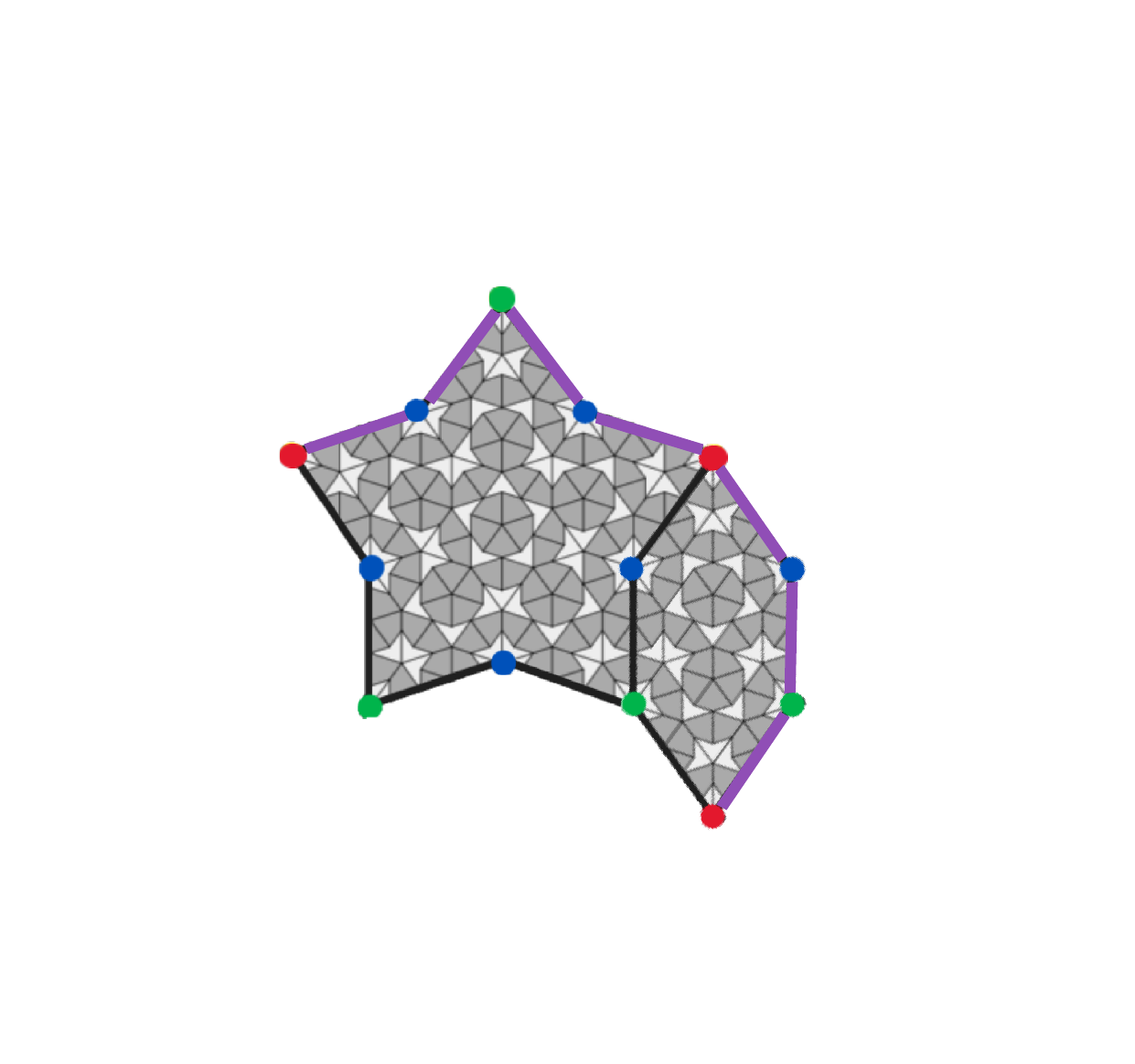} &

\includegraphics[width=0.2\textwidth, trim=0cm 10cm 0cm 0cm, clip]{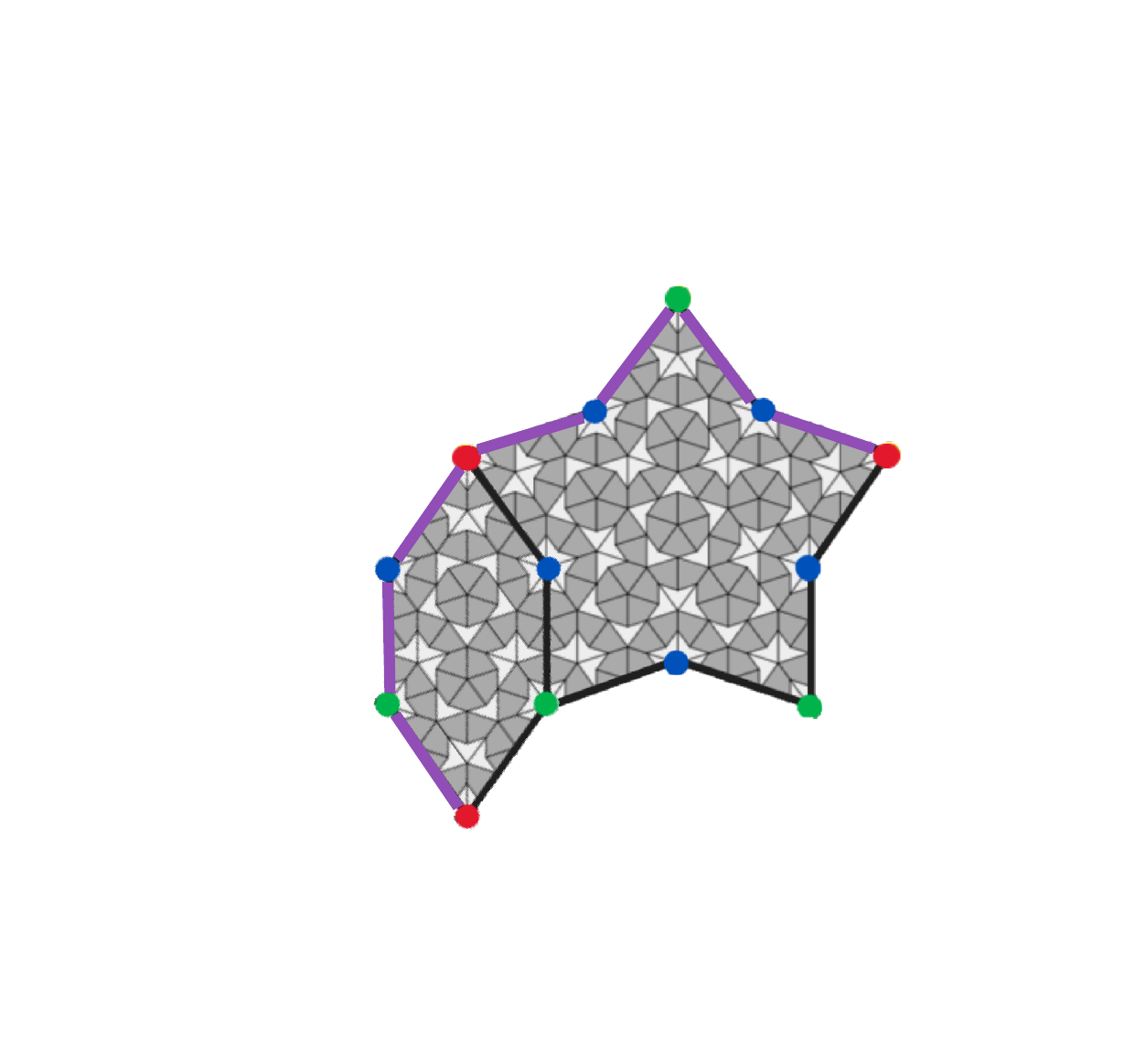} &

\includegraphics[width=0.2\textwidth, trim=0cm 10cm 0cm 0cm, clip]{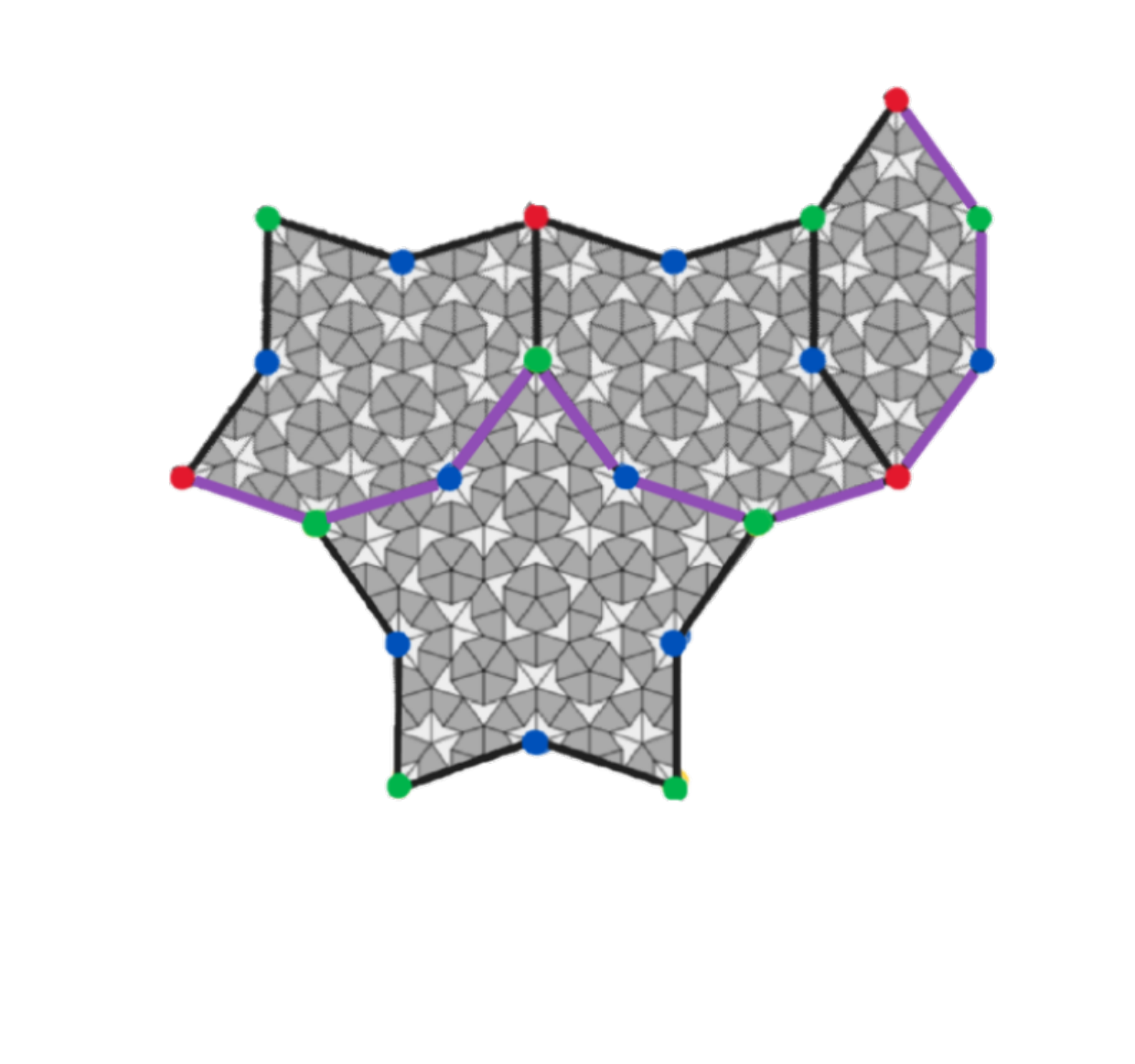} &

\includegraphics[width=0.2\textwidth, trim=0cm 10cm 0cm 0cm, clip]{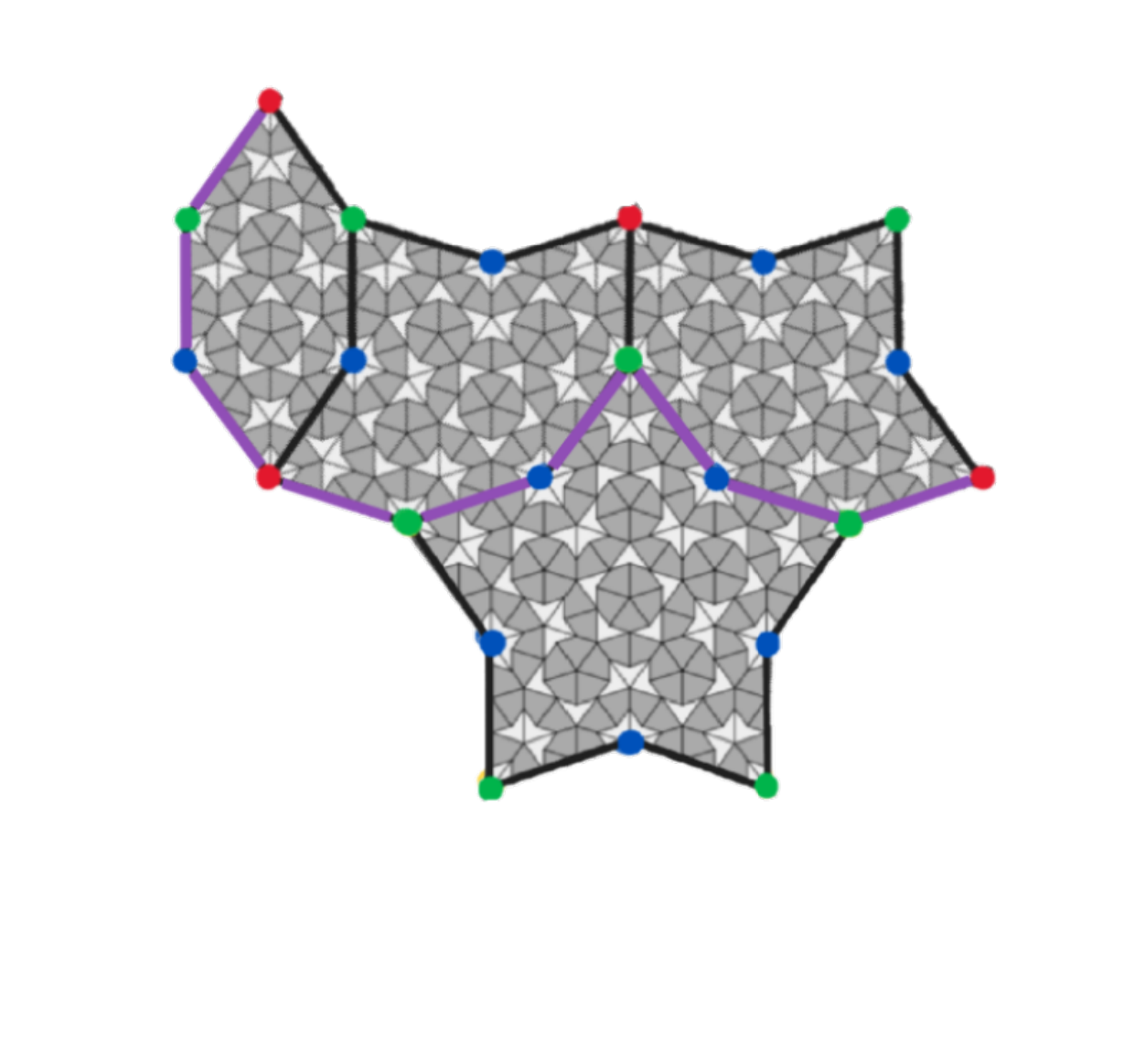}\\

    \small{(m) cape 1 $\diamond$ dock 2} &
    \small{(n) dock 4 $\diamond$ cape 1} &
    \small{(o) cape 4 $\diamond$ dock 3} &
    \small{(p) dock 1 $\diamond$ cape 4}

\end{tabular}
\caption{All sea graftings known to belong to a bi-infinite fully leafed caterpillar or not yet eliminated}
\label{fig:greffages marins}
\end{figure}

\newpage
\begin{proposition}\label{e}
The sea graftings (j) and (l) from Figure \ref{fig:greffages marins} do not belong to any bi-infinite fully leafed caterpillar.
\end{proposition}

\begin{proof}
It is enough to examine the possible extensions of these sea graftings and observe that none of them can be bi-infinite. See Appendix D.

\end{proof}

\begin{proposition}\label{f}
The sea grafting (h) from Figure \ref{fig:greffages marins} does not belong to any bi-infinite fully leafed caterpillar.
\end{proposition}

\begin{proof}
It is enough to examine the possible extensions of these sea graftings and observe that none of them can be bi-infinite. See Appendix E.

\end{proof}

The last sea grafting of Figure \ref{fig:greffages marins} that has not yet been analyzed is sea grafting (g). We think that this sea grafting can be extended into a bi-infinite fully leafed caterpillar and that proving this requires a non-trivial construction similar to the one used in Theorem \ref{thm:deuxième chenille infinie}. We also believe that in order to find this construction, it may be useful to rely on the construction from Theorem \ref{thm:deuxième chenille infinie}: the new construction to be formalized seems to be definable by a certain sort of reflection of the construction from Theorem \ref{thm:deuxième chenille infinie} (see Figure \ref{conj}). If Conjecture \ref{conj_g} is true, it would be even better if such a bi-infinite fully leafed caterpillar containing the sea grafting $(g)$ of Figure \ref{fig:greffages marins} also contains the sea grafting (k) (the only other sea grafting, up to reflection, for which we have not reached a conclusion). We would then know, among all possible sea graftings, which ones are extendable into a bi-infinite fully leafed caterpillar.

 \begin{conjecture}\label{conj_g}
 The sea grafting (g) of Figure \ref{fig:greffages marins} can be extended into a bi-infinite fully leafed caterpillar.
 \end{conjecture}

\begin{figure}[H]
    \centering
        \includegraphics[width=0.5\textwidth]{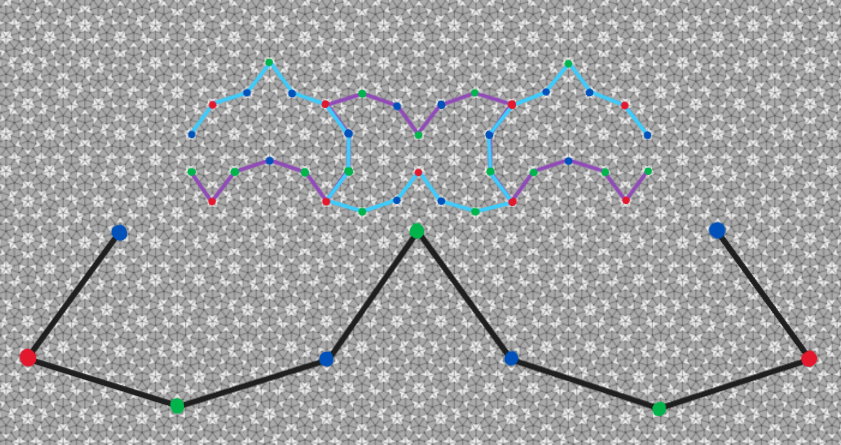}
        \caption{A fully leafed caterpillar (in light blue) that seems to follow $\psi(\text{Super cape 4})$ and contains the sea grafting (g). It seems that a certain kind of reflection of the bi-infinite fully leafed caterpillar constructed in Theorem \ref{thm:deuxième chenille infinie} could be defined to find a new bi-infinite caterpillar that contains the sea grafting (g) of Figure \ref{fig:greffages marins}.}
        \label{conj}
\end{figure}

\section{Final remarks}

In conclusion, we now have a better understanding of the fully leafed induced subtrees in the Penrose P2 tilings. We have determined their graph structure and, consequently, reduced the study of fully leafed induced subtrees of P2-graphs to that of fully leafed induced subcaterpillars of P2-graphs. We also know that every saturated fully leafed induced subtree is precisely either a prime caterpillar or a caterpillar obtained by grafting prime caterpillars.

We have also studied the bi-infinite fully leafed caterpillars in the Penrose P2 tilings. Not only does such a caterpillar exist, but it is not unique. The reader may have noticed that enumerating all of them quickly becomes a challenging problem, mainly because of the complexity of their construction.

We have presented results, particularly on the sea caterpillars and their grafting, that provide the basis for the following problem:\\

\textbf{Problem 1.} Find all bi-infinite fully leafed induced subcaterpillars in the Penrose P2 tilings.\\

We hope that by restricting the sea caterpillars and sea graftings that can belong to one of these caterpillars (as we have begun to do), it will become easier to identify the complete set of bi-infinite caterpillars. One possible direction would be to study the relationship between the bi-infinite caterpillar of Theorem \ref{thm 1} and the one (or the ones) of Theorem \ref{thm:deuxième chenille infinie}. For example, can we construct \textit{hybrid} caterpillars, whose construction would be based on the inflation rules of each of these two constructions?

More generally, in order to construct all bi-infinite fully leafed caterpillars, one should investigate whether there exists a construction that generalizes the two bi-infinite fully leafed caterpillars that we have presented and that is sufficiently general to generate all bi-infinite fully leafed caterpillars.

We then propose the following problem, which is an extension  of Problem 1:\\

\textbf{Problem 2.} Define the formal language that characterizes the set of bi-infinite fully leafed caterpillars of P2-graphs and study the properties of the words in this language.\\

We initiated this problem following Theorem \ref{thm 1}, where we suggested describing the caterpillar from that construction by a bi-infinite word. It would then be interesting to do the same for all bi-infinite fully leafed caterpillars of P2-graphs, and to study the relationships among these words as well as their combinatorial properties. These words are probably all aperiodic. The Fibonacci word and the golden ratio are known to play central roles in Penrose tilings. In this sense, it would be interesting to see whether the words describing the bi-infinite fully leafed caterpillars are related to the Fibonacci word and the golden ratio.\\

Of course, studying the fully leafed induced subtrees in other aperiodic tilings, such as the Penrose P1 and P3 tilings, the HBS tilings or the Einstein tiling\footnote{In 2023, a first aperiodic monotile tiling was discovered, which is called the \textit{Einstein} tiling. For further details, see \cite{smith2023aperiodic}.}, would also be of great interest. It would then be natural to compare these fully leafed induced subtrees with those of the Penrose P2 tilings.\\

\textbf{Acknowledgements:} We would like to thank Kevin Bertman, the developer of the website \url{https://www.mrbertman.com/penroseTilings.html}. This website enabled us to construct tile patches of Penrose P2 tilings.

\newpage

\section*{Appendix A: Proof of Proposition \ref{prop:élimination de cap 2 et cap 3}}

It suffices to treat cape 2 sea caterpillars. The same holds for cape 3 sea caterpillars by symmetry.

The following figures show all attempts to extend a cape 2 into a bi-infinite fully leafed caterpillar. Two cases are considered: a cape 2 on a S1 star (Case 1) and a cape 2 on a S2 star (Case 2). In each case, the initial patch corresponds to the kingdom of the relevant HBS tiles.

\begin{figure}[H]
    \centering
    \includegraphics[width=0.7\textwidth, trim=0cm 5cm 0cm 0cm, clip]{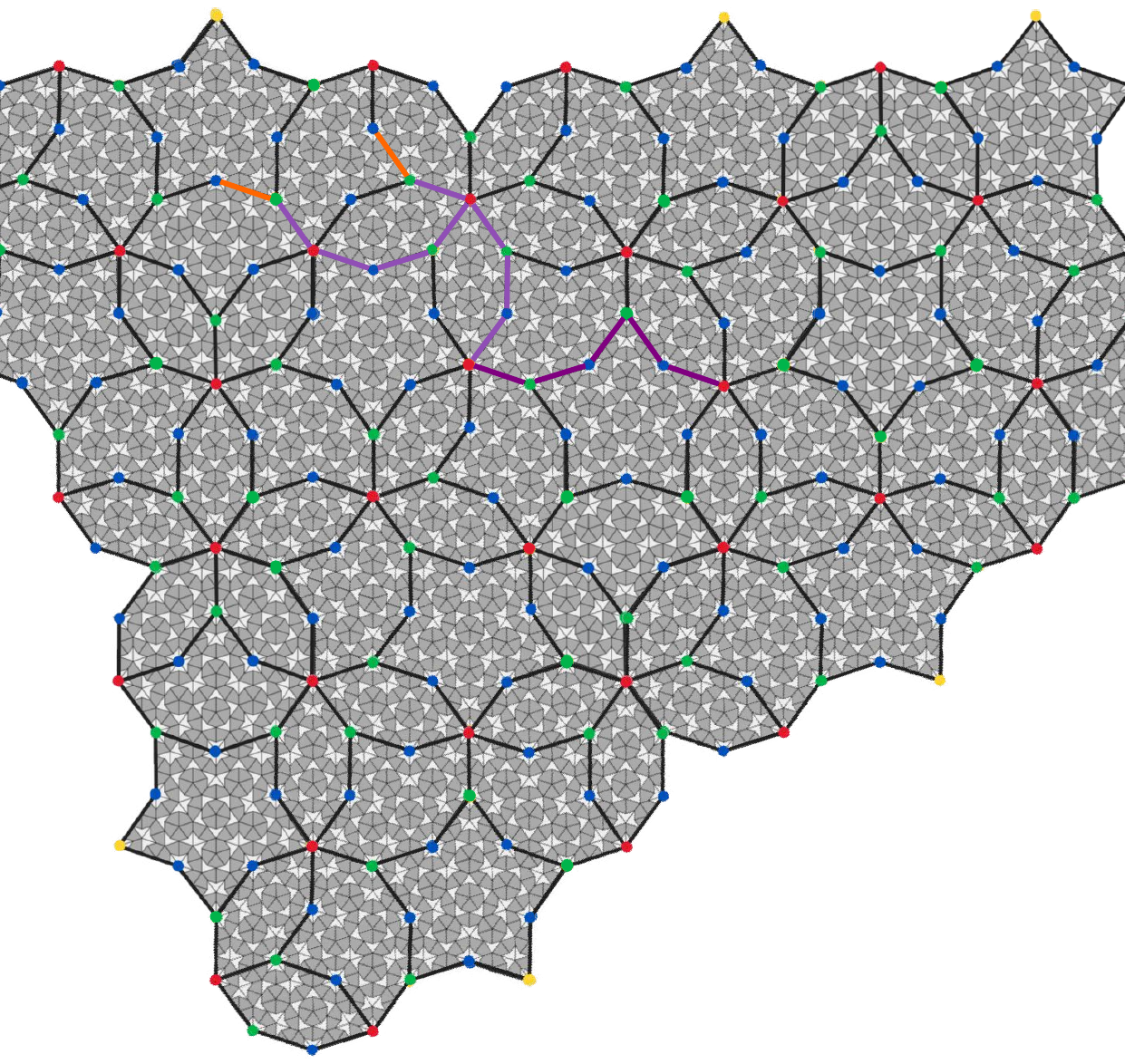}
    \caption{Case 1 — cape 2 on a S1 star. In dark purple, there is the HBS chain of a cape 2. In light purple, there is all possible extensions. Orange edges mark dead ends where previous lemmas or propositions fail. If all extensions end in orange, a cape 2 cannot be extended into a bi-infinite fully leafed caterpillar.}
    \label{Étude cap 2 cas 1}
\end{figure}

\nopagebreak
\begin{figure}[H]
    \centering
    \includegraphics[width=1.2\textwidth, trim=9cm 11cm 0cm 0cm, clip]{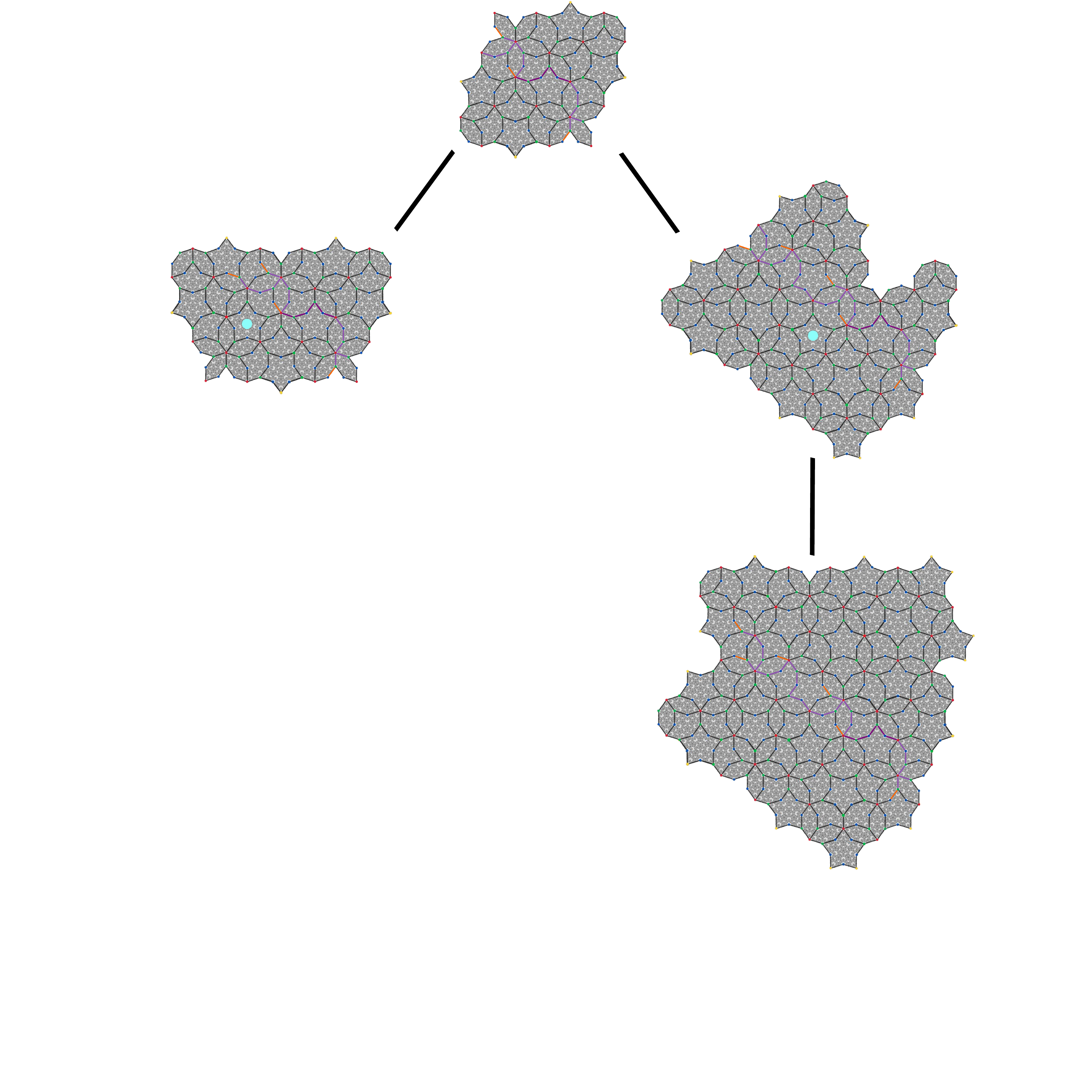}
    \caption{Case 2 — cape 2 on a S2 star. The kingdom of a S2 star is too small to conclude directly. For a yellow-centered vertex, subcases are distinguished (we try to fix the yellow vertex in red or in green), and the patch is enlarged step by step by adding forced tiles. A large light-blue circle marks the star where the subcase distinction occurs. Each step is shown below the previous one. Enlargement stops once the patch is large enough to conclude. Remaining yellow vertices do not affect the proof.}
    \label{Étude cap 2 cas 2}
\end{figure}

\newpage
\section*{Appendix B: Proof of Proposition \ref{c}}
Figure~\ref{fig:cap 4 cas pas prolongeable} demonstrates the proposition.
\begin{figure}[H]
\centering
\includegraphics[width=1\textwidth]{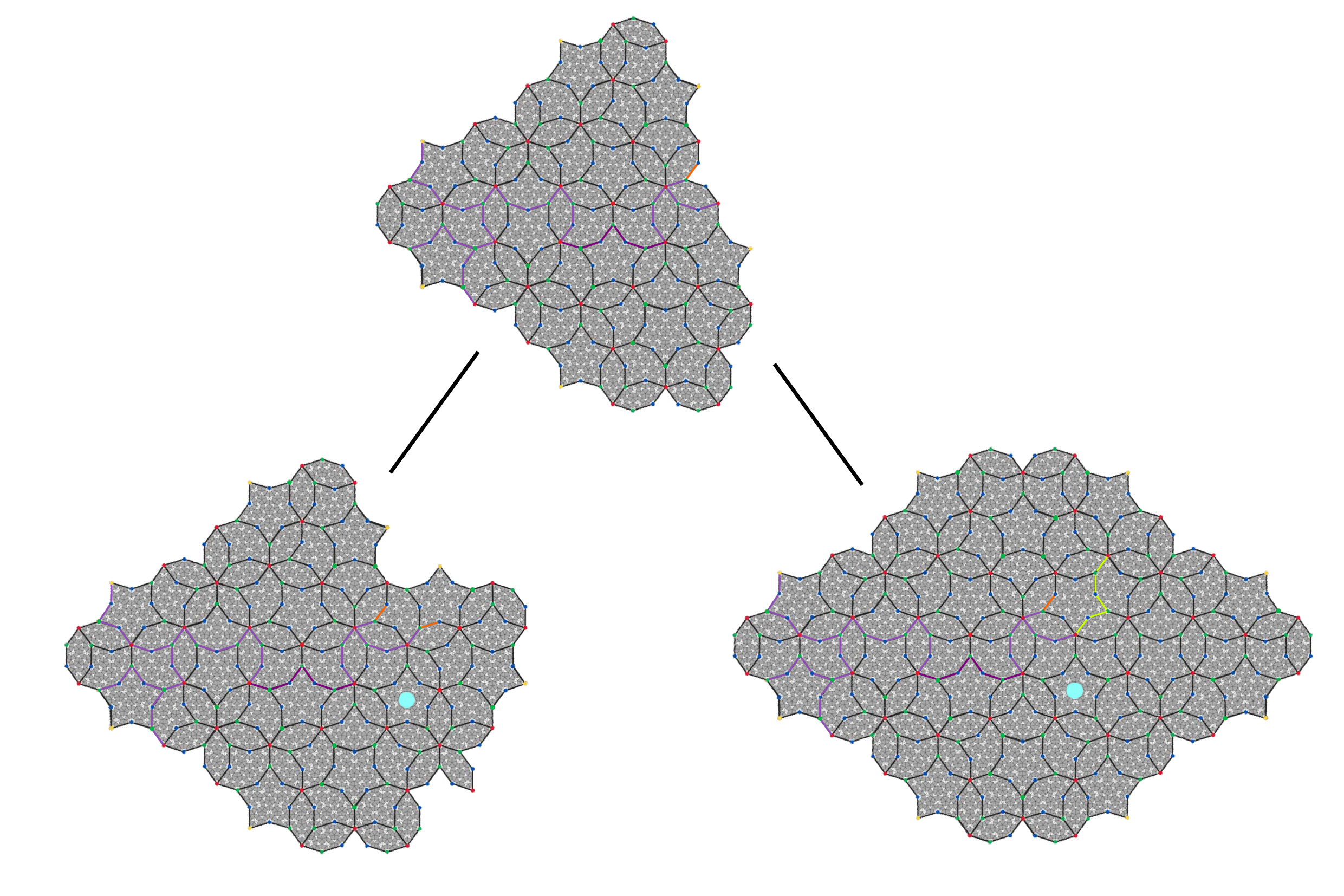}
\caption{Attempts to extend a cape 4 into a fully leafed bi-infinite caterpillar when the cape 4 is built on a type 1 star. In bright yellow-green, we indicate a sea caterpillar previously shown to be non-extendable into a fully leafed bi-infinite caterpillar.}
\label{fig:cap 4 cas pas prolongeable}
\end{figure}
\newpage
\section*{Appendix C: Proof of Proposition \ref{d}}
We consider only the sea grafting (k). By reflection, the same conclusion holds for sea grafting (i). All cases are covered using the same method as in the previous subsection.

\begin{figure}[H]
\centering
\includegraphics[width=0.9\textwidth, trim=0cm 0cm 0cm 0cm, clip]{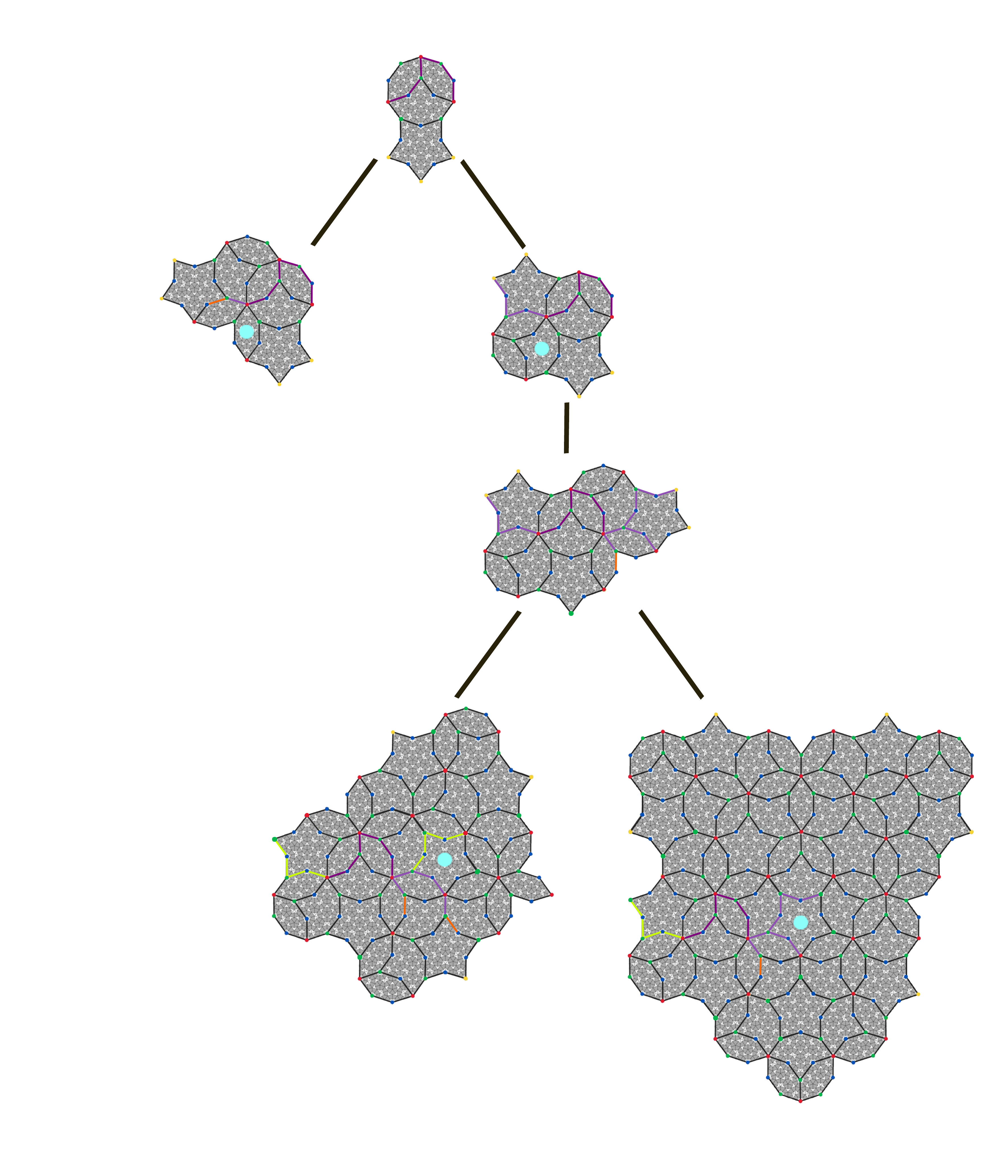 }
\caption{Study of possible extensions of sea grafting (k) (case 1)}
\end{figure}

\begin{figure}[H]
\centering
\includegraphics[width=1\textwidth, trim=0cm 19cm 0cm 0cm, clip]{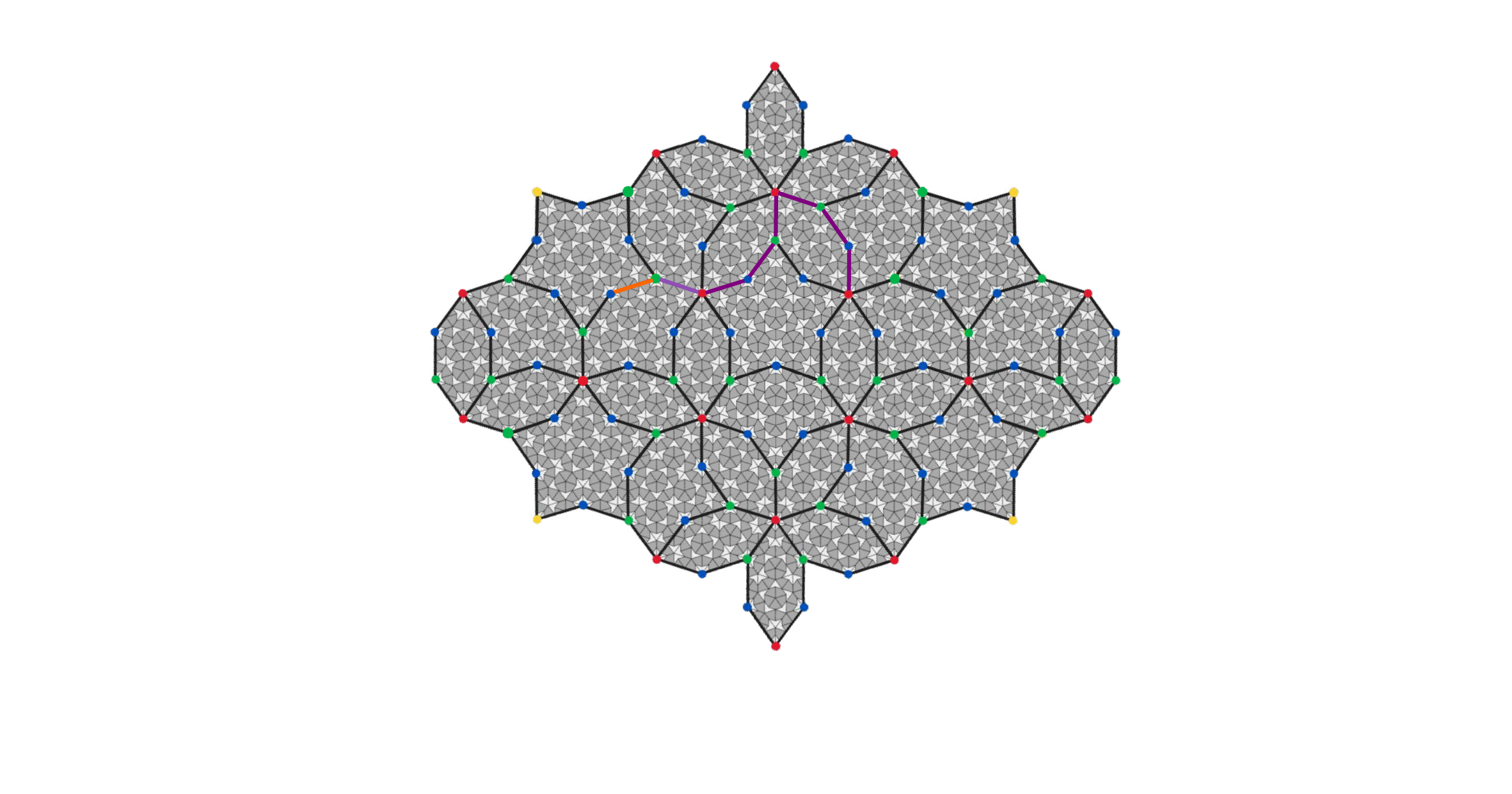 }
\caption{Study of possible extensions of sea grafting (k) (case 2)}
\end{figure}

The proposition is thus verified for sea grafting (k).

\newpage

\section*{Appendix D: Proof of Proposition \ref{e}}
The following figure shows the elimination of sea grafting (j) as a possible subcaterpillar of a bi-infinite fully leafed caterpillar.

\begin{figure}[H]
\centering
\includegraphics[width=0.7\textwidth, trim=0cm 5cm 0cm 0cm, clip]{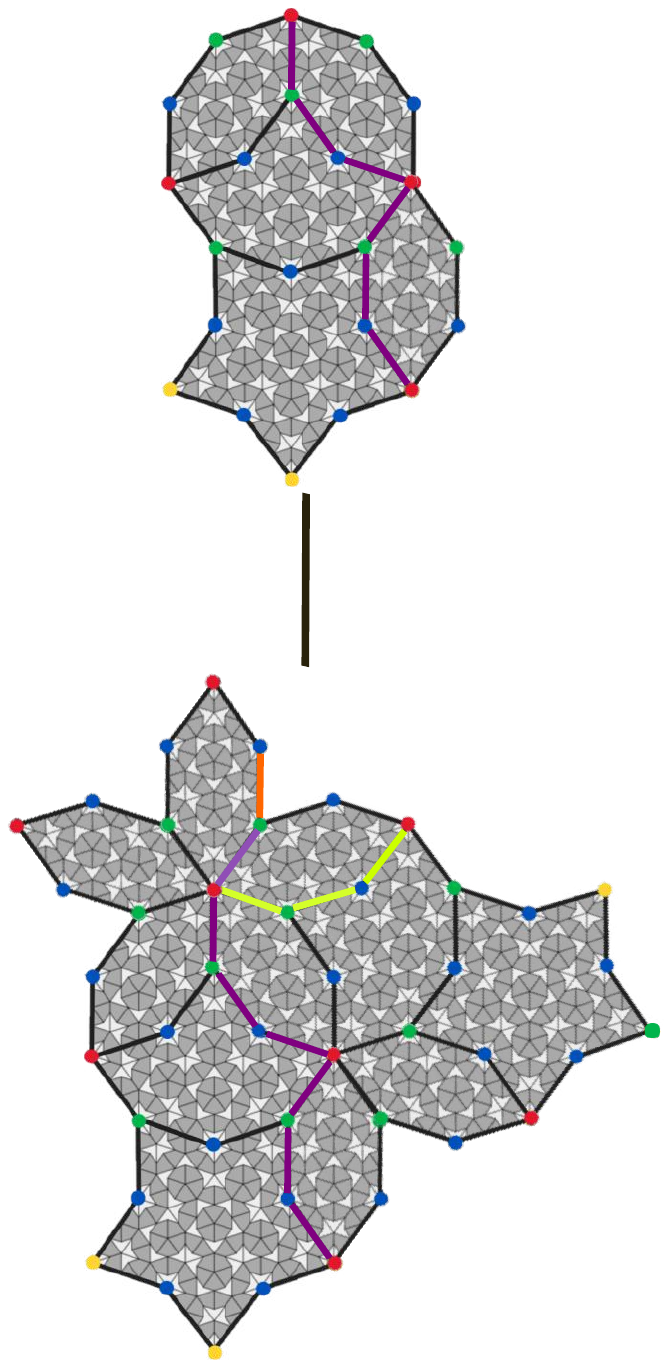}
\caption{Study of sea grafting (j). We kept sea grafting (j) in dark purple on the lower patch. The light-yellow-green chain is then incomplete: the forbidden chain actually corresponds to the sea grafting of the previous proposition.}
\end{figure}

By reflection, we can also conclude that sea grafting (l) does not belong to any bi-infinite fully leafed caterpillar.

\newpage
\section*{Appendix E: Proof of Proposition \ref{f}}
The following figure completes the proof.
\begin{figure}[H]
\centering
\includegraphics[width=1\textwidth, trim=0cm 13cm 0cm 1cm, clip]{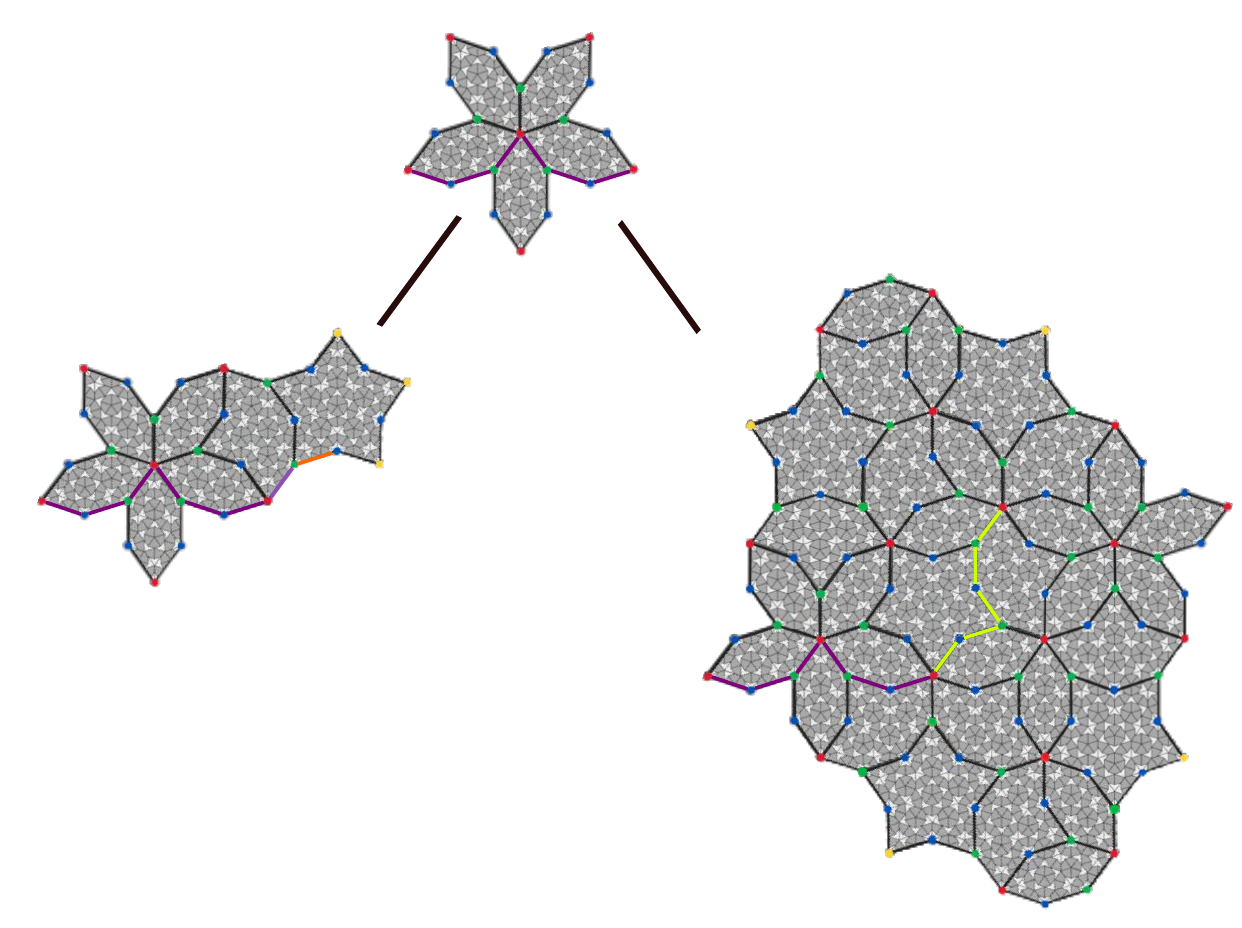}
\caption{Study of sea grafting (h)}
\end{figure}

\newpage
\bibliographystyle{alpha}
\bibliography{Bibliographie}

\end{document}